\documentclass[a4paper]{amsart}

\usepackage[all]{xy}
\usepackage{amsmath}
\usepackage{amssymb}
\usepackage{latexsym}
\usepackage{enumerate}
\usepackage{prettyref}
\usepackage{hyperref}
\usepackage{bm}
\usepackage{mathrsfs}
\usepackage{color}

\newtheorem{theorem}{Theorem}[section]
\newtheorem{definition}[theorem]{Definition}
\newtheorem{lemma}[theorem]{Lemma}
\newtheorem{proposition}[theorem]{Proposition}
\newtheorem{corollary}[theorem]{Corollary}
\newtheorem{remark}[theorem]{Remark}

\newtheorem{question}[theorem]{Question}
\newtheorem{examplecore}[theorem]{Example}

\newenvironment{example}{\begin{examplecore}}{\hspace*{\fill}
$\square$\par\vspace{.1cm}\end{examplecore}}

\newcommand{\matthias}[1]{#1}

\newcommand{\op}{\operatorname}
\newcommand{\et}{\mathrm{\acute{e}t}}

\begin{document}

\title[Chow--Witt rings of $\op{B}\op{Sp}_{2n}$ and $\op{B}\op{SL}_{n}$]{Chow--Witt rings of classifying spaces for symplectic and special linear groups}  

\author{Jens Hornbostel and Matthias Wendt}

\address{Jens Hornbostel\\Fachgruppe Mathematik/Informatik\\Bergische Universit\"at Wuppertal\\ Gau\ss{}stra\ss{}e 20\\42119 Wuppertal\\Germany}
\email{hornbostel@math.uni-wuppertal.de}
\address{Matthias Wendt\\ Institut Mathematik\\Universit\"at Osnabr\"uck\\Albrechtstra\ss{}e 28a\\49176 Osnabr\"uck\\Germany}
\email{m.wendt.c@gmail.com}

\date{January 2019}


\maketitle

\begin{abstract}
We compute the Chow--Witt rings of the classifying spaces for the symplectic and special linear groups. In the structural description we give, contributions from real and complex realization are clearly visible. In particular, the computation of cohomology with $\mathbf{I}^j$-coefficients is done closely along the lines of Brown's computation of integral cohomology for special orthogonal groups. The computations for the symplectic groups show that Chow--Witt groups are a symplectically oriented ring cohomology theory. Using our computations for special linear groups, we also discuss the question when an oriented vector bundle of odd rank splits off a trivial summand.
\end{abstract}

\setcounter{tocdepth}{1}
\tableofcontents

\section{Introduction}

Recently we wished to know the Chow--Witt rings $\widetilde{\op{CH}}^\bullet({\op{B}}\op{Sp}_{2n})$ and $\widetilde{\op{CH}}^\bullet({\op{B}}\op{SL}_n)$ of the classifying spaces for symplectic and special linear groups in terms of characteristic classes. There are many reasons for being interested in these. For example, $\widetilde{\op{CH}}^\bullet({\op{B}}\op{SL}_n)$ is essentially by definition the natural home for the universal {\em Euler class}, and we will show that $\widetilde{\op{CH}}^\bullet({\op{B}}\op{Sp}_{2n})$ is the one
for {\em Pontryagin classes}. Detailed knowledge of the characteristic classes for vector bundles in Chow--Witt theory is indispensable  in the obstruction-theoretic vector bundle classification as initiated by Morel \cite{MField} and recently explored in a series of papers by Asok and Fasel, cf. \cite{AsokFasel,AsokFaselSpheres,AsokFaselSecondary} and more. Our results confirm the role of $\widetilde{\op{CH}}$ as the universal intersection theory which is symplectically orientable in the sense of Panin and Walter,
and they refine Totaro's computations for Chow groups \cite{totaro:bg}. (To complete both the picture and the terminology, recall that Totaro's computation confirms that ${\op{CH}}^\bullet({\op{B}}\op{GL}_n)$
is the natural home for Chern classes in intersection theory.)  Looking at
Chow--Witt groups of projective spaces and their relation to the real realization in \cite{fasel:ij}, there is also the obvious question how closely Chow--Witt groups (or more precisely the associated $\mathbf{I}^j$-cohomology) of other classifying spaces are tied to the real realization. This list of motivating questions could be extended, but we stop here, and explain some of the answers we have to offer.

The following result provides a formula for the Chow--Witt rings of the classifying spaces of the symplectic groups. Results of a similar nature for Witt groups are established in \cite{PaninWalter}. Our techniques differ from theirs in some respects, see Section \ref{sec:sympor} for further details.

\begin{theorem}
\label{thm:spn}
Let $F$ be a perfect field of characteristic $\neq 2$. For any $n$, there is an isomorphism of graded rings
\[
\widetilde{\op{CH}}^\bullet({\op{B}}\op{Sp}_{2n})\cong \op{GW}(F)[\op{p}_1,\dots,\op{p}_n]. 
\]
Here, the $\op{p}_i$ are Pontryagin classes for symplectic bundles, with $deg(\op{p}_i)=2i$. These are uniquely determined by their compatibility with stabilization in the symplectic series and the requirement that the top Pontryagin class equals the Thom class of the universal bundle over ${\op{B}}\op{Sp}_{2n}$.
\end{theorem}

We also obtain a proof of a splitting principle which expresses the Pontryagin classes as elementary symmetric polynomials in the first Pontryagin classes for symplectic line bundles, and realizes $\widetilde{\op{CH}}^\bullet({\op{B}}\op{Sp}_{2n})$ as ring of symmetric polynomials in $\widetilde{\op{CH}}^\bullet({\op{B}}\op{SL}_{2}^{\times n})\cong\op{GW}(F)[\op{p}_{1,1},\dots,\op{p}_{1,n}]$. For details, see Proposition \ref{prop:symsplit} and  Theorem~\ref{thm:symplectic}.

After this preprint was essentially finished, we learned that related aspects of the quaternionic projective bundle formula have been discussed independently in the Ph.D. thesis of Yang \cite{yang}, with a different perspective. Also, Ivan Panin informed us that the above classes $\op{p}_i$ should rather be called \emph{Borel classes}, and this terminology will be employed in the revised version of \cite{PaninWalter}.

The above theorem provides us with a well-behaved theory of Pontryagin classes for symplectic bundles in Chow--Witt groups which was claimed by \cite[Section 7]{PaninWalter} without proof. It shows that Pontryagin classes for 
$\op{Sp}_{2n}$ play the same fundamental role for Chow--Witt theory as Chern classes for $\op{GL}_n$ play for Chow theory. This has the following consequence:

\begin{corollary}
Chow--Witt groups $\widetilde{\op{CH}}^{\bullet}(-)$ form a symplectically oriented ring cohomology theory.
\end{corollary}

For the precise formulation of symplectically oriented theories, one may either work in a bigraded setting, i.e., representable theories in the Morel--Voevodsky $\mathbb{A}^1$-homotopy category, or with single-graded theories as in \cite{PaninWalter}. For the latter, however, some care has to be taken as to what the exact form of the localization sequence is. See Section~\ref{sec:sympor} for a more detailed discussion. 

Our main result is the computation of  $\widetilde{\op{CH}}^\bullet({\op{B}}\op{SL}_{n})$, which is stated in Section \ref{sec:statethm}. The proof is much lengthier and more involved than the one for ${\op{B}}\op{Sp}_{2n}$, and is finally achieved in Section \ref{sec:special}. The following is a combination of Theorems~\ref{thm:slnin} and \ref{thm:slnchw}. The description of the Chow--Witt ring as a cartesian square is actually a more general observation for schemes with 2-torsion-free Chow rings which is of independent interest.

\begin{theorem}
\label{thm:sln}
Let $F$ be a perfect field of characteristic $\neq 2$. 
Then the following square, induced from the cartesian square description of the Milnor--Witt K-theory sheaf, is cartesian:
\[
\xymatrix{
\widetilde{\op{CH}}^\bullet({\op{B}}\op{SL}_n)\ar[r] \ar[d] & \ker\partial \ar[d]^{\bmod 2} \\
\op{H}^{\bullet}_{\op{Nis}}({\op{B}}\op{SL}_n,\mathbf{I}^\bullet)\ar[r]_\rho &\op{Ch}^\bullet({\op{B}}\op{SL}_n).
}
\]
Here $\op{Ch}:=\op{CH}/2$,  $\partial\colon \op{CH}^\bullet({\op{B}}\op{SL}_n)\to \op{Ch}^\bullet({\op{B}}\op{SL}_n)\xrightarrow{\beta} \op{H}^{\bullet+1}({\op{B}}\op{SL}_n,\mathbf{I}^{\bullet+1})$ is an integral Bockstein operation, and we have 
\[
\matthias{\ker\partial=\mathbb{Z}[\op{c}_{2i+1}, 2{\op{c}}_{2i_1}\cdots{\op{c}}_{2i_k},{\op{c}}_{2i}^2, \op{c}_n]}\subseteq\mathbb{Z}[\op{c}_2,\dots,\op{c}_n]\cong \op{CH}^\bullet({\op{B}}\op{SL}_n).
\]
The cohomology ring $\op{H}^{\bullet}_{\op{Nis}}({\op{B}}\op{SL}_n,\mathbf{I}^\bullet)$ can be explicitly identified as the graded-commutative  $\op{W}(F)$-algebra 
\[
\op{W}(F)\left[\op{p}_2,\op{p}_4,\dots,\op{p}_{[(n-1)/2]},\op{e}_n, \{\beta_J\}_J\right]
\]
where $\op{e}_n$ is the Euler class, the index set $J$ runs through the sets $\{j_1,\dots,j_l\}$ of natural numbers with $0<j_1<\dots<j_l\leq[(n-1)/2]$ and $\beta_J=\beta(\overline{\op{c}}_{2j_1}\cdots \overline{\op{c}}_{2j_l})$, modulo the following relations
\begin{enumerate}
\item $\op{I}(F) \beta_J=0$.
\item If $n=2k+1$ is odd, $\op{e}_{2k+1}=\beta_{\{k\}}$. 
\item For two index sets $J$ and $J'$, we have 
\[
\beta_J\cdot \beta_{J'}=\sum_{k\in J} \beta_{\{k\}}\cdot \op{p}_{(J\setminus\{k\})\cap J'}\cdot \beta_{\Delta(J\setminus\{k\},J')}
\]
where we set $\op{p}_A=\prod_{i=1}^l \op{p}_{2a_i}$ for an index set $A=\{a_1,\dots,a_l\}$.
\end{enumerate}
\end{theorem}

Note that in both the symplectic and special linear case, we only have to consider the trivial twist for $\widetilde{\op{CH}}$ because the corresponding classifying spaces have trivial Picard groups. 

The above theorems provide a good overview of the Chow--Witt-theoretic characteristic classes and their relations. Pontryagin classes already appear in the work
of Ananyevskiy \cite{ananyevskiy}. The general Bockstein classes $\beta_J$ seem not to have been considered before, although integral Stiefel--Whitney
classes (corresponding to sets $J$ with one element) appear in \cite{fasel:ij}.

It has been realized for quite some time that the sheaves $\mathbf{I}^n$ are strongly related to real realization, see the work of Colliot-Th{\'e}l{\`e}ne, Scheiderer, Sujatha, and also Morel \cite{morel:puissances}. In \cite[p. 417/418]{fasel:ij}, Fasel gave a more precise incarnation of this philosophy, proposing an explicit relation between $\op{H}^\bullet(X,\mathbf{I}^\bullet)$ and the singular cohomology $\op{H}^\bullet(X(\mathbb{R}),\mathbb{Z})$ of the real realization for certain "nice" $X$. A much more precise version of this philosophy, comparing $\mathbf{I}^j$-cohomology over $\op{Spec}(\mathbb{R})$ to the integral cohomology of the real realization, has been very recently given by Jacobson in \cite{jacobson}. Our results completely agree with this philosophy. A detailed account of the compatibility of these identifications with push-forwards, pullbacks, localization sequences and intersection products will be given in \cite{4real}; this implies that our algebraic characteristic classes map to their topological counterparts and provides a further strengthening of the link between $\mathbf{I}^\bullet$-cohomology and real realization. 

\medskip

Using Theorem \ref{thm:sln}, we deduce the following interesting splitting condition for odd rank bundles in terms of stable invariants, cf. Proposition~\ref{greatapplication}: 

\begin{corollary}
\label{cor:enodd}
Let $F$ be a perfect field of characteristic $\neq 2$, and let $X=\op{Spec}A$ be a smooth affine scheme over $F$ of odd dimension $n=2k+1$. Assume that ${}_2\op{CH}^n(X)=0$, i.e., the 2-torsion in $\op{CH}^n(X)$ is trivial. Then an oriented vector bundle $\mathscr{E}$ over $X$ of rank $n$ splits off a free rank one summand if and only if ${\op{c}_{n}}(\mathscr{E})=0$ and $\beta({\overline{\op{c}}_{n-1}})(\mathscr{E})=0$.
\end{corollary}

While the vanishing of the Euler class and the corresponding splitting is well-known for stably free modules, the above result seems to be new. 

For the sake of perspective, recall that in \cite[Question 7.12]{bhatwadekar:sridharan}, Bhatwadekar and Sridharan asked if vanishing of the top Chern class was the only obstruction to splitting off  a free rank 1 summand in the above odd rank situation. The above result contributes to understanding the difference between the Euler class and the top Chern class for odd rank bundles (at least in the absence of 2-torsion) -- it is given by the top integral Stiefel--Whitney class $\beta({\overline{\op{c}}_{n-1}})$. Moreover, there is a compatibility requiring that the images of $\beta({\overline{\op{c}}_{n-1}})$ and $\op{c}_{n}$ agree in $\op{Ch}^n(X)$. From this perspective, it seems likely that the question of Bhatwadekar and Sridharan actually has a positive answer. For the moment, however, we are able to prove that vanishing of $\op{Sq}^2(\overline{\op{c}}_{n-1})=\overline{c}_n$ implies vanishing of $\beta(\overline{\op{c}}_{n-1})$ only in a few cases, see e.g. Proposition~\ref{prop:greatapplication2}. In general, our results underline the importance of \matthias{$\overline{\op{c}}_{n-1}$} and Voevodsky's motivic cohomology operation $\op{Sq}^2$ for this question.
 
A similar analysis, combining techniques of this article for $\op{SL}_n$ and \cite{fasel:ij} to a motivic generalization of \cite{cadek}, provides a full computation of the total Chow--Witt rings of ${\op{B}}\op{GL}_n$ as well as the finite Grassmannians $\op{Gr}(k,n)$, for details cf.~\cite{real-grassmannian}.

Finally, using techniques from stable ${\mathbb{A}^1}$-homotopy theory,
it is possible to identify $\widetilde{\op{CH}^{\bullet}}(-)$ as part
of a bigraded representable cohomology theory and reformulate some of the arguments in terms of the Morel--Voevodsky $\mathbb{A}^1$-homotopy theory. We will not use this, but see Remarks~\ref{A1rep}, \ref{rem:classmap} and 
Subsection~\ref{sec:sympor}. 

\emph{Structure of the paper:} In Section~\ref{sec:prelims}, we provide a recollection on basic facts from Chow--Witt theory which will be used in the computations. Then Section~\ref{sec:cwbg} recalls the Totaro-style definition of Chow--Witt groups of classifying spaces. The Chow--Witt groups of classifying spaces of symplectic groups are computed in Section~\ref{sec:symplectic}. The rest of the paper is concerned with the computations for the classifying spaces of special linear groups: first, we define all the relevant characteristic classes in Section~\ref{sec:character} and formulate the main structural results for $\widetilde{\op{CH}}^\bullet({\op{B}}\op{SL}_n)$ in Section~\ref{sec:statethm}. In Section~\ref{sec:relations}, we establish basic relations between Chow--Witt characteristic classes, and in Section~\ref{sec:special} we finally provide the inductive proof of the main theorem describing the Chow--Witt ring of the classifying space of the special linear groups. A short discussion of applications to splitting questions and Euler classes is provided in Section~\ref{sec:bs}.

\emph{Conventions:} Throughout the article, we consider a perfect base field $F$ with $\op{char}(F) \neq 2$. Schemes are separated $F$-schemes of finite type. All the cohomology groups we consider will be Nisnevich cohomology groups, i.e., Ext-groups between Nisnevich sheaves on the small site of a smooth scheme. Nisnevich sheaves are considered both on the big site $\op{Sm}/F$ of smooth $F$-schemes as well as on the small Nisnevich site of particular smooth $F$-schemes $X$. Strictly $\mathbb{A}^1$-invariant sheaves will be denoted by bold-face letters, such as in $\mathbf{K}^{\op{MW}}_n$ or $\mathbf{I}^n$. For this reason, we occasionally drop the index ``Nis'' for the sake of less cumbersome notation. For an abelian group $A$, the subgroup of 2-torsion elements is denoted by ${}_2A$. 

\emph{Acknowledgements:} The project was initiated during a pleasant stay of the second author at Wuppertal financed by the DFG SPP 1786 ``Homotopy theory and algebraic geometry''; and it was finished during a pleasant stay of both authors at the Institut Mittag-Leffler during the special program ``Algebro-geometric and homotopical methods''. We are grateful to Jean Fasel for email exchanges concerning $\mathbb{A}^1$-representability of Milnor--Witt K-cohomology and Totaro's analysis of the Steenrod square,  to Alexey Ananyevskiy and Aravind Asok for comments on an earlier version of the preprint, and to Marco Schlichting for pointing out a problem in a previous version of Section \ref{sec:bs}. We are very grateful for the detailed and insightful comments provided by two anonymous referees. 

\section{Recollection on Chow--Witt rings}
\label{sec:prelims}

In this section, we provide a recollection of definitions and basic properties of Chow--Witt rings. Most of these can be found in \cite{fasel:chowwitt} and \cite{fasel:memoir}, and we assume some familiarity with these works.

\subsection{Functorialities and localization sequence}\label{funcandloc}
Recall the following fundamental cartesian square of strictly $\mathbb{A}^1$-invariant Nisnevich sheaves 
\[
\xymatrix{
\mathbf{K}^{\op{MW}}_n\ar[r] \ar[d] & \mathbf{K}^{\op{M}}_n \ar[d] \\
\mathbf{I}^n \ar[r] & \mathbf{K}^{\op{M}}_n/2
}
\]
established by Morel \cite{morel:puissances}. More precisely, in loc.cit., Morel showed that the fiber product $\mathbf{J}^\bullet:=\mathbf{I}^\bullet\times_{\mathbf{K}^{\op{M}}_\bullet/2}\mathbf{K}^{\op{M}}_\bullet$ is naturally isomorphic to Milnor--Witt K-theory $\mathbf{K}^{\op{MW}}_\bullet$, described in terms of explicit generators ($[u]$ for $u\in F^\times$ in degree $1$ and $\eta$ in degree $-1$) and relations. The classical definition of the Chow--Witt groups $\widetilde{\op{CH}}^\bullet(X)$ for a smooth scheme $X$, see e.g. \cite{fasel:memoir}, is as the cohomology of the Gersten complex $C(X,\mathbf{J}^{\bullet})$,  which is a fiber product of the Gersten complex for Milnor K-theory and the Gersten complex for the sheaves of powers $\mathbf{I}^{\bullet}$ of the fundamental ideal in the Witt ring, fibered over the Gersten complex for $\mathbf{K}^{\op{M}}_{\bullet}/2 \cong \mathbf{I}^{\bullet}/\mathbf{I}^{\bullet +1}$ (for Chow--Witt groups, only the easy part of the Milnor conjecture in degrees $0$ and $1$ is used here). It will be useful to consider the above square as one of graded algebras, extending $\mathbf{I}^{\bullet}$ to negative degrees by the Witt sheaf $\mathbf{W}$ in every degree. According to \cite{morel:puissances} we might denote this graded algebra by $\mathbf{K}_{\bullet}^{\op{W}}$ rather than $\mathbf{I}^\bullet$, and the inclusion $\mathbf{I}^{n+1} \subset \mathbf{I}^n$ then becomes multiplication with $\eta$ in the graded algebra. 
Alternatively, one might also define  Chow--Witt groups as sheaf cohomology $\op{H}^n_{\op{Nis}}(X,\mathbf{K}^{\op{MW}}_n)$ since by  \cite[Proposition 2.3.1, Theorem 2.3.4]{AsokFaselEuler} these two definitions are equivalent. Note also that the work of Morel, in particular \cite[Corollary 5.43 and Theorem 5.47]{MField}, implies that the Gersten complexes for $\mathbf{K}^{\op{MW}}_\bullet$ compute both the Zariski and Nisnevich cohomology of the Milnor--Witt K-sheaves. If $X$ is a smooth $F$-scheme and $\mathscr{L}$ is a line bundle over $X$, there is a similar square of Nisnevich sheaves on the small site of $X$ where $\mathbf{K}^{\op{MW}}_n$ and $\mathbf{I}^n$ are twisted by $\mathscr{L}$, for details cf. \cite[Section 2.3]{AsokFaselEuler}. 
However, for us this will be mostly irrelevant since all varieties we will consider have trivial Picard groups.

By \cite{fasel:memoir} we have several standard properties relevant for our computations. First of all, as mentioned above, the Chow--Witt groups can be computed in terms of Gersten complexes $C(X,\mathbf{J}^\bullet)$ (sometimes also called Rost--Schmid complexes). This implies, in particular, that Chow--Witt groups can be identified with both Zariski and Nisnevich cohomology of $\mathbf{J}^\bullet$.

The Chow--Witt groups have pullbacks for flat maps between regular schemes \cite[Corollaire 10.4.2]{fasel:memoir} and the maps induced by change of coefficients in the fundamental square are compatible with flat pullbacks (Corollaire 10.4.3 in loc.cit.). Moreover, as discussed in \cite[Section 2]{AsokFaselEuler}, pullbacks for Chow--Witt groups exist for arbitrary morphisms $f\colon X\to Y$ of smooth schemes 
\[
f^\ast\colon \widetilde{\op{CH}}^i(Y,\mathscr{L})\to\widetilde{\op{CH}}^i(X,f^\ast\mathscr{L}).
\]
These pullbacks are compatible with composition for smooth morphisms and the Chow--Witt theoretic pullbacks of \cite{fasel:memoir} coincide with the pullbacks defined via Gersten complexes \cite[Section 2.2]{AsokFaselEuler}, cf. \cite[Theorem 2.3.4]{AsokFaselEuler}. 

There are also push-forward maps for proper morphisms between smooth schemes \cite[Corollaire 10.4.5]{fasel:memoir}, again compatible with the change of coefficients. However, some care is needed: for a proper morphism $f\colon X\to Y$, a push-forward 
\[
f_\ast:\widetilde{\op{CH}}^n(X)\to \widetilde{\op{CH}}^{n+\dim(Y)-\dim(X)}(Y)
\]
exists only if the relative canonical bundle $\omega_{X/Y}$ admits an orientation in the sense of \cite[Definition 4.3]{MField} or \cite[Section 2.2]{calmes:fasel}, and the push-forward depends on the choice of such. Alternatively, push-forwards exist in the form 
\[
\widetilde{\op{CH}}^n(X,\omega_{X/Y}\otimes f^\ast\mathscr{L})\to \widetilde{\op{CH}}^{n+\dim(Y)-\dim(X)}(Y,\mathscr{L}),
\]
cf. \cite[Section 3]{calmes:fasel}. There is a general base change result, cf. \cite[Theorem 2.4.1]{AsokFaselEuler}.

Using the identification of Zariski and Nisnevich cohomology, \cite[Corollaire 10.4.10]{fasel:memoir} implies that for a closed immersion $i\colon Z\subseteq X$ of pure codimension $q$ between regular schemes with open complement $j\colon U=X\setminus Z\hookrightarrow X$, there is a localization sequence 
\begin{eqnarray*}
\cdots\to\op{H}^i_{\op{Nis}}(U,\mathbf{K}^{\op{MW}}_j,\mathscr{L}) &\xrightarrow{\partial}& \op{H}^{i-q+1}_{\op{Nis}}(Z,\mathbf{K}^{\op{MW}}_{j-q},\mathscr{L}\otimes\mathscr{N}) \xrightarrow{i_\ast} \\\to
\op{H}^{i+1}_{\op{Nis}}(X,\mathbf{K}^{\op{MW}}_j,\mathscr{L})&\xrightarrow{j^\ast} 
&\op{H}^{i+1}_{\op{Nis}}(U,\mathbf{K}^{\op{MW}}_j,\mathscr{L})\xrightarrow{\partial}\cdots
\end{eqnarray*}
where $\mathscr{N}=\overline{\kappa}^\ast\op{Ext}^q_{\mathscr{O}_X}(\kappa_\ast\mathscr{O}_Z,\mathscr{O}_X)$. Again, in all our applications all dualities/twists will be trivial, and moreover the trivializations of the relevant relative canonical bundles will be canonical. A consequence of the localization sequence is that an open immersion $U\to X$ induces isomorphisms on Chow--Witt groups in degrees below $\op{codim}(Z\hookrightarrow X)-1$, cf. \cite[Lemma 2.4.2]{AsokFaselEuler}.

The d\'evissage isomorphism allows to replace in the above sequence 
\[
\op{H}^{i-q+1}_{\op{Nis}}(Z,\mathbf{K}^{\op{MW}}_{j-q},\mathscr{L}\otimes\mathscr{N}) \cong \op{H}^{i+1}_{\op{Nis},Z}(X,\mathbf{K}^{\op{MW}}_{j},\mathscr{L}),
\]
and the map $i_\ast$ can be written as composition of the d\'evissage isomorphism and the extension of support, cf. \cite[Remarque 9.3.5 and Remarque 10.4.8]{fasel:memoir}.

\begin{remark}\label{A1rep}
As we first learned from Jean Fasel, the groups $\op{H}^\bullet(-,\mathbf{K}^{\op{MW}}_\bullet)$ are representable in the stable motivic homotopy category. For this, consider the $\mathbb{P}^1$-spectrum given by $E_i=\op{K}(\mathbf{K}^{\op{MW}}_i,i)$. The identifications 
\[
\Omega_{\mathbb{P}^1}\op{K}(\mathbf{K}^{\op{MW}}_i,i)\simeq \op{K}(\mathbf{K}^{\op{MW}}_{i-1},i-1)
\]
follow from \cite[Theorem 6.13]{MField} and \cite[Proposition 2.9]{AsokFaselSpheres}. For these spaces, we have $[X,\op{K}(\mathbf{K}^{\op{MW}}_i,i)]\cong\widetilde{\op{CH}}^i(X)$. Details of this representability and its consequences have now appeared in \cite{deglise:fasel}, related issues are also discussed in \cite{levine}. Using this, the localization sequence can also be obtained from the homotopy purity sequence of \cite{MV}, but we won't use this approach. In a similar vein, we use explicit varieties rather than motivic spaces for classifying spaces, see e.g. Remark \ref{rem:classmap} and Section \ref{geomsetup}.
\end{remark}

 \begin{remark}
One of the fundamental differences between the topological arguments as e.g. in \cite{brown} and their algebraic counterparts is that there is no long localization sequence (only a short piece of the long localization sequence is in the ``geometric bidegrees''). This problem already appears for Chow groups. Some of the material in \cite[Section 2]{RojasVistoli} and the splitting principle arguments in \cite{RojasVistoli} are related to this. The same phenomena will also appear in our computations.
\end{remark}

\begin{lemma}
\label{lem:module}
Let $\mathbf{A}$ be any strictly $\mathbb{A}^1$-invariant Nisnevich sheaf. Then for any smooth $F$-scheme $X$, the cohomology groups $\op{H}^\bullet(X,\mathbf{A})$ are natural modules over $\op{H}^0(\op{Spec}F,\mathbf{A})=\mathbf{A}(F)$ via the pullback along the structure map $X\to\op{Spec}F$. The natural push-forward, pullback and boundary maps in cohomology are all $\mathbf{A}(F)$-module maps. For example, looking at $\widetilde{\op{CH}}$, all maps are $\op{GW}(F)$-linear.
\end{lemma}

One of the main applications of the above localization sequence is the definition of a Chow--Witt-theoretic Euler class, cf. \cite{AsokFaselEuler}. If $X$ is smooth over $F$ and $\gamma\colon E\to X$ is a vector bundle of rank $r$ with zero section $s_0\colon X\to E$,
 the {\em Euler class} $\op{e}_r$ is defined as
\[
\op{e}_r(\gamma):=(\gamma^\ast)^{-1}(s_0)_\ast\langle 1\rangle \in \widetilde{\op{CH}}^r(X,\det(\gamma)^\vee),
\]
cf. \cite[Definition 13.2.1]{fasel:memoir}, see also the discussion in \cite[Section 3]{AsokFaselEuler}. Recall that $\langle 1\rangle = 1$ in $\op{GW}(F)$ and more generally in $\widetilde{\op{CH}}^0(X)$. 
By \cite[Theorem 1]{AsokFaselEuler}, this Chow--Witt-theoretic Euler class agrees up to multiplication by a unit in $\op{GW}(F)$ with the definition of \cite{MField} of the Euler class as the obstruction to splitting off a free rank one summand from $E$. In the localization sequence associated to $E$ and the closed subscheme $s_0(X)$, the Euler class corresponds under the d\'evissage isomorphism to the Thom class of $E$. In this situation, the composition
\[
\widetilde{\op{CH}}^{q-r}(X) \cong \widetilde{\op{CH}}^q_{s_0(X)}(E)\to \widetilde{\op{CH}}^q(E)\cong \widetilde{\op{CH}}^q(X)
\]
is often called ``multiplication with the Euler class'', and the first isomorphism "multiplication with the Thom class"
or "Thom isomorphism". Finally, the Euler class is compatible with pullbacks of morphisms between smooth schemes, cf.~\cite[Proposition 3.1.1]{AsokFaselEuler}.

\subsection{Oriented intersection products}

We provide some recollection on the definition and properties of the oriented intersection products giving the graded ring structure on the Chow--Witt groups. Most of the material can be found in \cite{fasel:chowwitt}. 

The first step in defining an oriented intersection product is the exterior product, cf. \cite[Theorem 4.16 and Corollary 4.17]{fasel:chowwitt} resp. \cite[Section 3]{calmes:fasel}:
\[
\widetilde{\op{CH}}^i_Z(X,\mathscr{L})\times\widetilde{\op{CH}}^j_W(Y,\mathscr{N})\to \widetilde{\op{CH}}^{i+j}_{Z\times W}(X\times Y,\op{pr}_1^\ast\mathscr{L}\otimes\op{pr}_2^\ast\mathscr{N}). 
\]
In the fiber product of Gersten complexes, this comes from the exterior product for Chow-groups with Milnor K-theory coefficients \cite{rost} and the exterior product on Witt groups. The exterior product is associative \cite[Proposition 4.18]{fasel:chowwitt}.
The exterior products commute with proper push-forwards, cf. \cite{calmes:fasel}, but only up to the following signs as Jean Fasel explained to us. Fix a smooth scheme $Y$ and a proper morphism  $p\colon X\to Z$ of codimension $n$. For classes $\sigma\in\widetilde{\op{CH}}^i(X,\omega_{X/Z}\otimes p^\ast \mathscr{L})$ and $\tau\in\widetilde{\op{CH}}^j(Y,\mathscr{N})$,  we have the following equalities in the group $\widetilde{\op{CH}}^{i+j}(Z\times Y,\op{pr}_1^\ast\mathscr{L}\otimes\op{pr}_2^\ast\mathscr{N})$:
\begin{eqnarray*}
(p\times \op{id}_Y)_\ast (\sigma\times \tau)&=&p_\ast(\sigma)\times\tau\\
(\op{id}_Y\times p)_\ast(\tau\times \sigma)&=&\langle(-1)^{in}\rangle \tau\times p_\ast(\sigma)
\end{eqnarray*}
Here, we can get rid of twisted coefficients if $\omega_{X/Z}$ is orientable (but there will be a dependence on the choice of orientation). Combining these statements, for two proper morphisms $p\colon X\to X'$, $q\colon Y\to Y'$ with $n=\op{codim}(q)$ and two oriented cycles $\sigma\in\widetilde{\op{CH}}^i(X,\omega_{X/X'})$, $\tau\in\widetilde{\op{CH}}^j(Y,\omega_{Y/Y'})$ we have 
\[
(p\times q)_\ast(\sigma\times \tau)=\langle(-1)^{jn}\rangle p_\ast(\sigma)\times q_\ast(\tau).
\]

The second step in the definition of intersection product is given by a pullback along the diagonal inclusion $\Delta\colon X\to X\times X$. In \cite{fasel:chowwitt} a pullback $f^!\colon \widetilde{\op{CH}}^\bullet(Y,\mathscr{L})\to\widetilde{\op{CH}}^\bullet(X,f^\ast\mathscr{L})$ was defined for an arbitrary morphism $f\colon X\to Y$ of smooth schemes. By the reinterpretation in terms of sheaf cohomology in \cite[Theorem 2.3.4]{AsokFaselEuler}, the pullback $f^!$ defined in \cite{fasel:chowwitt} agrees with the natural definition of pullback in the Gersten complex for Milnor--Witt K-theory as discussed in \cite[Section 2.3]{AsokFaselEuler}. The composition of exterior product and pullback along the diagonal provides the oriented intersection product, cf. \cite[Definition 6.1]{fasel:chowwitt}
\[
\widetilde{\op{CH}}^i(X,\mathscr{L})\times\widetilde{\op{CH}}^j(X,\mathscr{N})\to \widetilde{\op{CH}}^{i+j}(X,\mathscr{L}\otimes\mathscr{N}).
\]
This product is associative \cite[Proposition 6.6]{fasel:chowwitt}. The change-of-coefficients morphism $\widetilde{\op{CH}}^\bullet(X)\to\op{CH}^\bullet(X)$ is a ring homomorphism, cf. \cite[Proposition 6.12]{fasel:chowwitt}. Moreover, by \cite[Proposition 7.2]{fasel:chowwitt} (together with the reinterpretation in \cite[Section 2.3]{AsokFaselEuler}), the pullbacks for arbitrary morphisms between smooth schemes are ring homomorphisms for the oriented intersection product.

Note \cite[Remark 6.7]{fasel:chowwitt} that the product on Chow--Witt rings is in general neither commutative nor anti-commutative. However, by \cite[Corollary 2.8]{MField}, Milnor--Witt K-theory has a modified version of graded commutativity,
and so have the other graded rings we consider.

\begin{definition}
\label{def:epsgraded}
Let $R$ be a commutative ring and $r  \in R$. A $\mathbb{Z}$-graded $R$-algebra $A_\bullet$ is called \emph{$r$-graded commutative} if for all elements $x\in A_n$ and $y\in A_m$ we have
\[
x\cdot y = r^{nm} y\cdot x.
\]
The special case $r=-1$ is simply called \emph{graded commutative}.
\end{definition}

Recall \cite{morel:puissances},\cite[Corollary 2.8]{MField} that $\op{K}^{\op{MW}}_{\bullet}(F)$ is an $\epsilon$-graded commutative $\op{GW}(F)$-algebra, for $\epsilon:=-\langle-1\rangle$ in the Grothendieck--Witt ring $\op{GW}(F)$. As the canonical projection $\op{K}^{\op{MW}}_{\bullet}(F)\to\op{K}^{\op{M}}_\bullet(F)$ maps $\epsilon$ to $-1$, we recover the well-known $(-1)$-graded commutativity
of $\op{K}^{\op{M}}_\bullet(F)$. Interpreting $\op{CH}$ as cohomology of the Gersten complex for the sheaf $\mathbf{K}^{\op{M}}_\bullet$, we obtain an additional $-1$ which leads to the well-known $(+1)$-graded commutativity of Chow groups.

\begin{proposition}
\label{prop:cwepsgraded}
Let $X$ be a smooth $F$-scheme. Then the Chow--Witt ring $\widetilde{\op{CH}}^\bullet(X)$ is a $\langle-1\rangle$-graded commutative $\op{GW}(F)$-algebra. Similarly, the cohomology ring $\op{H}^\bullet(X,\mathbf{I}^\bullet)$ is a $(-1)$-graded commutative $\op{W}(F)$-algebra. 
\end{proposition}

\begin{proof} 
As $\op{K}^{\op{MW}}_\bullet$ is $\epsilon$-graded commutative, the first statement follows as in our above remark on Chow groups, using that passing to the cohomology of the Gersten complex contributes an extra $(-1)$ and $-\epsilon=\langle-1\rangle$. For the second statement, note that $\epsilon$ maps to $1$ in $\mathbf{I}^{\bullet}=\mathbf{K}^{\op{W}}_{\bullet}$ which hence
is $(+1)$-graded commutative, and hence induces a $(-1)$-graded commutative structure on $\op{H}^\bullet(X,\mathbf{I}^\bullet)$, cf. also \cite[Remarks 6.7 and  6.13]{fasel:chowwitt}.
\end{proof}

\begin{proposition}
\label{prop:gysinder}
Let $X$ be a smooth $F$-scheme and let $\gamma\colon E\to X$ be a vector bundle over $X$. Consider the localization sequence associated to $E$ with closed subscheme $s_0(X)$ and open complement $U=E\setminus s_0(X)$. The boundary map
\[
\partial\colon \op{H}^q(U,\mathbf{I}^q)\to\op{H}^{q+1}_{s_0(X)}(E,\mathbf{I}^q)
\]
is a derivation for the intersection products, i.e., we have 
\[
\partial(x\cdot y)=\partial(x)\cdot y + (-1)^{|y|} x\cdot\partial(y).
\]
\end{proposition}

\begin{proof}
Recall that the oriented intersection product is given by the exterior product of cycles composed with the restriction along the diagonal. We first consider the behaviour of the boundary map with respect to the exterior product. Recall that the boundary map in the localization sequence is given by taking a representative of a cycle $\alpha\in \op{H}^q(U,\mathbf{I}^q)$, viewing it as a chain in $C^q(E,\mathbf{I}^q)$. Its boundary will have trivial restriction to $U$ and will therefore, by definition, be a cycle on $E$ with support in $X$. An analogue of \cite[Lemma 4.9]{fasel:chowwitt}, following Rost's proof \cite[14.4, p. 391]{rost}, is true for $\op{H}^\bullet(-,\mathbf{I}^\bullet)$. This implies that the differential is a derivation with respect to the exterior product of cycles, i.e., that 
\[
\partial(\sigma\boxtimes \tau) = \partial(\sigma)\boxtimes\tau + (-1)^{|\tau|} \sigma\boxtimes\partial(\tau). 
\]
Now the statement for the intersection product follows since the boundary of the localization commutes with restriction along the diagonal map. 
\end{proof}

\begin{remark}\label{ringextra}
A more general statement is required for Chow--Witt groups to be a ring cohomology theory in the sense of \cite[Definition 2.2]{PaninWalter}. This follows essentially by deformation to the normal cone.
\end{remark}

Finally, as a consequence of the above statements on intersection products, we can show multiplicativity of the Euler class:

\begin{lemma}
\label{lem:eulermult}
Let $X_1,X_2$ be smooth schemes over $F$. Let $p_1\colon E_1\to X_1$ and $p_2\colon E_2\to X_2$ be two vector bundles of ranks $r_1$ and $r_2$, respectively. Then we have 
\[
\op{e}_{r_1+r_2}(E_1\boxtimes E_2)=\op{e}_{r_1}(E_1)\boxtimes \op{e}_{r_2}(E_2)\in \widetilde{\op{CH}}^{r_1+r_2}(X_1\times X_2,\det(E_1\oplus E_2)),
\]
i.e., the Euler class of the exterior product of vector bundles is the exterior product of the Euler classes. If $X_1=X_2$, then $e_{r_1 + r_2}(E_1\oplus E_2)=e_{r_1}(E_1)e_{r_2}(E_2)$, i.e. the Euler class of the Whitney sum is the intersection product of the Euler classes.
\end{lemma}

\begin{proof}
We have a commutative diagram
\[
\xymatrix{
\widetilde{\op{CH}}^{0}(X_1)\times\widetilde{\op{CH}}^0(X_2) \ar[r] \ar[d]_\boxtimes & \widetilde{\op{CH}}^{r_1}(X_1,\det E_1)\times\widetilde{\op{CH}}^{r_2}(X_2,\det E_2) \ar[d]^\boxtimes\\
\widetilde{\op{CH}}^{0}(X_1\times X_2) \ar[r] & \widetilde{\op{CH}}^{r_1+r_2}(X_1\times X_2, \det(E_1\oplus E_2)). 
}
\]
The top horizontal morphism is the product of the two Euler classes, viewed as push-forward maps. The bottom horizontal map is the Euler class for the exterior product of the bundles, which is what we want to determine. The vertical maps are the exterior product maps. Since the exterior product commutes with push-forward up to a sign which is trivial in codimension $0$, we get the result. The statement about the Whitney sum follows since the Whitney sum is given by pullback of the exterior sum of the vector bundles along the diagonal. On the level of Euler classes, this translates directly into the intersection product since Euler classes are compatible with pullbacks.
\end{proof}

\subsection{Cohomology operations}
\label{sec:cohops}

For the computations of  $\widetilde{\op{CH}}^\bullet({\op{B}}\op{SL}_n)$, we will need some facts about the motivic $\op{Sq}^2$ and a Bockstein-type map. 

Using the quadratic part of the Milnor conjecture as above, we have an exact sequence of strictly $\mathbb{A}^1$-invariant Nisnevich sheaves of abelian groups 
\[
0\to\mathbf{I}^{n+1}\to\mathbf{I}^n\to\mathbf{K}^{\op{M}}_n/2\to 0. 
\]
This should be interpreted as analogue of the exact sequence $0\to\mathbb{Z}\to\mathbb{Z}\to\mathbb{Z}/2\mathbb{Z}\to 0$ with additional weight information. The associated long exact sequence of Nisnevich cohomology groups is given by 
\begin{equation}
\label{eq:baer}
\cdots\to \op{H}^i_{\op{Nis}}(X,\mathbf{K}^{\op{M}}_n/2)\xrightarrow{\beta} \op{H}^{i+1}_{\op{Nis}}(X,\mathbf{I}^{n+1})\xrightarrow{\eta}
\op{H}^{i+1}_{\op{Nis}}(X,\mathbf{I}^n) \xrightarrow{\rho} \op{H}^{i+1}_{\op{Nis}}(X,\mathbf{K}^{\op{M}}_n/2)\to\cdots
\end{equation}
which some people call the {\em B\"ar sequence}. The notation $\beta$ stresses the analogy with the classical Bockstein map. The map $\eta$ is simply the map induced by inclusion $\mathbf{I}^{n+1}\subseteq \mathbf{I}^n$ and can alternatively be viewed as multiplication by the element $\eta\in\op{K}^{\op{MW}}_{-1}(F)$, cf.~\cite{morel:puissances} and above, and refines multiplication by $2$ in classical topology. 

In the critical degree $i=n$, the Bockstein  map is of the form $\beta\colon \op{Ch}^n(X)\to\op{H}^{n+1}(X,\mathbf{I}^{n+1})$, and has been used in \cite{fasel:ij} to define integral Stiefel--Whitney classes. On the other hand, there is a reduction map $\rho\colon \op{H}^n_{\op{Nis}}(X,\mathbf{I}^n)\to\op{Ch}^n(X)$. The following statement is due to \cite[Theorem 1.1]{totaro:witt}, cf. also \cite[Remark 10.5]{fasel:ij} and refines results about $\op{Sq}^1$ in classical algebraic topology.

\begin{proposition}[(Totaro)]
\label{prop:sq2}
The composition 
\[
\op{Ch}^n(X)\xrightarrow{\beta}\op{H}^{n+1}_{\op{Nis}}(X,\mathbf{I}^{n+1}) \xrightarrow{\rho}\op{Ch}^{n+1}(X)
\]
is the motivic Steenrod square $\op{Sq}^2$. 
\end{proposition}

Note that this together with the B\"ar sequence implies that $\op{Sq}^2\op{Sq}^2=0$. Even more is true:
\begin{lemma}
\label{lem:jacobi}
The Steenrod square $\op{Sq}^2\colon \op{Ch}^n(-)\to\op{Ch}^{n+1}(-)$ has the following properties:
\begin{enumerate}
\item It is a derivation, i.e., $\op{Sq}^2(X\cdot Y)=X\cdot\op{Sq}^2(Y)+ \op{Sq}^2(X)\cdot Y$. 
\item It satisfies the Jacobi identity
\[
\op{Sq}^2(X)\op{Sq}^2(YZ) + \op{Sq}^2(XY)\op{Sq}^2(Z) + \op{Sq}^2(XZ)\op{Sq}^2(Y) = 0.
\]
\end{enumerate}
\end{lemma}

\begin{proof}
The following equality is a special case of the motivic Cartan formula, cf. \cite[Proposition 9.7]{voevodsky}
with $X$ and $Y$ in arbitrary bidegrees:
\[
\op{Sq}^2(X\cdot Y)=\op{Sq}^2(X)\cdot Y+ X\op{Sq}^2(Y) + \tau \op{Sq}^1(X)\cdot\op{Sq}^1(Y)
\]
where $\tau \in \op{H}^{0,1}_{\op{mot}}(\op{Spec}(F),\mathbb{Z}/2\mathbb{Z})\cong \mu_2(F)$ is the generator and the map 
\[
\op{Sq}^1\colon \op{H}^{i,j}_{\op{mot}}(-,\mathbb{Z}/2\mathbb{Z})\to \op{H}^{i+1,j}_{\op{mot}}(-,\mathbb{Z}/2\mathbb{Z})
\]
is the motivic cohomology Bockstein.
Now for $X$ and $Y$ of respective bidegrees $(2i,i)$ and $(2j,j)$, $\op{Sq}^1(X)\cdot\op{Sq}^1(Y)$ has bidegree $(2i+2j+2,i+j)$ and hence vanishes by \cite[Theorem 19.3]{mvw}. This implies that $\tau\op{Sq}^1(X)\cdot\op{Sq}^1(Y)=0$, proving (1). 
Alternatively, this can be derived from \cite[Theorem 9.2]{brosnan}.

A standard computation, using the motivic Cartan formula above and the fact that $\op{Sq}^2$ is a differential, implies a precursor of the  Jacobi identity:
\[
\op{Sq}^2(X)\op{Sq}^2(YZ) - \op{Sq}^2(XY)\op{Sq}^2(Z) + (-1)^{|Y||Z|} \op{Sq}^2(XZ)\op{Sq}^2(Y) = \tau f(X,Y,Z)
\]
where $f(X,Y,Z)$ is an explicit polynomial in the Steenrod squares of $X$, $Y$ and $Z$. If we restrict to elements $X,Y,Z\in \op{Ch}^\bullet(X)$, i.e., to ``geometric'' bidegrees $(2n,n)$, everything is $2$-torsion and the signs disappear.
 As before, still restricting to geometric bidegrees, the element $f(X,Y,Z)$ vanishes. We then obtain the required Jacobi identity. 
\end{proof}

\subsection{The key diagram}
\label{sec:key}
 
The above fiber product description of Milnor--Witt K-theory leads to various exact sequences which fit together in a large commutative diagram.
Its key part which we now display will be crucial for various parts of the article:
\[
\xymatrix{
&\op{CH}^n(X)\ar[r]^= \ar[d] & \op{CH}^n(X)\ar[d]^2 \\
\op{H}^n_{\op{Nis}}(X,\mathbf{I}^{n+1})\ar[r]\ar[d]_=& \widetilde{\op{CH}}^n(X)\ar[r]\ar[d]&\op{CH}^n(X)\ar[r]^{\partial}\ar[d]^{\bmod 2}& \op{H}^{n+1}_{\op{Nis}}(X,\mathbf{I}^{n+1})\ar[d]^=\\
\op{H}^n_{\op{Nis}}(X,\mathbf{I}^{n+1})\ar[r]_\eta& \op{H}^{n}_{\op{Nis}}(X,\mathbf{I}^{n})\ar[r]_\rho\ar[d]& \op{Ch}^n(X)\ar[r]^\beta\ar[rd]_{\op{Sq}^2}\ar[d]& \op{H}^{n+1}_{\op{Nis}}(X,\mathbf{I}^{n+1})\ar[d]^\rho \\
&0&0&\op{Ch}^{n+1}(X)
}
\]
 The lower exact sequence is the B\"ar sequence of (\ref{eq:baer}) discussed above, and the commutativity of the lower-right triangle is Totaro's identification of $\op{Sq}^2$. The key square in the middle expresses Chow--Witt groups in terms of ordinary Chow-theory  and the $\mathbf{I}^j$-cohomology. All  maps in the square are change-of-coefficient maps and commute with flat pullbacks and proper push-forwards.

\subsection{A fiber product statement}

We now establish a general condition under which the Chow--Witt ring is simply the fiber product of some part of the Chow ring and $\mathbf{I}^\bullet$-cohomology.

\begin{proposition}
\label{prop:cartesian}
Let $X$ be a smooth scheme over $F$. The canonical ring homomorphism
\[
c\colon \widetilde{\op{CH}}^\bullet(X)\to \op{H}^\bullet(X,\mathbf{I}^\bullet)\times_{\op{Ch}^\bullet(X)}\ker\partial
\] 
induced from the above key square in Section~\ref{sec:key} is always surjective. It is injective if one of the following conditions holds:
\begin{enumerate}
\item $\op{CH}^\bullet(X)$ has no non-trivial $2$-torsion. 
\item the map $\eta\colon \op{H}^n(X,\mathbf{I}^{n+1})\to\op{H}^n(X,\mathbf{I}^n)$ is injective.
\end{enumerate}
\end{proposition}

\begin{proof}
The existence of $c$ follows from the commutativity of the key diagram together with the universal property of the fiber product. As all maps considered here are ring homomorphisms, the fiber product in graded abelian groups has a natural ring structure such that $c$ is a graded ring homomorphism. Hence it suffices to prove the claim for graded abelian groups,
and even in individual degrees.


Surjectivity of the map follows from the long exact sequences giving rise to the key diagram. The lower vertical morphisms are surjective in general, and the upper horizontal morphism in the square we claim cartesian is surjective by definition. Consider an element $(\sigma,\tau)\in  \op{H}^n(X,\mathbf{I}^n)\times_{\op{Ch}^n(X)} \ker\partial$.  We can choose an arbitrary element $\widetilde{\tau}\in \widetilde{\op{CH}}^n(X)$ lifting $\tau$. This lift will map to some element  $\gamma\in\ker\partial$ whose reduction to $\op{Ch}^n(X)$ equals the reduction of $\tau$. Therefore, the difference $\tau-\gamma$ lifts to  $\op{CH}^n(X)$ (along the multiplication by $2$ map in the top square) and we can add its image in $\widetilde{\op{CH}}^n(X)$ to $\widetilde{\tau}$. This proves surjectivity.

It remains to prove the injectivity of the natural map. For case (1), assume that $\op{CH}^n(X)$ has no non-trivial $2$-torsion. Let $\alpha\in\widetilde{\op{CH}}^n(X)$ be a class whose images in $\op{CH}^n(X)$ and $\op{H}^n(X,\mathbf{I}^n)$ are trivial. The latter statement implies that the class lifts to a class $\widetilde{\alpha}\in\op{CH}^n(X)$. The fact that its image in $\op{CH}^n(X)$ (under the natural change-of-coefficients map $\widetilde{\op{CH}}^n(X)\to\op{CH}^n(X)$) is trivial implies that $\widetilde{\alpha}\in{{}_2\op{CH}^n(X)}$. By assumption, $\widetilde{\alpha}=0$, implying $\alpha=0$ as claimed. The case (2) is done similarly, using the map $\eta$ instead of the multiplication by $2$ on Chow groups.
\end{proof}

\begin{remark}
The preceding arguments carry over to non-trivial twists when restricting to graded abelian groups. For the ring structure, however, the total Chow--Witt ring (direct sum over all possible twists indexed by  ${\op{Pic}}/2$) has to be considered.
\end{remark}


\begin{remark}
The conditions of Proposition~\ref{prop:cartesian} are satisfied for the cases $X={\op{B}}\op{Sp}_{2n}$, ${\op{B}}\op{SL}_n$ and ${\op{B}}\op{GL}_n$ (or rather  their finite-dimensional approximations) by Totaro's computations \cite{totaro:bg}.
 More generally, they are satisfied for cellular varieties. The special case $X=\mathbb{P}^n$ of the above result recovers statements implicit in the Chow--Witt ring computation of \cite[Section 11]{fasel:ij}.
\end{remark} 

\begin{remark}
Denote by ${\op{B}}_{\op{gm}}\op{O}(1)$ the geometric classifying space as considered in \cite{MV}. It coincides with the classifying space ${\op{B}}_{\et}\op{O}(1)$ of \'etale-locally trivial $\op{O}(1)$-torsors. 
The space ${\op{B}}_{\op{gm}}\op{O}(1)$ doesn't satisfy the $2$-torsion condition above as it has Picard group $\mathbb{Z}/2\mathbb{Z}$. We do not expect it to be a fiber product, at least not for both possible twists.
 Actually, this should already be visible on the level of $\op{H}^1$ for which one can use $X:=(\mathbb{A}^3\setminus\{0\})/\mu_2$ where $\mu_2$ acts via the antipodal action as a suitable model for $\widetilde{\op{CH}}^1({\op{B}}_{\et}\op{O}(1))$. The failure of the fiber product statement should be related to some elements in $\op{H}^1(X,\mathbf{I}^2)$ being in the image of $\op{H}^0(X,\mathbf{K}_1^{\op{M}}/2)$ but not in the image of $\op{H}^0(X,\mathbf{K}_1^{\op{M}})$. More generally, the Chow--Witt rings of (special) orthogonal groups seem to be significantly more complicated than the Chow rings and will be subject to further research.
\end{remark}

\section{Chow--Witt rings of classifying spaces}
\label{sec:cwbg} 


We recall the definition of Chow--Witt rings of classifying spaces and some of the relevant functoriality properties.

For a linear algebraic group $G$ over $F$, the Chow ring of the classifying space of $G$ has been defined by Totaro \cite{totaro:bg} via smooth finite-dimensional partial resolutions of the trivial $G$-action on $\op{Spec}(F)$. The resulting Chow ring coincides with the ``geometric'' bidegrees $(2n,n)$ of the motivic cohomology of the geometric classifying space ${\op{B}}_{\op{gm}}G$ in the Morel--Voevodsky $\mathbb{A}^1$-homotopy category. The same procedure can be applied to define the Chow--Witt ring of the classifying space, cf. \cite[Section 3.2]{AsokFaselEuler}. We shortly recall the definition and stabilization properties. 

The following is a restatement of \cite[Theorem 1.1]{totaro:bg} for the Chow--Witt rings. A similar statement has been proved in \cite[Theorem 3.2.2]{AsokFaselEuler}. However, we really need the full independence of the choice of representation because we need to deal with restriction maps associated to group homomorphisms and the appropriate functorialities. 

\begin{proposition}
\label{prop:chwbg}
Let $F$ be a perfect field and let $G$ be a linear algebraic group over $F$. Let $V$ be a representation of $G$ over $F$ such that $G$ acts freely outside a $G$-invariant closed subset $S\subset V$ of codimension $\geq s$. Assume that the geometric quotient $(V\setminus S)/G$ exists as a scheme over $F$. Then $\widetilde{\op{CH}}^{<s-1}((V\setminus S)/G)$ is independent of the representation $V$ and the closed subset $S$. 
\end{proposition}

\begin{proof}
The proof follows \cite[Theorem 1.1]{totaro:bg}. The independence of $S$ follows from the excision lemma \cite[Lemma 2.4.2]{AsokFaselEuler} (which is a consequence of the localization sequence). The independence from $V$ follows from Bogomolov's double fibration construction. Let $V$ and $W$ be two representations of $G$ such that $G$ acts freely outside subsets $S_V$ and $S_W$ of codimension $\geq s$ such that the quotients $(V\setminus S_V)/G$ and $(W\setminus S_W)/G$ exist as schemes. Then we have vector bundles $((V\setminus S_V)\times W)/G\to (V\setminus S_V)/G$ and $(V\times (W\setminus S_W))/G\to (W\setminus S_W)/G$. Applying the above independence of $S$ to the representation $V\oplus W$ together with homotopy invariance for Chow--Witt groups yields the claim.
\end{proof}

\begin{definition}
Let $F$ be a field and let $G$ be a linear algebraic group over $F$. Define the Chow--Witt ring $\widetilde{\op{CH}}^\bullet({\op{B}}G)$ of the classifying space by setting 
\[
\widetilde{\op{CH}}^i({\op{B}}G):=\widetilde{\op{CH}}^i((V\setminus S)/G)
\]
for an appropriate $(V,S)$ as in Proposition~\ref{prop:chwbg} where $S$ has codimension greater than $i+1$ in $V$. The ring structure on $\widetilde{\op{CH}}^\bullet({\op{B}}G)$ is the one inherited from the Chow--Witt rings of finite-dimensional approximations. 
\end{definition}

We need to establish finite-dimensional models for the universal bundles. Let $V$ be a representation such that $G$ acts freely outside a closed $G$-stable subset $S\subset V$ of codimension $\geq s$ such that the quotient $(V\setminus S)/G$ exists as an $F$-scheme. Then $V\setminus S\to (V\setminus S)/G$ is a principal $G$-bundle. In fact, in the Morel--Voevodsky $\mathbb{A}^1$-homotopy category \cite{MV} we may write ${\op{B}}G:=\op{colim}_{(V,S)} (V\setminus S)/G$ and ${\op{E}}G:=\op{colim}_{(V,S)}V\setminus S \to{\op{B}}G$
with the colimit taken along the natural projection maps. Then it follows that $V\setminus S\to (V\setminus S)/G$ is the pullback of the universal bundle ${\op{E}}G\to{\op{B}}G$ along the natural approximation map $(V\setminus S)/G\to{\op{B}}G$. In this sense, the quotient maps $V\setminus S\to(V\setminus S)/G$ are finite-dimensional approximations to the universal principal $G$-bundle.

As proved in the following proposition, the Chow--Witt ring of the classifying space of a reductive group $G$ equals the ring of Chow--Witt-valued characteristic classes of $G$-bundles, in complete analogy with \cite[Theorem 1.3]{totaro:bg}. Here {\em characteristic class} (of degree $i$) means by definition an assignment $\alpha$ associating to a smooth quasi-projective $F$-scheme $X$ with a principal $G$-bundle $E\to X$ an element $\alpha(E)\in\widetilde{\op{CH}}^i(X)$ such that for any morphism $f\colon Y\to X$ we have $\alpha(f^\ast E)=f^\ast(\alpha(E))$.

\begin{proposition}
\label{prop:characteristic}
Let $F$ be a field and let $G$ be a reductive group over $F$. There is a natural bijection between $\widetilde{\op{CH}}^i({\op{B}}G)$ and the set of degree $i$ characteristic classes of $G$-bundles.
\end{proposition}

\begin{proof}
The proof follows the one for \cite[Theorem 1.3]{totaro:bg}. As in loc.cit., we get a natural homomorphism from the ring of characteristic classes to $\widetilde{\op{CH}}^\bullet({\op{B}}G)$. Using \cite[Lemma 1.6]{totaro:bg} and homotopy invariance of Chow--Witt groups, the injectivity and surjectivity of the natural morphism follow as in loc.cit. 
\end{proof}

\begin{remark}
A similar statement is true for twisted characteristic classes in $\widetilde{\op{CH}}^\bullet(X,\mathscr{L}_E)$. Here $\mathscr{L}_E$ is a line bundle functorially associated to the principal $G$-bundle $E\to X$ whose characteristic classes are being considered, i.e., it corresponds to a characteristic class in $\op{Pic}({\op{B}}G)/2$. This is slightly more complicated to formulate, and irrelevant for us here, since we only deal with groups where the Picard group of the classifying space is trivial. 
\end{remark}

\begin{remark}
\label{rem:classmap}
Note that \cite[Lemma 1.6]{totaro:bg} justifies the use of the terminology ``classifying space'' for the object approximated by varieties of the form $(V\setminus S)/G$. It states that for any quasi-projective variety $X$ with a principal $G$-bundle $E\to X$, there is an affine space bundle $\pi\colon X'\to X$, a suitable representation $V$ where $G$ acts freely outside a closed subset of codimension $\geq s$, and a classifying map $f\colon X'\to(V\setminus S)/G$ such that $\pi^\ast E\cong f^\ast \gamma$ where $\gamma\colon V\setminus S\to(V\setminus S)/G$ is the approximation of the universal principal $G$-bundle. 

Actually, it follows from \cite[Proposition 4.2.6]{MV} that the colimit of the schemes $(V\setminus S)/G$ over a suitable system of representations (where the codimensions of the sets $S$ tend to infinity) is equivalent to the geometric classifying space ${\op{B}}_{\op{gm}}G$ considered by Morel and Voevodsky. In particular, this colimit is a classifying space in the sense that there is a natural bijection between \'etale locally trivial torsors over a smooth scheme $X$ and the homotopy classes of maps $[X,{\op{B}}_{\op{gm}}G]$ in the simplicial homotopy category. This yields another (equivalent) way of defining Chow--Witt groups of the classifying space, using the $\mathbb{A}^1$-representability of Chow--Witt groups. In this language there is e.g. a different
proof for the next result, using $\mathbb{A}^1$-representability of Chow-Witt groups.
\end{remark}

Now we want to consider induced morphisms on Chow--Witt rings of classifying spaces. 

\begin{proposition}
\label{prop:inducedhom}
Let $F$ be a perfect field, and let $\iota\colon H\hookrightarrow G$ be a closed monomorphism of linear algebraic groups over $F$. Then there is a ring homomorphim
\[
\iota^\ast\colon \widetilde{\op{CH}}^\bullet({\op{B}}G)\to \widetilde{\op{CH}}^\bullet({\op{B}}H).
\]
These ring homomorphisms are compatible with compositions $G_1 \hookrightarrow G_2 \hookrightarrow G_3$.
\end{proposition}

\begin{proof}
We can define a restriction homomorphism
\[
\iota^\ast\colon \widetilde{\op{CH}}^\bullet({\op{B}}G)\to \widetilde{\op{CH}}^\bullet({\op{B}}H)
\]
as follows: fix a degree $i$ and let $V$ be a representation of $G$ such that there exists an open subscheme $U\subset V$ where the $G$-action is free and whose closed complement has codimension $>i+1$. Then, by the above discussion,  $\widetilde{\op{CH}}^i({\op{B}}G)\cong \widetilde{\op{CH}}^i(U/G)$. Moreover, since $H\subset G$ is a closed subgroup, we also have $\widetilde{\op{CH}}^i({\op{B}}H)\cong \widetilde{\op{CH}}^i(U/H)$. Finally, the natural quotient morphism $U/H\to U/G$ induced by $\op{id}_U$ is a $G/H$-fiber bundle in the sense that there exists an \'etale covering $W\to U/G$ such that there is an isomorphism
\[
W\times_{U/G}U/H \cong W\times G/H
\]
of $F$-schemes with $G$-action. In particular, the quotient morphism is smooth. We can define the restriction map in degree $i$ by 
\[
\iota^\ast\colon  \widetilde{\op{CH}}^i({\op{B}}G) \cong \widetilde{\op{CH}}^i(U/G)\to \widetilde{\op{CH}}^i(U/H)\cong \widetilde{\op{CH}}^i({\op{B}}H)
\]
with the morphism in the middle given by the ordinary (flat) pullback for Chow--Witt groups. By the double fibration trick and the ensuing independence of the representation, this yields a well-defined homomorphism of Chow--Witt rings (because flat pullback is a homomorphism of Chow--Witt rings).

For the functoriality, let $G_1\subset G_2\subset G_3$ be composable closed inclusion homomorphisms of linear algebraic groups. As above, we can fix a degree $i$ and consider a $G_3$-representation $V$ over $F$ which has a subset $U\subset V$ where $G_3$ acts freely and whose complement has codimension $>i+1$. Then we get a sequence of smooth quotients
\[
U/G_1\xrightarrow{f} U/G_2\xrightarrow{g} U/G_3. 
\]
In this situation, \cite[Theorem 2.1.3]{AsokFaselEuler} tells us that we have 
\[
(g\circ f)^\ast=f^\ast g^\ast\colon  \widetilde{\op{CH}}^\bullet(U/G_3)\to \widetilde{\op{CH}}^\bullet(U/G_1). 
\]
Together with the independence of choice of representation in the definition of Chow--Witt rings of classifying spaces, this proves the functoriality claim.
\end{proof}

\begin{remark}
One can define more generally restriction maps for arbitrary group homomorphisms $\phi\colon H\to G$: start with a $G$-representation $V$ with open subset $U$ satisfying the codimension condition for degree $i$. Take a $H$-representation $V'$ with open subset $U'$ satisfying the codimension condition for degree $i$. Then we have a morphism $(U\times U')/H\to U/H \to U/G$. All the above statements remain valid.
\end{remark}

\section{Chow--Witt rings of classifying spaces of symplectic groups}
\label{sec:symplectic}

In this section, we compute $\widetilde{\op{CH}}^{\bullet}(\op{B}\op{Sp}_{2n})$. We adapt arguments of \cite{brown} and \cite{RojasVistoli}. The global structure of the argument is the following: the localization sequence for a geometric model of the open stabilization inclusion ${\op{B}}\op{Sp}_{2n-2} \subset {\op{B}}\op{Sp}_{2n}$ implies existence of Pontryagin classes, and that they are uniquely defined by requiring compatibility with stabilization and the identification of the top Pontryagin class with the Thom class. Then one proves a Whitney-sum formula for the top Pontryagin class and uses the quaternionic projective bundle formula to deduce a splitting principle. The description of the Chow--Witt ring as polynomial ring in the Pontryagin classes then follows using rather standard arguments.

Comparing our arguments with those of \cite{PaninWalter} in their axiomatic setting, there are both similarities and differences. A more detailed comparison as well as consequences on symplectic orientability of Chow--Witt groups can be found in Subsection~\ref{sec:sympor}. 

\subsection{Geometric setup}\label{geomsetup}

We begin by setting up the localization sequence, which will be used in the inductive computation of Chow--Witt rings of classifying spaces of symplectic groups. This sequence will relate the Chow--Witt rings of ${\op{B}}\op{Sp}_{2n}$ and ${\op{B}}\op{Sp}_{2n-2}$.

On the level of actual classifying spaces (for instance in the Morel--Voevodsky $\mathbb{A}^1$-homotopy category), the idea is the following: we have a classifying space ${\op{B}}\op{Sp}_{2n}$ and a universal $\op{Sp}_{2n}$-torsor $\epsilon_{2n}\colon {\op{E}}\op{Sp}_{2n}\to {\op{B}}\op{Sp}_{2n}$ over it. For the tautological representation of $\op{Sp}_{2n}$, we consider the associated rank $2n$ symplectic bundle $E_{2n}\to{\op{B}}\op{Sp}_{2n}$. By homotopy invariance, the total space can be identified with ${\op{B}}\op{Sp}_{2n}$. The zero section is ${\op{B}}\op{Sp}_{2n}$ by construction. The remaining piece, the complement of the zero section, can actually be identified with ${\op{B}}\op{Sp}_{2n-2}$ because a Levi subgroup of the stabilizer of a nonzero point of the tautological representation of $\op{Sp}_{2n}$ is $\op{Sp}_{2n-2}$. In the following, we will explain in detail how to realize this picture on the level of finite-dimensional approximations of the classifying spaces. 

Recall from Proposition~\ref{prop:chwbg} that finite-dimensional approximations of ${\op{B}}\op{Sp}_{2n}$ can be obtained by taking a finite-dimensional $\op{Sp}_{2n}$-representation $V$ on which $\op{Sp}_{2n}$ acts freely outside a closed stable subset $Y$ of codimension $s$, and considering the quotient $X(V):=(V\setminus Y)/\op{Sp}_{2n}$. This computes $\widetilde{\op{CH}}^\bullet({\op{B}}\op{Sp}_{2n})$ in degrees $\leq s-2$. A finite-dimensional model for the universal torsor over $X(V)$ is given by the projection $p\colon V\setminus Y\to X(V)$.

Consider the tautological $\op{Sp}_{2n}$-representation on $\mathbb{A}^{2n}$, viewed as underlying affine space of the standard symplectic vector space of dimension $2n$. We can pass from the $\op{Sp}_{2n}$-torsor $p\colon V\setminus Y\to X(V)$ to the associated symplectic vector bundle for the tautological representation; denote this symplectic vector bundle by $\gamma_V\colon E_{2n}(V)\to X(V)$. We denote by $S_{2n}(V)\to X(V)$ the complement of the zero-section of $\gamma_V$. Since $\gamma_V\colon E_{2n}(V)\to X(V)$ is a vector bundle and hence an $\mathbb{A}^1$-weak equivalence, we have $\widetilde{\op{CH}}^i(E_{2n}(V))\cong\widetilde{\op{CH}}^i(X(V))$. Moreover, the zero section  of $E_{2n}(V)\to X(V)$ is a subscheme of $E_{2n}(V)$ isomorphic to $X(V)$ and hence has the same Chow--Witt groups. All these are isomorphic to $\widetilde{\op{CH}}^i({\op{B}}\op{Sp}_{2n})$ in the range $i\leq s-2$. 

Now we want to compare this to the classifying space of $\op{Sp}_{2n-2}$ and actually identify $S_{2n}(V)$ as a bundle over a finite-dimensional approximation of its classifying space. Considering the representation $V$, the group $\op{Sp}_{2n-2}$, embedded via the standard stabilization embedding, acts freely on $V\setminus Y$ and the quotient morphism $(V\setminus Y)/\op{Sp}_{2n-2}\to X(V)$ is a Zariski locally trivial fiber bundle with fiber $\op{Sp}_{2n}/\op{Sp}_{2n-2}$. To identify the open complement of the zero section $S_{2n}(V)\to X(V)$, consider the $\op{Sp}_{2n}$-action on $\mathbb{A}^{2n}$ in the tautological representation. There are two orbits, $\{0\}$ and $\mathbb{A}^{2n}\setminus\{0\}$. The stabilizer group $G$ of a point in the open orbit is a semi-direct product of $\op{Sp}_{2n-2}$ acting on a unipotent subgroup of dimension $2n-1$ in $\op{Sp}_{2n}$, cf. e.g. \cite[p.280]{RojasVistoli}. Moreover, the inclusion of $\op{Sp}_{2n-2}\subset \op{Sp}_{2n}$ is conjugate to the standard stabilization embedding. This implies that the morphism $(V\setminus Y)/\op{Sp}_{2n-2}\to X(V)$ lifts to a map $(V\setminus Y)/\op{Sp}_{2n-2}\to S_{2n}(V)$ and this lift is a torsor under the unipotent radical of the stabilizer group $G$. In particular, we get an $\mathbb{A}^1$-weak equivalence $(V\setminus Y)/\op{Sp}_{2n-2}\to S_{2n}(V)$ which therefore induces an isomorphism of Chow--Witt rings. In particular, the open complement $S_{2n}(V)$ computes the Chow--Witt ring of ${\op{B}}\op{Sp}_{2n-2}$ in a range and by construction, the  morphism 
\[
\widetilde{\op{CH}}^\bullet(X(V))\xrightarrow{\gamma_V^\ast} \widetilde{\op{CH}}^\bullet(S_{2n}(V)) \cong \widetilde{\op{CH}}^\bullet((V\setminus Y)/\op{Sp}_{2n-2})
\]
induced from the natural quotient $(V\setminus Y)/\op{Sp}_{2n-2}\to X(V)$ is the stabilization morphism.

Consequently, by going to the limit over a suitable system of representations $V$, we find that there is an exact localization sequence for Chow--Witt rings of classifying spaces of symplectic groups
\begin{eqnarray*}
\cdots&\to&\widetilde{\op{CH}}^q_{{\op{B}}\op{Sp}_{2n}}(E_{2n})\to \widetilde{\op{CH}}^q({\op{B}}\op{Sp}_{2n})\to \widetilde{\op{CH}}^q({\op{B}}\op{Sp}_{2n-2})\\&\to& \op{H}^{q+1}_{\op{Nis},{\op{B}}\op{Sp}_{2n}}(E_{2n},\mathbf{K}^{\op{MW}}_{q})\to \cdots
\end{eqnarray*}
arising in finite-dimensional models from the schemes $E_{2n}(V)$ where the closed subscheme is the zero-section $X(V)\subset E_{2n}(V)$ and its open complement is $S_{2n}(V)$. The map $\widetilde{\op{CH}}^q({\op{B}}\op{Sp}_{2n})\to \widetilde{\op{CH}}^q({\op{B}}\op{Sp}_{2n-2})$ is the restriction along the stabilization inclusion $\op{Sp}_{2n-2}\to\op{Sp}_{2n}$. 

It remains to identify the cohomology of support in the localization sequence above with the appropriate cohomology groups of ${\op{B}}\op{Sp}_{2n}$. To do this, one uses  the d\'evissage isomorphisms
\begin{eqnarray*}
\widetilde{\op{CH}}^{q-2n}({\op{B}}\op{Sp}_{2n}) &\cong& \widetilde{\op{CH}}^q_{{\op{B}}\op{Sp}_{2n}}(E_{2n})\\
\op{H}^{q+1-2n}_{\op{Nis}}({\op{B}}\op{Sp}_{2n},\mathbf{K}^{\op{MW}}_{q-2n})&\cong&
\op{H}^{q+1}_{\op{Nis},{\op{B}}\op{Sp}_{2n}}(E_{2n},\mathbf{K}^{\op{MW}}_{q})
\end{eqnarray*}
which arise from the homotopy purity theorem of Morel--Voevodsky \cite{MV} and the identification of contractions $(\mathbf{K}^{\op{MW}}_n)_{-1}\cong \mathbf{K}^{\op{MW}}_{n-1}$, cf. \cite[Proposition 2.8]{AsokFaselSpheres}. 
As mentioned before, the d\'evissage isomorphism depends on the choice of orientation of the normal bundle of the closed subscheme. The right choice of such orientation is explained in the following remark: 

\begin{remark}
\label{rem:axiomchoice}
The d\'evissage isomorphism involves that the coefficients are twisted by the determinant of the normal bundle. 
From Totaro's computation in \cite{totaro:bg} we know that $\op{CH}^\bullet({\op{B}}\op{Sp}_{2n})\cong \mathbb{Z}[\op{c}_2,\dots,\op{c}_{2n}]$, hence the Picard group of line bundles on ${\op{B}}\op{Sp}_{2n}$ is trivial. In particular the determinant of the normal bundle can be trivialized. However, the d\'evissage isomorphism 
\[
\widetilde{\op{CH}}^{q-2n}({\op{B}}\op{Sp}_{2n}) \cong \widetilde{\op{CH}}^q_{{\op{B}}\op{Sp}_{2n}}(E_{2n}),
\]
depends on the choice of a trivialization of the normal bundle. There is a natural choice of trivialization of the universal symplectic bundle over ${\op{B}}\op{Sp}_{2n}$: the choice of a symplectic form $\omega$ on $\mathbb{A}^{2n}$ yields an orientation given by the volume form $\omega^n$. The natural representation of $\op{Sp}_{2n}$ on $\mathbb{A}^{2n}$ preserves the volume form by definition, and hence the vector bundle associated to the universal $\op{Sp}_{2n}$-torsor over ${\op{B}}\op{Sp}_{2n}$ and the natural representation yields an orientation on the universal symplectic bundle. 
Since the normal bundle of the zero section is the universal bundle itself, we obtain a preferred orientation for the normal bundle for ${\op{B}}\op{Sp}_{2n}\hookrightarrow E_{2n}$. This is the orientation we use for the identification in the d\'evissage isomorphism.
\end{remark}

Combining the above exact localization sequence with the d\'evissage isomorphisms we obtain the following statement: 

\begin{proposition}
\label{prop:sploc}
There is a long exact sequence of Milnor--Witt K-cohomology groups of classifying spaces
\begin{eqnarray*}
\cdots&\to&\widetilde{\op{CH}}^{q-2n}({\op{B}}\op{Sp}_{2n})\to \widetilde{\op{CH}}^q({\op{B}}\op{Sp}_{2n})\to \widetilde{\op{CH}}^q({\op{B}}\op{Sp}_{2n-2})\\&\to&
\op{H}^{q+1-2n}_{\op{Nis}}({\op{B}}\op{Sp}_{2n},\mathbf{K}^{\op{MW}}_{q-2n})\to \op{H}^{q+1}_{\op{Nis}}({\op{B}}\op{Sp}_{2n},\mathbf{K}^{\op{MW}}_q)\to\cdots
\end{eqnarray*}
The first map is the composition of the d\'evissage isomorphism with forgetting of support ("multiplication with the Euler class of the universal symplectic bundle $E_{2n}\to{\op{B}}\op{Sp}_{2n}$"). The second map is the restriction along the stabilization inclusion $\op{Sp}_{2n-2}\to\op{Sp}_{2n}$. 
\end{proposition}

\subsection{Definition of characteristic classes}

Denote by $\op{e}_n=\op{e}(\op{Sp}_{2n})\in\widetilde{\op{CH}}^{2n}({\op{B}}\op{Sp}_{2n})$ the Euler class for the universal symplectic bundle $E_{2n}\to{\op{B}}\op{Sp}_{2n}$, cf. Section~\ref{sec:prelims}, given as the image of $\langle 1\rangle\in\op{GW}(F)$ under the map 
\[
\op{GW}(F) \to \widetilde{\op{CH}}^0({\op{B}}\op{Sp}_{2n})\cong \widetilde{\op{CH}}^{2n}_{{\op{B}}\op{Sp}_{2n}}(E_{2n})\to \widetilde{\op{CH}}^{2n}(E_{2n}) \cong \widetilde{\op{CH}}^{2n}({\op{B}}\op{Sp}_{2n}).
\] 
Note again that the Euler class depends on the choice of orientation, cf.~Remark~\ref{rem:axiomchoice}. 

\begin{proposition}
\label{prop:spuniq}
There are unique classes $\op{p}_q(\op{Sp}_{2n})\in\widetilde{\op{CH}}^{2q}({\op{B}}\op{Sp}_{2n})$  such that 
\begin{enumerate}
\item $i^\ast \op{p}_q(\op{Sp}_{2n})=\op{p}_q(\op{Sp}_{2n-2})$ for $q<n$, and
\item  $\op{p}_n(\op{Sp}_{2n})=\op{e}(\op{Sp}_{2n})$.
\end{enumerate}
\end{proposition}

\begin{proof}
Using the d\'evissage isomorphisms, cf. the discussion before Remark~\ref{rem:axiomchoice}, we have $\widetilde{\op{CH}}^{q-2n}({\op{B}}\op{Sp}_{2n})=0$ for $q<2n$ and 
\[
\op{H}^{q+1}_{\op{Nis},{\op{B}}\op{Sp}_{2n}}(E_{2n},\mathbf{K}^{\op{MW}}_{2n})\cong 
\op{H}^{q+1-2n}_{\op{Nis}}({\op{B}}\op{Sp}_{2n},\mathbf{K}^{\op{MW}}_{q-2n})=0
\]
for $q<2n-1$ because the Chow--Witt groups are Nisnevich cohomology groups of $\mathbf{K}^{\op{MW}}_\bullet$. In particular, the restriction along $i\colon {\op{B}}\op{Sp}_{2n-2}\to{\op{B}}\op{Sp}_{2n}$ induces isomorphisms
\[
i^\ast\colon \widetilde{\op{CH}}^q({\op{B}}\op{Sp}_{2n})\cong \widetilde{\op{CH}}^q({\op{B}}\op{Sp}_{2n-2})
\]
for $q<2n-1$. Existence and uniqueness of the classes $\op{p}_q$ then follows from the definition of $\op{p}_n(\op{Sp}_{2n})$ as Euler class. 
\end{proof}

\begin{remark} \label{CHbig}
Once we know that there is a bigraded extension of $\widetilde{\op{CH}}$ satisfying the properties of Panin and Walter \cite{PaninWalter}, their machinery will produce Pontryagin classes in this bigraded theory which then may be compared to ours. Indeed, their Pontryagin classes also satisfy the restriction property, cf. \cite[Theorem 11.4]{PaninWalter}, and their top Pontryagin class coincides with the Thom class, cf. \cite[Theorem 13.2]{PaninWalter}. Hence by our uniqueness theorem below, their Pontryagin classes will be equal to ours and it follows that they all live in geometric bidegrees, which is not immediate from their axiomatic setting. See Section \ref{sec:sympor} for more details.
\end{remark}

Observe that as Pontryagin classes live in even degrees, they commute with all other classes in $\widetilde{\op{CH}}^{\bullet}({\op{B}}\op{Sp}_{2n})$ by Proposition \ref{prop:cwepsgraded}.

\begin{corollary}
\label{cor:spuniq}
Under the natural morphism 
\[
\widetilde{\op{CH}}^\bullet({\op{B}}\op{Sp}_{2n})\to \op{CH}^\bullet({\op{B}}\op{Sp}_{2n})\cong\mathbb{Z}[\op{c}_2,\op{c}_4,\dots,\op{c}_{2n}]
\]
the Pontryagin class $\op{p}_i$ maps to the Chern class $\op{c}_{2i}$. 
\end{corollary}

\begin{proof}
The concrete identification of the Chow-ring of the classifying space ${\op{B}}\op{Sp}_{2n}$ can be found in \cite{totaro:bg} or \cite{RojasVistoli}. An argument analogous to the proof of Proposition~\ref{prop:spuniq} shows that the Chern classes in $\op{CH}^\bullet({\op{B}}\op{Sp}_{2n})$ are uniquely determined by requiring that they are compatible with the stabilization in the symplectic series and that $\op{c}_{2n}$ equals the Euler class $\op{e}_{2n}$ in the Chow ring. These two properties are satisfied for the Chern classes in the Chow ring, essentially by inspection of the computation of $\op{CH}^{\bullet}({\op{B}}\op{Sp}_{2n})$ in \cite{RojasVistoli}. Now both the top Pontryagin class in the Chow--Witt ring and the top Chern class in the Chow ring are defined in terms of the Euler class.
Moreover, the Chow--Witt-theoretic Euler class naturally maps to the Chow-theoretic Euler class, as they are both defined in terms of the localization sequence for the universal bundle, and these
localization sequences are compatible under the forgetful map
$\widetilde{\op{CH}} \to \op{CH}$, we conclude that the top Pontryagin class must map to the top Chern class. Compatibility with stabilization proves the claim for all Pontryagin classes.
\end{proof}

\subsection{Computation of the Chow--Witt ring}

Next, we have to discuss the independence of the classes. In Brown's argument, cf. \cite{brown}, this also follows directly from the long exact cohomology sequence. This argument doesn't quite work for the Chow--Witt rings because the long exact sequence for $\mathbf{K}^{\op{MW}}_n$-cohomology only has a short piece (three terms) with Chow--Witt groups, and we have little information on the remaining pieces of the $\mathbf{K}^{\op{MW}}_n$-cohomology of ${\op{B}}\op{Sp}_{2n}$. This problem already appears with Chow-rings. One way to deal with it, cf. \cite{RojasVistoli}, is to restrict to tori, to compute the images of the characteristic classes and to show independence there. This is basically the splitting principle. However, for our situation, we should not look at the maximal torus of $\op{Sp}_{2n}$ but at the ``quaternionic torus'' $\op{SL}_2^{\times n}\hookrightarrow\op{Sp}_{2n}$. There are  corresponding quaternionic projective spaces $\mathbb{HP}^n$ which are finite-dimensional approximations to ${\op{B}}\op{Sp}_2$. The quaternionic projective $n$-space $\mathbb{HP}^n$ can be defined as the open subscheme of those $2$-dimensional subspaces in a $2n+2$-dimensional symplectic vector space such that the restriction of the symplectic form is nondegenerate. Said differently there is an identification  $\mathbb{HP}^n\cong\op{Sp}_{2n+2}/\op{Sp}_2\times\op{Sp}_{2n}$. More information on the subtle geometry of quaternionic projective spaces can be found in \cite{PaninWalter}, which also discusses the symplectic splitting principle and the relevant flag schemes. Actually, to get the splitting principle, we can use the quaternionic projective bundle formula for trivial bundles, found in \cite[Theorem 8.1]{PaninWalter}. An alternative description of the motivic cell structure of $\mathbb{HP}^n$ which also allows to compute the cohomology and deduce the symplectic splitting principle is obtained in \cite{voelkel}.

\begin{proposition}
\label{prop:hpn}
Let $S$ be a smooth scheme. For any $n$, there are isomorphisms 
\[
\widetilde{\op{CH}}^\bullet(\mathbb{HP}^n\times S)\cong \widetilde{\op{CH}}^\bullet(S)[\op{p}_1]/(\op{p}_1^{n+1}),
\]
where $\op{p}_1$ is the Pontryagin class of the tautological rank 2 symplectic bundle over $\mathbb{HP}^n$. Stabilization to the limit $n\to\infty$ yields an isomorphism
\[
\widetilde{\op{CH}}^\bullet(S\times{\op{B}}\op{SL}_2)\cong \widetilde{\op{CH}}^\bullet(S)[\op{p}_1]. 
\]
\end{proposition}

\begin{proof}
Even though Chow--Witt groups do not exactly fit the definition of symplectically oriented theory in \cite{PaninWalter}, the argument in loc.cit. can be done verbatim for Chow--Witt groups. Note that the Pontryagin class $\op{p}_1$ depends on the choice of orientation of the tautological bundle on $\mathbb{HP}^\infty$. For the preferred choice of orientation, cf.~Remark~\ref{rem:axiomchoice}.

Computations with the localization sequence show that the inclusion $\widetilde{\op{CH}}^i(\mathbb{HP}^n\times S)\to \widetilde{\op{CH}}^i(\mathbb{HP}^{n+1}\times S)$ induces an isomorphism in degrees $i<2n$. The explicit localization situation, which presents $\mathbb{HP}^{n+1}\times S$ as homotopy pushout of a Hopf-type map $\mathbb{H}\op{S}^{n}\cong\op{Q}_{4n+3}\times S\to \mathbb{HP}^n\times S$, can be found in \cite[Sections 4.2, 4.4]{voelkel}. This implies that 
\[
\widetilde{\op{CH}}^\bullet(\mathbb{HP}^n\times S)\cong \widetilde{\op{CH}}^\bullet(S)[\zeta]/(\zeta^{n+1}).
\]

The identification $\mathbb{HP}^\infty\cong{\op{B}}\op{SL}_2$ can be seen as follows: as a first step, we identify $\mathbb{HP}^n\cong \op{Sp}_{2n+2}/\op{SL}_2\times\op{Sp}_{2n}$. In particular, we can write $\mathbb{HP}^n\cong (\op{Sp}_{2n+2}/\op{Sp}_{2n})/\op{SL}_2$. Now $\op{Sp}_{2n+2}/\op{Sp}_{2n}\cong\op{Q}_{4n+3}$ and there is a natural $\op{Sp}_{2n+2}$-equivariant vector bundle torsor $\op{Q}_{4n+3}\to\mathbb{A}^{2n+2}\setminus\{0\}$ where the target has the tautological $\op{Sp}_{2n+2}$-action. The induced action of $\op{SL}_2$ is free and induces a vector bundle torsor $\mathbb{HP}^n\to(\mathbb{A}^{2n+2}\setminus\{0\})/\op{SL}_2$. This shows that up to $\mathbb{A}^1$-weak equivalence, $\mathbb{HP}^n$ is a $2n$-acyclic approximation of ${\op{B}}\op{SL}_2$. 
\end{proof}

\begin{corollary}
\label{cor:sl2n}
For every $n$, there are canonical isomorphisms
\[
\widetilde{\op{CH}}^\bullet({\op{B}}\op{SL}_2^{\times n})\cong\op{GW}(F)[\op{p}_{1,1},\dots,\op{p}_{1,n}],
\]
where the class $\op{p}_{1,i}$ is the first Pontryagin class of the universal $\op{SL}_2$-bundle over the $i$-th factor.
\end{corollary}

Before establishing the symplectic splitting principle, we need an explicit geometric model for the Whitney sum map ${\op{B}}\op{Sp}_{2m}\times{\op{B}}\op{Sp}_{2n-2m}\to{\op{B}}\op{Sp}_{2n}$. For this, consider a slight modification from Proposition~\ref{prop:inducedhom}: let $V$ be an appropriate representation for $\op{Sp}_{2m}$ with closed subset $S_V$ of codimension $s$ and let $W$ be an appropriate representation for $\op{Sp}_{2n-2m}$ with closed subset $S_W$ of codimension $s$. Then $V\oplus W$ is naturally a representation for $\op{Sp}_{2m}\times\op{Sp}_{2n-2m}$ and we obtain a model for ${\op{B}}(\op{Sp}_{2m}\times\op{Sp}_{2n-2m})$ by 
$((V\setminus S_V)\times (W\setminus S_W))/(\op{Sp}_{2m}\times\op{Sp}_{2n-2m})$. 
If we now take the induced $\op{Sp}_{2n}$-representation $U:=\op{Ind}^{\op{Sp}_{2n}}_{\op{Sp}_{2m}\times\op{Sp}_{2n-2m}}(V\oplus W)$ and add a suitable $\op{Sp}_{2n}$-representation so that for the resulting representation  $U'$ the group $\op{Sp}_{2n}$ acts freely outside a closed subset of codimension $s$, we obtain a model for the Whitney sum map as
\[
\omega\colon ((V\setminus S_V)\times (W\setminus S_W))/(\op{Sp}_{2m}\times\op{Sp}_{2n-2m})\to (U'\setminus S_{U'})/\op{Sp}_{2n}.
\]
We have the $\op{Sp}_{2n}$-torsor $U'\setminus S_{U'}\to (U'\setminus S_{U'})/\op{Sp}_{2n}$. Pulling back this torsor along $\omega$ yields, by construction, the $\op{Sp}_{2n}$-torsor given by applying the inclusion of the block matrix to the transition functions of the two torsors $V\setminus S_V\to (V\setminus S_V)/\op{Sp}_{2m}$ and $W\setminus S_W\to (W\setminus S_W)/\op{Sp}_{2n-2m}$. Consider the vector bundle $E_{2n}(U')\to (U'\setminus S_{U'})/\op{Sp}_{2n}$ obtained as associated bundle for the torsor $U'\setminus S_{U'}\to(U'\setminus S_{U'})/\op{Sp}_{2n}$ with the tautological $2n$-dimensional representation of $\op{Sp}_{2n}$. By the above discussion, the pullback of the bundle $E_{2n}(U')\to (U'\setminus S_{U'})/\op{Sp}_{2n}$ along this map to $((V\setminus S_V)\times (W\setminus S_W))/(\op{Sp}_{2m}\times\op{Sp}_{2n-2m})$ yields the Whitney sum of the pullbacks of the vector bundles $E_{2m}(V)\to (V\setminus S_V)/\op{Sp}_{2m}$ and $E_{2n-2m}(W)\to(W\setminus S_W)/\op{Sp}_{2n-2m}$, respectively. 

The following result is our {\em weak splitting principle}. The actual splitting principle states that the map in question is an isomorphism,
which is proved as parts (1) and (2) of Theorem \ref{thm:symplectic} further below.

\begin{proposition}
\label{prop:symsplit}
The restriction of the top Pontryagin class $\op{p}_n$ along the orthogonal sum map ${\op{B}}(\op{Sp}_{2m}\times\op{Sp}_{2n-2m})\to{\op{B}}\op{Sp}_{2n}$ equals the exterior product $\op{p}_m\boxtimes\op{p}_{n-m}$. Consequently, the restriction morphism
\[
\widetilde{\op{CH}}^\bullet({\op{B}}\op{Sp}_{2n})\to \widetilde{\op{CH}}^\bullet({\op{B}}\op{SL}_2^{\times n})\cong \op{GW}(F)[\op{p}_{1,1},\dots,\op{p}_{1,n}]
\]
maps the Pontryagin class $\op{p}_i$ to the $i$-th elementary symmetric polynomial in the classes $\op{p}_{1,1},\dots,\op{p}_{1,n}$. 
\end{proposition}

\begin{proof}
We first prove the statement about the orthogonal sum map for the top Pontryagin class, i.e., the restriction of $\op{p}_n$ along ${\op{B}}(\op{Sp}_{2m}\times\op{Sp}_{2n-2m})\to {\op{B}}\op{Sp}_{2n}$ is equal to $\op{p}_m\boxtimes \op{p}_{n-m}$. We freely use the notation introduced in discussing the finite-dimensional model for the Whitney sum map. Note also that the model introduced for ${\op{B}}(\op{Sp}_{2m}\times\op{Sp}_{2n-2m})$ has a product structure, so that talking about the exterior product makes sense.

By Proposition~\ref{prop:spuniq}, the class $\op{p}_n$ is the Euler class of the bundle $\op{E}_{2n}(U')\to (U'\setminus S_{U'})/\op{Sp}_{2n}$. By functoriality of Euler classes discussed in Section~\ref{sec:prelims}, $\op{p}_n$ restricts to the Euler class of the Whitney sum of the vector bundles $E_{2m}(V)\to (V\setminus S_V)/\op{Sp}_{2m}$ and $E_{2n-2m}(W)\to(W\setminus S_W)/\op{Sp}_{2n-2m}$. 

It remains to show that the Euler class of the Whitney sum is the exterior product of the Euler classes of the summands. By definition, the Euler classes associated to $E_{2m}(V)$ and $E_{2n-2m}(W)$ are the top Pontryagin classes $\op{p}_m\in\widetilde{\op{CH}}^{2m}((V\setminus S_V)/\op{Sp}_{2m})$ and $\op{p}_{n-m}\in\widetilde{\op{CH}}^{2n-2m}((W\setminus S_W)/\op{Sp}_{2n-2m})$, respectively. Now Lemma~\ref{lem:eulermult} implies that the Euler class of the exterior product sum $E_{2m}(V)\boxtimes E_{2n-2m}(W)$ is the exterior product of the Euler classes of the summands, i.e., equals $\op{p}_m\boxtimes\op{p}_{n-m}$. Since the Euler classes considered here live in even degrees, there is no sign issue, cf.~Proposition~\ref{prop:cwepsgraded}.

Now we prove the second statement, the weak splitting principle, by induction. Assume we know the weak splitting principle for the Pontryagin classes for $\op{Sp}_{2n}$, we want to deduce it for $\op{Sp}_{2n+2}$. For the top Pontryagin class $\op{p}_{n+1}$, this follows from the multiplicativity statement already established and the inductive assumption for $\op{p}_{n}$: the class $\op{p}_n$ maps to the product $\op{p}_{1,1}\cdot\cdots\cdot\op{p}_{1,n}$, hence $\op{p}_{n+1}$ maps to $\op{p}_{1,1}\cdot\cdots\cdot\op{p}_{1,n+1}$. To deal with a Pontryagin class $\op{p}_i$ for $i<n+1$, recall that $\op{p}_i$ is defined for $\op{Sp}_{2n+2}$ by the fact that its restriction to $\op{Sp}_{2n}$ is $\op{p}_i$. The image of $\op{p}_i$ in $\op{GW}[\op{p}_{1,1},\dots,\op{p}_{1,n+1}]$ is a homogeneous polynomial of degree $2i$ and we need to determine the monomials. For each of the $\binom{n+1}{i}$ possibilities of choosing $i$ factors from $\op{SL}_2^{\times n+1}$, the corresponding inclusion $\op{SL}_2^{\times i}\hookrightarrow\op{SL}_2^{\times n+1}\hookrightarrow\op{Sp}_{2n+2}$ can be factored as $\op{SL}_2^{\times i}\hookrightarrow\op{Sp}_{2i}\hookrightarrow\op{Sp}_{2n+2}$. In this latter composition, $\op{p}_i(\op{Sp}_{2n+2})$ restricts to $\op{p}_i(\op{Sp}_{2i})$, which is the top Pontryagin class and consequently maps to the product $\prod\op{p}_{1,j}$ running over $j\in I$ with the appropriate index set. We find that each monomial of degree $2i$ in the $\op{p}_{1,j}$ appears with coefficient 1 in the restriction of $\op{p}_i(\op{Sp}_{2n+2})$. This means that $\op{p}_i(\op{Sp}_{2n+2})$ restricts to the elementary symmetric polynomial in the $\op{p}_{1,1},\dots,\op{p}_{1,n+1}$. 
\end{proof}

\begin{remark}
Note that we are not using a K\"unneth formula in the above proof. All the computations are done in Chow--Witt rings of finite-dimensional models of ${\op{B}}\op{SL}_2^{\times n}$ which are obtained as products of finite-dimensional quaternionic projective spaces. The exterior product of cycles in this situation corresponds to the polynomial ring multiplication under the isomorphism of Corollary~\ref{cor:sl2n}.
\end{remark}

Now we have all the information required to compute the Chow--Witt rings of the classifying spaces of the symplectic groups. The final argument is the same induction as in \cite{brown}, where we have already established the existence and uniqueness of the Pontryagin classes and have a weak form of the splitting principle to prove their independence. 

\begin{theorem}
\label{thm:symplectic}
\begin{enumerate}[(1)]
\item 
For every $n$, the Pontryagin classes $\op{p}_{i}\in\widetilde{\op{CH}}^{2i}({\op{B}}\op{Sp}_{2n})$ defined in Proposition~\ref{prop:spuniq} induce an isomorphism
\[
\widetilde{\op{CH}}^\bullet({\op{B}}\op{Sp}_{2n})\cong \op{GW}(F)[\op{p}_1,\dots,\op{p}_n].
\]
\item 
The restriction along ${\op{B}}\op{SL}_2^{\times n}\to{\op{B}}\op{Sp}_{2n}$ is injective on Chow--Witt rings. 
\item 
The restriction along the orthogonal sum  ${\op{B}}\op{Sp}_{2m}\times{\op{B}}\op{Sp}_{2n-2m}\to{\op{B}}\op{Sp}_{2n}$ maps the Pontryagin classes as follows:
\[
\op{p}_i\mapsto\sum^m_{j=i+m-n}\op{p}_j\boxtimes\op{p}_{i-j}.
\]
\item 
Under the canonical morphism $\widetilde{\op{CH}}^\bullet({\op{B}}\op{Sp}_{2n})\to \op{CH}^\bullet({\op{B}}\op{Sp}_{2n})$, the Pontryagin class $\op{p}_i$ maps to the Chern class $\op{c}_{2i}$. 
\end{enumerate}
\end{theorem}

\begin{proof}
(4) has already been proved in Corollary~\ref{cor:spuniq}. 

We prove (1) and (2) by induction on $n$.  The base case is $\op{Sp}_0$ in which case the claim reduces to $\widetilde{\op{CH}}^\bullet(\op{Spec}F)\cong\op{GW}(F)$, concentrated in degree $0$. We can alternatively use $\op{Sp}_2=\op{SL}_2$ as base case. Point (1) then follows directly from the quaternionic projective bundle theorem, cf.~Proposition~\ref{prop:hpn}. Note that this base case trivially satisfies the splitting principle in the sense that the restriction along ${\op{B}}\op{SL}_2^{\times n}\to{\op{B}}\op{Sp}_{2n}$ is injective. 

Now assume that we know statements (1) and (2) for the Chow--Witt ring of ${\op{B}}\op{Sp}_{2n-2}$. By the result on existence and uniqueness of the Pontryagin classes, cf. Proposition~\ref{prop:spuniq}, we have $\iota^\ast\op{p}_i(\op{Sp}_{2n})=\op{p}_i(\op{Sp}_{2n-2})$ for $\iota\colon {\op{B}}\op{Sp}_{2n-2}\to {\op{B}}\op{Sp}_{2n}$, hence the restriction map 
\[
\iota^\ast\colon \widetilde{\op{CH}}^\bullet({\op{B}}\op{Sp}_{2n})\to \widetilde{\op{CH}}^\bullet({\op{B}}\op{Sp}_{2n-2})
\]
is surjective. By Proposition~\ref{prop:spuniq},  $\op{p}_n(\op{Sp}_{2n})$ is defined as the Euler class of the universal symplectic bundle on ${\op{B}}\op{Sp}_{2n}$, and the localization sequence of Proposition~\ref{prop:sploc} implies that any class with trivial restriction to ${\op{B}}\op{Sp}_{2n-2}$ is divisible by the Euler class. In particular, $\op{p}_1,\dots,\op{p}_n$ generate the Chow--Witt ring. 

We now prove the splitting principle, point (2), for ${\op{B}}\op{Sp}_{2n}$, by induction on the cohomological degree $q$. For the base case, we want to show that $\widetilde{\op{CH}}^0({\op{B}}\op{Sp}_{2n})\to\widetilde{\op{CH}}^0({\op{B}}\op{SL}_{2}^{\times n})$ is injective. The structural morphism ${\op{B}}\op{Sp}_{2n}\to\op{Spec} F$ provides a ring homomorphism $\op{GW}(F)\to \widetilde{\op{CH}}^0({\op{B}}\op{Sp}_{2n})$. This homomorphism is injective because it is split by evaluation at the $F$-rational point $\op{Spec}F\to{\op{B}}\op{Sp}_{2n}$ which classifies the symplectic structure on $\mathbb{A}^{2n}_F$, viewed as trivial symplectic bundle. By the localization sequence of Proposition~\ref{prop:sploc}, the morphism $\widetilde{\op{CH}}^0({\op{B}}\op{Sp}_{2n})\to \widetilde{\op{CH}}^0({\op{B}}\op{Sp}_{2n-2})$ is injective. Inductively, we find that the morphism $\op{GW}(F)\to \widetilde{\op{CH}}^0({\op{B}}\op{Sp}_{2n})$ is an isomorphism for all $n$. A similar argument applies to ${\op{B}}\op{SL}_2^{\times n}$. This implies that the map $\widetilde{\op{CH}}^0({\op{B}}\op{Sp}_{2n})\to\widetilde{\op{CH}}^0({\op{B}}\op{SL}_{2}^{\times n})$ is an isomorphism of free $\op{GW}(F)$-modules of rank 1, compatible with the identifications with $\op{GW}(F)$ given by the structural morphisms.

For the inductive step, suppose we know that $\widetilde{\op{CH}}^i({\op{B}}\op{Sp}_{2n})\to\widetilde{\op{CH}}^i({\op{B}}\op{SL}_{2}^{\times n})$ is injective for all degrees $i<q$. Let 
\[
\alpha\in\ker\left(\widetilde{\op{CH}}^q({\op{B}}\op{Sp}_{2n})\to\widetilde{\op{CH}}^q({\op{B}}\op{SL}_{2}^{\times n})\right)
\]
be a class in the kernel of the restriction. By the outer inductive assumption, we know that $\widetilde{\op{CH}}^q({\op{B}}\op{Sp}_{2n-2})\to\widetilde{\op{CH}}^q({\op{B}}\op{SL}_{2}^{\times (n-1)})$ is injective. In particular, the restriction of the image $\overline{\alpha}\in\widetilde{\op{CH}}^q({\op{B}}\op{Sp}_{2n-2})$ to ${\op{B}}\op{SL}_2^{\times (n-1)}$ is trivial only if $\overline{\alpha}$ is trivial. Since $\widetilde{\op{CH}}^q({\op{B}}\op{Sp}_{2n-2})\cong\op{GW}(F)[\op{p}_1,\dots,\op{p}_{n-1}]$, we can lift $\overline{\alpha}$ to $f(\op{p}_1,\dots,\op{p}_{n-1})\in\widetilde{\op{CH}}^q({\op{B}}\op{Sp}_{2n})$. Then the difference $\alpha':=\alpha-f(\op{p}_1,\dots,\op{p}_{n-1})$ will have trivial image in $\widetilde{\op{CH}}^q({\op{B}}\op{Sp}_{2n-2})$, and will still be in the kernel of the restriction map. In particular, to establish the splitting principle in degree $q$, it suffices to show that $\alpha'=0$. Since the restriction of $\alpha'$ to ${\op{B}}\op{Sp}_{2n-2}$ is trivial, it is in the image of the multiplication by the Euler class
\[
\op{e}(\op{Sp}_{2n})\colon \widetilde{\op{CH}}^{q-2n}({\op{B}}\op{Sp}_{2n})\to \widetilde{\op{CH}}^{q}({\op{B}}\op{Sp}_{2n}),
\]
i.e., $\alpha'$ is divisible by the Euler class. In particular, there is a class $\alpha''$ with $\op{e}(\op{Sp}_{2n})\cdot\alpha''=\alpha'$. Under the restriction 
$\widetilde{\op{CH}}^q({\op{B}}\op{Sp}_{2n})\to\widetilde{\op{CH}}^q({\op{B}}\op{SL}_{2}^{\times n})$, this equality expresses the restriction of $\alpha'$ as the product of the restriction of the Euler class $\op{e}(\op{Sp}_{2n})=\op{p}_n$ and the restriction of $\alpha''$. By Proposition~\ref{prop:symsplit}, the Euler class $\op{p}_n$ restricts to the elementary symmetric polynomial in $\op{p}_{1,1},\dots\op{p}_{1,n}$ which is non-zero. In particular, multiplication with the Euler class restricts to an injection on $\widetilde{\op{CH}}^\bullet({\op{B}}\op{SL}_2^{\times n})$. If $\alpha'$ has trivial restriction, then $\alpha''$ must have trivial restriction. But then the inner inductive assumption implies that $\alpha''$ is already trivial, proving that $\alpha'$ is trivial.

We have thus established (2), the splitting principle for the Chow--Witt ring of ${\op{B}}\op{Sp}_{2n}$. From the splitting principle, it follows that multiplication with the Euler class is injective because the Euler class maps to a non-zero element of $\widetilde{\op{CH}}^\bullet({\op{B}}\op{SL}_2^{\times n})$, cf. Corollary~\ref{cor:sl2n}. The injectivity together with Proposition~\ref{prop:symsplit} implies that  $\op{p}_n(\op{Sp}_{2n})$ is algebraically independent of the other $\op{p}_i(\op{Sp}_{2n})$. This implies the required identification  $\widetilde{\op{CH}}^\bullet({\op{B}}\op{Sp}_{2n})\cong \op{GW}(F)[\op{p}_1,\dots,\op{p}_n]$, proving (1).

The Whitney sum formula (3) now follows directly from the splitting principle (2): to check the formula, it suffices to check it after restriction to ${\op{B}}\op{SL}_2^{\times n}$, but there it holds by Corollary~\ref{cor:sl2n}. 
\end{proof}

\begin{proposition}
\label{prop:sympconj}
Let $F$ be a perfect field of characteristic unequal to $2$, let $X$ be a smooth scheme and let $\mathscr{E}\to X$ be a symplectic vector bundle on $X$. There is a conjugate symplectic vector bundle, obtained by applying the conjugation automorphism $M\mapsto (M^{-1})^{\op{t}}$. The Pontryagin classes of a symplectic bundle and its conjugate are related as follows:
\[
\op{p}_i(\overline{\mathscr{E}})= \langle-1\rangle^i\op{p}_i(\mathscr{E}). 
\]
\end{proposition}

\begin{proof}
Viewing $\op{Sp}_{2n}$ as automorphism group of a rank $n$ module over the split quaternion algebra $\op{M}_{2\times 2}$, the conjugation corresponds to the conjugation of the matrix algebra. On the level of $\op{Sp}_{2}$-bundles, this conjugation will not actually change the isomorphism class of the bundle, but it will change the orientation. The induced map on the Chow--Witt ring $\widetilde{\op{CH}}^\bullet({\op{B}}\op{SL}_2)\cong\op{GW}(F)[\op{p}_1]$ is the $\op{GW}(F)$-algebra map given by sending $\op{p}_1\mapsto\langle-1\rangle\op{p}_1$. The general claim now follows from the splitting principle in Theorem~\ref{thm:symplectic} above.
\end{proof}

\subsection{Application: symplectic orientation for Chow--Witt theory}
\label{sec:sympor}
As an application of the above computation, we now explain in which sense Chow--Witt groups are a symplectically oriented theory. This is claimed in \cite[section 7]{PaninWalter} without proof, and the results of this section will make this statement precise.

First, we observe that $\widetilde{\op{CH}}^\bullet(-)$ does not satisfy the localization condition \cite[Definition 2.1 (1)]{PaninWalter} in the definition of a ring cohomology theory. However, it does when instead of only the geometric bidegrees $\widetilde{\op{CH}}^\bullet(-)$ we consider the full bigraded theory $\op{H}^\bullet(-,\mathbf{K}^{\op{MW}}_\bullet)$,  see \cite[10.4.7 and 10.4.9]{fasel:memoir}. We now check that it satisfies all other
axioms of \cite{PaninWalter} as well. The "\'etale excision" condition (2)  (which in fact is a Nisnevich descent condition) can be deduced from \cite{MField}, but as Panin--Walter say Zariski excision usually is enough in practice. Condition (3), that is algebraic homotopy invariance, is established in \cite[section 11]{fasel:memoir}. Concerning the ring structure of \cite[Definition 2.2]{PaninWalter},
everything but condition (2) follows from \cite{fasel:chowwitt}. Condition (2) could be established along the lines of our proof of Proposition~\ref{prop:gysinder}, see Remark \ref{ringextra}.

The bigraded theory $\op{H}^\bullet(-,\mathbf{K}^{\op{MW}}_\bullet)$ and in particular $\widetilde{\op{CH}}^\bullet(-)$ carries a {\em symplectic Thom structure} in the sense of \cite[Definition 7.1]{PaninWalter} by \cite[10.4.8]{fasel:memoir}, see also \cite[section 3.5]{AsokFaselEuler}. That is, for any rank $2n$ vector bundle $E \to X$ with a symplectic structure, we can push forward $\langle 1\rangle \in \op{GW}(F)$ (or rather its image under the morphism $\op{GW}(F)\to \widetilde{\op{CH}}^0(X)$) along the zero-section to a Thom class in $\widetilde{\op{CH}}^{2n}_X(E)$. 
In \cite[section 14]{PaninWalter} the authors use their axioms to prove that the existence of a symplectic Thom structure is equivalent to several other extra data on a ring cohomology theory. In particular, it is shown that it is equivalent to a {\em Pontryagin structure} \cite[Definition 12.1]{PaninWalter} and a {\em Pontryagin classes theory} \cite[Definition 14.1]{PaninWalter}. In short, their machinery produces Pontryagin classes for the bigraded $\op{H}^\bullet(-,\mathbf{K}^{\op{MW}}_\bullet)$.

Above, we have computed $\widetilde{\op{CH}}^{\bullet}(\op{B}\op{Sp}_{2n})$. For a rank $2n$ vector bundle $E\to X$ with a symplectic structure, we get a classifying map $X\to \op{B}\op{Sp}_{2n}$, cf. Remark~\ref{rem:classmap}, and the induced ring homomorphism $\widetilde{\op{CH}}^\bullet({\op{B}}\op{Sp}_{2n})\to\widetilde{\op{CH}}^\bullet(X)$ yields Pontryagin classes for the symplectic bundle $E\to X$. From our construction of Pontryagin classes above, we deduce immediately that the properties and formulas of \cite[Definition 12.1 and Definition 14.1]{PaninWalter} hold for the induced  Pontryagin classes in $\widetilde{\op{CH}}^\bullet(X)$. 

The above discussion applies verbatim to $\op{H}^\bullet(-,\mathbf{I}^\bullet)$. From Theorem~\ref{thm:symplectic}, we can deduce 
\[
\op{H}^\bullet({\op{B}}\op{Sp}_{2n},\mathbf{I}^\bullet)\cong \op{W}(F)[\op{p}_1,\dots,\op{p}_n],
\]
implying a theory of Pontryagin classes for the $\mathbf{I}^\bullet$-cohomology as well. This will be used for our computations of $\op{H}^\bullet({\op{B}}\op{SL}_n,\mathbf{I}^\bullet)$ in subsequent sections. 

\begin{remark}
The maximal compact subgroup of $\op{Sp}_n(\mathbb{R})$ is $\op{U}(n)$. Hence under real realization, the above formula becomes 
$\op{H}^\bullet({\op{B}}\op{U}(n)) \cong \mathbb{Z}[\op{c}_1^{\op{top}},...,\op{c}_n^{\op{top}}]$.
\end{remark}

We shortly compare our proof strategy to the techniques used in \cite{PaninWalter}. We already mentioned the fact that the localization sequence for Chow--Witt theory does not have the form required for the definition of cohomology theory in \cite{PaninWalter}. Another difference is that we are interested here in computing the value of a cohomology theory on ${\op{B}}\op{Sp}_{2n}$ where cohomology theory means ``something representable in the motivic homotopy category''. In \cite{PaninWalter}, the value of the cohomology theory on infinite-dimensional objects like ${\op{B}}\op{Sp}_{2n}$ is simply \emph{defined} as inverse limit of the cohomology evaluated on finite-dimensional approximation. In particular, while our computations would agree with those in \cite{PaninWalter} for finite-dimensional approximations to ${\op{B}}\op{Sp}_{2n}$, they would disagree in the final result: if we think of Chow--Witt theory as represented by the Eilenberg--Mac Lane spectrum for $\mathbf{K}^{\op{MW}}_\bullet$ then the value for ${\op{B}}\op{Sp}_{2n}$ is a polynomial ring in the Pontryagin classes, while the result in \cite{PaninWalter} is a power series ring. The difference is due to grading and stabilization, and is similar to what happens for cohomology vs. K-theory of classifying spaces. Nevertheless, the agreement on the finite-dimensional level implies that the Pontryagin classes we use in our computations can be identified with those obtained using the approach of  \cite{PaninWalter}, see Remark \ref{CHbig}.

Finally, there is also a difference of emphasis. For Panin and Walter (similar
to the standard intersection-theoretic construction of Chern classes for $\op{CH}$), the first Pontryagin class of a symplectic line bundle is the fundamental object, and the splitting principle is used to define higher Pontryagin classes in terms of elementary symmetric polynomials in the Pontryagin roots. In our approach, the top Pontryagin class (via its identification as Euler class) is the fundamental object, and the lower Pontryagin classes are uniquely defined by requiring compatibility with stabilization. The splitting principle is then required to show independence of the classes, compensating for the too short localization sequence.  That said, parts of the global strategy employed by us coincide with parts of \cite{PaninWalter}, which is in particular true for the splitting principle.

\section{Characteristic classes for (oriented) vector bundles}
\label{sec:character}

In this section, we will give definitions of the Chow--Witt-theoretic characteristic classes for vector bundles in $\widetilde{\op{CH}}^\bullet({\op{B}}\op{SL}_n)$ (most of which are well-known). Before we can do that, we introduce the geometric setup for the localization sequence, similar to what happened in Section~\ref{sec:symplectic}. This will, in particular, also allow to compare the classes considered here to others that already appeared in the literature. 

\subsection{Geometric setup of the localization sequence}

We begin by setting up the localization sequence for the inductive computation of $\widetilde{\op{CH}}^\bullet({\op{B}}\op{SL}_n)$, similar to the development in Section~\ref{sec:symplectic}. As for the latter, the basic idea on the level of classifying spaces in the Morel--Voevodsky $\mathbb{A}^1$-homotopy category is to consider the universal rank $n$ vector bundle  $E_n\to {\op{B}}\op{SL}_n$. The total space and zero-section can be identified with ${\op{B}}\op{SL}_n$ and the complement of the zero-section can be identified with ${\op{B}}\op{SL}_{n-1}$. The next paragraphs should make this precise on the level of finite-dimensional approximations of the classifying spaces.

Let $V$ be an $\op{SL}_n$-representation on which $\op{SL}_n$ acts freely outside a closed stable subset $Y$ of codimension $s$, and consider the quotient $X(V):=(V\setminus Y)/\op{SL}_n$. This computes $\widetilde{\op{CH}}^\bullet({\op{B}}\op{SL}_n)$ in degrees $\leq s-2$, and a finite-dimensional model for the universal $\op{SL}_n$-torsor is given by the projection $p\colon V\setminus Y\to X(V)$.

Now we may consider the tautological $\op{SL}_n$-representation on $\mathbb{A}^n$ and take the associated vector bundle for the $\op{SL}_n$-torsor $p\colon V\setminus Y\to X(V)$ for this representation; denote this vector bundle by $\gamma_V\colon E_n(V)\to X(V)$. Note that this vector bundle is orientable; if we choose an orientation on the vector space $\mathbb{A}^n$ underlying the tautological $\op{SL}_n$-representation, this will give us an orientation of the vector bundle $E_n(V)\to X(V)$. The vector bundle $\gamma_V\colon E_n(V)\to X(V)$ is an $\mathbb{A}^1$-weak equivalence and therefore we have $\widetilde{\op{CH}}^\bullet(X(V))\cong\widetilde{\op{CH}}^\bullet(E_n(V))$. Both these are isomorphic to $\widetilde{\op{CH}}^i({\op{B}}\op{SL}_n)$ for $i\leq s-2$. 

Denote by $S_n(V)$ the complement of the zero-section of $\gamma_V\colon E(V)\to X(V)$. We want to identify $S_n(V)$ as an approximation of ${\op{B}}\op{SL}_{n-1}$. The group $\op{SL}_{n-1}$ acts freely on $V\setminus Y$, and the quotient morphism $q\colon (V\setminus Y)/\op{SL}_{n-1}\to X(V)$ is a model for the stabilization morphism ${\op{B}}\op{SL}_{n-1}\to {\op{B}}\op{SL}_n$. Note that the quotient morphism $q$ is in fact a fiber bundle with fiber $\op{SL}_n/\op{SL}_{n-1}$. 

To identify $S_n(V)$, consider the $\op{SL}_n$-operation on $\mathbb{A}^n$ via the tautological representation and denote by $v_1,\dots,v_n$ the canonical basis of $\mathbb{A}^n$. There are two $\op{SL}_n$-orbits, the origin $\{0\}$ and its complement $U=\mathbb{A}^n\setminus\{0\}$. The stabilizer of $v_n$ is an extension of $\op{SL}_{n-1}$ (acting on $\op{span}(v_1,\dots,v_{n-1})$) by a unipotent group and the inclusion $\op{SL}_{n-1}\hookrightarrow\op{SL}_n$ is the standard stabilization embedding. This implies that the morphism $(V\setminus Y)/\op{SL}_{n-1}\to X(V)$ lifts to a map $(V\setminus Y)/\op{SL}_{n-1}\to S_n(V)$ and this lift is a torsor under the unipotent radical of the stabilizer of $v_n$. In particular, we get an $\mathbb{A}^1$-weak equivalence $(V\setminus Y)/\op{SL}_{n-1}\to S_{n}(V)$ which therefore induces an isomorphism of Chow--Witt rings. In particular, the open complement $S_n(V)$ computes the Chow--Witt ring of ${\op{B}}\op{SL}_{n-1}$ in a range and, by construction, the morphism
\[
\widetilde{\op{CH}}^\bullet(X(V))\xrightarrow{\gamma_V^\ast} \widetilde{\op{CH}}^\bullet(S_n(V))\cong \widetilde{\op{CH}}^\bullet((V\setminus Y)/\op{SL}_{n-1})
\]
induced from the natural quotient map $(V\setminus Y)/\op{SL}_{n-1}\to X(V)$ is the stabilization morphism. Consequently, keeping in mind Remark~\ref{rem:axiomchoice} and passing to the limit over a suitable system of representations $V$, we obtain the following localization sequence: 

\begin{proposition}
\label{prop:locsln}
There is a long exact sequence of Chow--Witt groups of classifying spaces
\begin{eqnarray*}
\cdots&\to&\widetilde{\op{CH}}^{q-n}({\op{B}}\op{SL}_{n})\to \widetilde{\op{CH}}^q({\op{B}}\op{SL}_{n})\to \widetilde{\op{CH}}^q({\op{B}}\op{SL}_{n-1})\\&\to&
\op{H}^{q+1-n}_{\op{Nis}}({\op{B}}\op{SL}_{n},\mathbf{K}^{\op{MW}}_{q-n})\to \op{H}^{q+1}_{\op{Nis}}({\op{B}}\op{SL}_{n},\mathbf{K}^{\op{MW}}_q)\to\cdots
\end{eqnarray*}
The first map is the composition of the d\'evissage isomorphism with the forgetting of support, alternatively ``multiplication with the Euler class of the universal  bundle $E_{n}\to{\op{B}}\op{SL}_{n}$''. The second map is the restriction along the stabilization inclusion $\op{SL}_{n-1}\to\op{SL}_{n}$. 

There are similar exact sequences for the other coefficients $\mathbf{I}^\bullet$ and $\mathbf{K}^{\op{M}}_\bullet$, and the change-of-coefficient maps induce commutative ladders of exact sequences. 
\end{proposition}

\begin{proof}
The localization sequences follow from the discussion prior to the statement. The compatibility of the exact sequences with change-of-coefficients follows from the discussion in Section~\ref{sec:key}. Essentially, changing the coefficients is a map of Milnor--Witt cycle modules \cite{feld}, which are an oriented version of Rost's cycle modules \cite{rost}.
\end{proof}

On the level of finite-dimensional models, this arises from the localization situation for the vector bundles $E_n(V)$ with the zero-section $X(V)$ and its open complement $S_n(V)$. It should be pointed out that, just as in Proposition~\ref{prop:sploc}, the morphism $\widetilde{\op{CH}}^{q-n}({\op{B}}\op{SL}_n)\to \widetilde{\op{CH}}^{q}({\op{B}}\op{SL}_n)$, using the d\'evissage isomorphism, depends on the choice of orientation of the universal vector bundle on ${\op{B}}\op{SL}_n$.

An immediate application of the localization sequences in Chow theory is that 
Chern classes are uniquely determined by their compatibility with stabilization and the identification of the top Chern class as Thom class of the universal bundle. Certainly this is well-known, but most authors proceed from the definition of the first characteristic class to a definition of the higher ones via the splitting principle. 

\begin{proposition}
\label{prop:chern}
There are unique classes $\op{c}_i(\op{SL}_n)\in \op{CH}^i({\op{B}}\op{SL}_n)$ for $2\leq i\leq n$, such that the natural stabilization morphism $\iota\colon {\op{B}}\op{SL}_{n-1}\to{\op{B}}\op{SL}_n$ satisfies $\iota^\ast \op{c}_i(\op{SL}_n)=\op{c}_i(\op{SL}_{n-1})$ for $i<n$ and $\op{c}_n(\op{SL}_n)=\op{e}_n(\op{SL}_n)$. These agree with the Chern classes in \cite{RojasVistoli}. In particular, the Chow--Witt-theoretic Euler class reduces to the top Chern class in the Chow theory. 
There is a natural isomorphism
\[
\op{CH}^\bullet({\op{B}}\op{SL}_n)\cong\mathbb{Z}[\op{c}_2,\op{c}_3,\dots,\op{c}_n]. 
\]
The restriction along the Whitney sum ${\op{B}}\op{SL}_m\times{\op{B}}\op{SL}_{n-m}\to{\op{B}}\op{SL}_n$ maps the Chern classes as follows:
\[
\op{c}_i\mapsto \sum_{j=i+m-n}^m\op{c}_j\boxtimes\op{c}_{i-j}.
\]
\end{proposition}

\begin{proof}
The argument is the same as in  Proposition~\ref{prop:spuniq}, with the difference that in Chow theory the restriction to an open subset is surjective, thus improving the connectivity bound by 1. The identification with the Chern classes in \cite{RojasVistoli} follows since they satisfy stabilization and the top Chern class is given by the Euler class of the universal bundle. The Whitney sum formula can be deduced from classical sources or be proved along the line of argument in Section~\ref{sec:symplectic}. 
\end{proof}

\begin{proposition}
\label{prop:vbconjugate}
Let $\mathscr{E}\to X$ be a vector bundle on a smooth scheme $X$. Recall that the dual vector bundle $\overline{\mathscr{E}}$ is obtained by applying the automorphism $\op{GL}_n\to\op{GL}_n\colon M\mapsto (M^{-1})^{\op{t}}$ to the transition functions of $\mathscr{E}$. The Chern classes of a vector bundle and its dual are related as follows: 
\[
\op{c}_i(\overline{\mathscr{E}})=(-1)^i\op{c}_i(\mathscr{E}).
\]
\end{proposition}

\begin{proof}
For line bundles, this is clear since the transposed inverse is simply the inverse, and this induces the additive inverse in the Picard group. The general case follows from the classical splitting principle for Grassmannians. 
\end{proof}

The following recovers the well-known description of the mod 2 Chow ring of ${\op{B}}\op{SL}_n$. We include it so as to be able to identify the integral Stiefel--Whitney classes appearing later with the ones defined in \cite{fasel:ij}. 

\begin{proposition}
\label{prop:sw}
There are unique classes $\overline{\op{c}}_i(\op{SL}_n)\in \op{Ch}^i({\op{B}}\op{SL}_n)$ for $2\leq i\leq n$, such that the natural stabilization morphism $\iota\colon {\op{B}}\op{SL}_{n-1}\to{\op{B}}\op{SL}_n$ satisfies $\iota^\ast \overline{\op{c}}_i(\op{SL}_n)=\overline{\op{c}}_i(\op{SL}_{n-1})$ for $i<n$ and $\overline{\op{c}}_n(\op{SL}_n)=\op{e}_n(\op{SL}_n)$. These agree with the Stiefel--Whitney classes in \cite[Definition 4.2]{fasel:ij}. There is a natural isomorphism
\[
\op{Ch}^\bullet({\op{B}}\op{SL}_n)\cong \mathbb{Z}/2\mathbb{Z}[\overline{\op{c}}_2,\dots,\overline{\op{c}}_n].
\]
\end{proposition}

\begin{proof}
Again, the existence and uniqueness proof follows the argument for Proposition~\ref{prop:spuniq}. The identification with the Stiefel--Whitney classes in \cite{fasel:ij} follows since the top Stiefel--Whitney class is the reduction of the Euler class, cf.~\cite[Section 3, 4]{fasel:ij}, and compatibility with stabilization follows from the Whitney sum formula in loc.cit. 
\end{proof}

\begin{remark}
We follow \cite{fasel:ij} in denoting the Stiefel--Whitney classes as reductions of Chern classes. In the topological context of $\op{H}^\bullet({\op{B}}\op{U}(n),\mathbb{Z})$, the mod 2 reduction of $\op{c}_i$ is the class $\op{w}_{2i}$ of the universal bundle. In particular, for those readers who want to compare our development to the computation of $\op{H}^\bullet({\op{B}}\op{SO}(n),\mathbb{Z})$ in \cite{brown}, our classes $\overline{\op{c}}_i$ play exactly the role played by the $\op{w}_{i}$ in loc.cit.
\end{remark}

\subsection{Characteristic classes for (oriented) vector bundles}
Now we can define the relevant characteristic classes of vector bundles. 

We begin with a discussion of the Pontryagin classes, which are obtained by symplectification. This is the algebraic analogue of the classical definition which defines Pontryagin classes for real bundles by complexification,
compare also \cite[Definition 7]{ananyevskiy}. 

Recall that for any $n\in \mathbb{N}$, we have natural group homomorphisms
\[
\sigma_n\colon \op{GL}_n\to\op{Sp}_{2n}\colon  M\mapsto \left(\begin{array}{cc}
M & 0 \\ 0 & (M^{-1})^{\op{t}}\end{array}\right).
\]
These homomorphisms are compatible with stabilization in the sense that there are commutative diagrams 
\[
\xymatrix{
\op{GL}_n \ar[r]^{\sigma_n} \ar[d]_{(-)\oplus 1} & \op{Sp}_{2n} \ar[d]^{(-)\oplus\mathbb{H}} \\ 
\op{GL}_{n+1} \ar[r]_{\sigma_{n+1}} & \op{Sp}_{2n+2}.
}
\]
The induced morphism on classifying spaces
\[
{\op{B}}\sigma_n\colon {\op{B}}\op{GL}_n\to{\op{B}}\op{Sp}_{2n}
\]
will be called \emph{symplectification morphism}. It corresponds to the standard hyperbolic functor. More precisely, it maps a vector bundle $\mathscr{E}$ over a scheme $X$ to the vector bundle $\mathscr{E}\oplus \overline{\mathscr{E}}$ equipped with the natural symplectic form, where $\overline{\mathscr{E}}$ is the dual bundle of $\mathscr{E}$ (obtained by applying the outer automorphism $M\mapsto (M^{-1})^{\op{t}}$ to the transition functions). By the above commutative diagram, symplectification of vector bundles is compatible with stabilization.

Composing with the natural inclusions $\op{SL}_n\hookrightarrow\op{GL}_n$ yields by definition the symplectification for oriented vector bundles, which will be (isomorphic to) the symplectification of the underlying vector bundle
enjoying similar properties. Next, we introduce Pontryagin classes for ${\op{B}}\op{SL}_{n}$. 

\begin{definition}
\label{def:pontryagingln}
We define the \emph{Pontryagin classes} $\op{p}_{i}(\op{GL}_n)$ resp. $\op{p}_{i}(\op{SL}_n)$  as the images of $\op{p}_{i}\in\widetilde{\op{CH}}^\bullet({\op{B}}\op{Sp}_{2n})$ from Theorem~\ref{thm:symplectic} under the homomorphisms
\[
\widetilde{\op{CH}}^\bullet({\op{B}}\op{Sp}_{2n}) \xrightarrow{{\op{B}}\sigma_n^\ast} 
\widetilde{\op{CH}}^\bullet({\op{B}}\op{GL}_{n})\to 
\widetilde{\op{CH}}^\bullet({\op{B}}\op{SL}_{n})
\]
\end{definition}

\begin{remark}
There is a difference to the classical definition,
and also to the one of \cite{ananyevskiy}. For us, the Pontryagin classes are characteristic classes for the symplectic groups, and the Pontryagin classes in the Chow--Witt ring of ${\op{B}}\op{GL}_n$ are the classes of a particular symplectic representation of $\op{GL}_n$. This is why we do not modify the Pontryagin classes by signs or reindexing. In particular, $\op{p}_{i}\in\widetilde{\op{CH}}^{2i}({\op{B}}\op{GL}_n)$. Note also that we include all the Pontryagin classes in the definition where the classical definition only includes the even classes. In Chow--Witt theory, the odd Pontryagin classes are non-torsion classes, as can be seen from Proposition~\ref{prop:pontchern}. In $\mathbf{I}$-cohomology, however, the odd Pontryagin classes will have torsion images. 
\end{remark}

\begin{proposition}
\label{prop:pontryaginstab}
The Pontryagin classes for the special linear group are compatible with stabilization in the sense that $j^\ast(\op{p}_i(\op{SL}_{n+1}))=\op{p}_i(\op{SL}_{n})$ for $i \leq n$ where $j^\ast\colon \widetilde{\op{CH}}^\bullet({\op{B}}\op{SL}_{n+1})\to \widetilde{\op{CH}}^\bullet({\op{B}}\op{SL}_n)$ is induced from the natural stabilization map $\op{SL}_n\to\op{SL}_{n+1}$. 
\end{proposition}

\begin{proof}
This follows from the fact that symplectification commutes with stabilization and Proposition~\ref{prop:inducedhom}.  
\end{proof}

We recalled in Section~\ref{sec:prelims} the definition of the Chow--Witt-theoretic Euler class of vector bundles from \cite{AsokFaselEuler}. This gives rise to a class in the Chow--Witt ring of ${\op{B}}\op{SL}_n$. The following is the definition given in \cite[Definition 13.2.1]{fasel:memoir}, see also the discussion in \cite[Section 3]{AsokFaselEuler}. 

\begin{definition}
\label{def:chweuler}
For smooth schemes $X$, the Chow--Witt-theoretic Euler class of a vector bundle $p\colon \mathscr{E}\to X$ of rank $n$ is defined via the formula
\[
\op{e}_n(p\colon \mathscr{E}\to X):=(p^\ast)^{-1}{s_0}_\ast(1)\in\widetilde{\op{CH}}^n(X,\det(p)^\vee),
\]
where $s_0\colon X\to \mathscr{E}$ is the zero section. 
Using smooth finite-dimensional approximations to the classifying space ${\op{B}}\op{GL}_n$ provides a well-defined Euler class 
\[
\op{e}_n\in \widetilde{\op{CH}}^n({\op{B}}\op{GL}_n,\det(\gamma_n)^\vee).
\]
The same definition applies to the group $\op{SL}_n$. By compatibility of Euler classes with pullback, the Euler class for oriented vector bundles can alternatively be defined as the image of the above $\op{e}_n$ under the induced homomorphism 
\[
\widetilde{\op{CH}}^n({\op{B}}\op{GL}_n,\det(\gamma_n)^\vee)\to \widetilde{\op{CH}}^n({\op{B}}\op{SL}_n). 
\]
\end{definition} 

By definition, in the model for the localization sequence, cf.~Proposition~\ref{prop:locsln}, the Euler class corresponds under the d\'evissage isomorphism to the Thom class for the universal rank $n$ vector bundle $\gamma_n$ on ${\op{B}}\op{GL}_n$. This 
justifies calling the composition
\[
\widetilde{\op{CH}}^{q-n}({\op{B}}\op{SL}_n) \cong \widetilde{\op{CH}}^q_{{\op{B}}\op{SL}_n}(E_n)\to \widetilde{\op{CH}}^q(E_n)\cong \widetilde{\op{CH}}^q({\op{B}}\op{SL}_n)
\]
``multiplication with the Euler class''. Via the correspondence of Proposition~\ref{prop:characteristic}, cf. also Remark~\ref{rem:classmap}, the Euler class $\op{e}_n\in\widetilde{\op{CH}}^n({\op{B}}\op{SL}_n)$ is the characteristic class which induces the Chow--Witt-theoretic Euler classes of rank $n$ vector bundles by pullback along the classifying map. As before, the Euler class depends on the choice of orientation for the universal bundle.
 
\begin{remark}
\label{rem:inclasses}
We have defined the Euler class and the Pontryagin classes as elements in $\widetilde{\op{CH}}^\bullet({\op{B}}\op{SL}_n)$. However, we can also get corresponding classes in $\mathbf{I}^j$-cohomology by taking their images under projection
\[
\widetilde{\op{CH}}^\bullet({\op{B}}\op{SL}_n)\to \op{H}^\bullet_{\op{Nis}}({\op{B}}\op{SL}_n,\mathbf{I}^\bullet). 
\]
These classes will be also be denoted by the symbols $\op{p}_i$ and $\op{e}_n$, respectively.
\end{remark}

There are more induced characteristic classes which naturally live in the $\mathbf{I}^j$-cohomology of the classifying spaces. In the real realization, they arise as integral Bocksteins of products of Stiefel--Whitney classes. We proceed similarly,
using the integral Bockstein operations of Section~\ref{sec:cohops}. The classes $\beta(\overline{\op{c}}_{2i})$ are the integral Stiefel--Whitney classes from  \cite{fasel:ij} by Proposition~\ref{prop:sw}. The following definition only uses the reductions of even Chern classes since the odd Chern classes have trivial Bockstein. This will be discussed in more detail in the proof of Theorem~\ref{thm:slnchw}.

\begin{definition}
\label{def:bockchern}
For a set $J=\{j_1,\dots,j_l\}$ of integers $0<j_1<\cdots<j_l\leq [(n-1)/2]$, there is a class 
\[
\beta_J:=\beta(\overline{\op{c}}_{2j_1}\overline{\op{c}}_{2j_2}\cdots\overline{\op{c}}_{2j_l})\in \op{H}^{d+1}_{\op{Nis}}({\op{B}}\op{SL}_n,\mathbf{I}^{d+1})
\]
where $d=\sum_{a=1}^l 2j_a$. 
\end{definition}

\begin{remark}
The Bockstein class $\beta(\emptyset)=0$ because by convention the empty product equals $1\in\op{Ch}^0(X)$ which is the reduction of $1\in\op{H}^0(X,\mathbf{I}^0)$.
\end{remark}

\section{\texorpdfstring{$\op{H}^{\bullet}_{\op{Nis}}({\op{B}}\op{SL}_n,\mathbf{I}^\bullet)$}{H(BSLn,I)}: statement of the theorem}
\label{sec:statethm}

In this section, we construct graded rings which we  will show later to be isomorphic to the Chow--Witt rings of the classifying spaces ${\op{B}}\op{SL}_n$. More precisely, we will formulate the main results describing the $\mathbf{I}^j$-cohomology of these classifying spaces and deduce consequences for the computation of the Chow--Witt rings of ${\op{B}}\op{SL}_n$. The proof of the main theorem about $\mathbf{I}^j$-cohomology is deferred to later sections.

\subsection{Candidate presentations of cohomology rings}

In the following subsection, we define some $\op{W}(F)$-algebras. These are the motivic translation of the presentation of $\op{H}^\bullet({\op{B}}\op{SO}(n),\mathbb{Z})$ in \cite[Theorem 1.4]{brown}. As $n$ varies,
we also define restriction maps between the rings and check that they are well-defined. This will enable us to state the main results on $\widetilde{\op{CH}}^{\bullet}({\op{B}}\op{SL}_n)$
and $\op{H}^{\bullet}_{\op{Nis}}(\op{B}\op{SL}_n,\mathbf{I}^\bullet)$ in the next subsection.  

Following \cite{cadek}, we use the notation $\Delta(J,J')=(J\cup J')\setminus (J\cap J')$ for the symmetric difference of two subsets $J$ and $J'$ of a given set. This notation will make the formula for products of Bockstein classes easier to parse. 

Recall from Proposition~\ref{prop:cwepsgraded} that $\epsilon \in \op{GW}(F)$ becomes $1$ in $\op{W}(F)$ and hence $\op{H}^{\bullet}_{\op{Nis}}({\op{B}}\op{SL}_n,\mathbf{I}^\bullet)$ will be
a $(-1)$-graded commutative (graded commutative for short) algebra in the sense of Definition~\ref{def:epsgraded}. 

\begin{definition}
\label{def:rnsln}
Let $F$ be a field of characteristic $\neq 2$.  For a natural number $n\geq 2$, we define the graded commutative $\op{W}(F)$-algebra 
\[
\mathscr{R}_n=\op{W}(F)\left[P_1,\dots,P_{[(n-1)/2]},X_n, \{B_J\}_{J}\right]
\]
where the index set $J$ runs through the sets $\{j_1,\dots,j_l\}$ of natural numbers with $0<j_1<\dots<j_l\leq[(n-1)/2]$. The degrees of elements are given as follows:
\[
\deg P_i=4i, \; \deg X_n=n, \; \deg B_J=1+\sum_{i=1}^l 2j_i.
\]
By convention $B_\emptyset=0$. Let $\mathscr{I}_n\subset \mathscr{R}_n$ be the ideal generated by the following relations:
\begin{enumerate}
\item $\op{I}(F) B_J=0$.
\item If $n=2k+1$ is odd, $X_{2k+1}=B_{\{k\}}$. 
\item For two index sets $J$ and $J'$, we have 
\[
B_J\cdot B_{J'}=\sum_{k\in J} B_{\{k\}}\cdot P_{(J\setminus\{k\})\cap J'}\cdot B_{\Delta(J\setminus\{k\},J')}
\]
where we set $P_A=\prod_{i=1}^l P_{a_i}$ for $A=\{a_1,\dots,a_l\}$.
\end{enumerate}
\end{definition}

\begin{remark}
The convention of setting $B_{\emptyset}=0$ will later turn out to be compatible with the statement that $\overline{\op{c}}_0=1$ has trivial Bockstein.
\end{remark}

\begin{remark}
In a graded-commutative $\mathbb{Z}$-algebra, if an element $x$ is $2$-torsion, then $x\cdot y=y\cdot x$, irrespective of degrees. This follows since the products will be $2$-torsion and consequently $x\cdot y=-x\cdot y$. In the above setting, the elements $B_J$ are $2$-torsion because of relation (1) in Definition~\ref{def:rnsln} and $\op{W}(F)/\op{I}(F)=\mathbb{Z}/2\mathbb{Z}$. Therefore the ideal generated by the elements $B_J$ in $\mathscr{R}_n/\mathscr{I}_n$ will be central. Moreover, the subalgebra of even-degree elements will always be commutative. Since the elements $P_i$ live in degrees divisible by $4$, they are necessarily central, and so is the element $X_n$ for $n$ even. However, for $n$ odd it is also central since $X_{2k+1}=B_{\{k\}}$. 
Consequently, the $\mathbb{Z}$-graded $\op{W}(F)$-algebra $\mathscr{R}_n/\mathscr{I}_n$ defined above is actually commutative.
\end{remark}

\begin{definition}
\label{def:rnres}
Let $n\geq 3$ be a natural number. Define the map $\Phi_n\colon \mathscr{R}_n\to\mathscr{R}_{n-1}$ as the  $W(F)$-algebra map  explicitly given on generators as follows: 
\begin{itemize}
\item the element $P_i$ maps to $P_i$ if $i<(n-1)/2$ and maps to $X_{n-1}^2$ if $i=(n-1)/2$, 
\item the element $X_n$ maps to $0$, and finally
\item the element $B_J$ for the index set $J=\{j_1,\dots,j_l\}$ maps to
\[
\left\{\begin{array}{ll}
B_J & j_l<(n-1)/2\\
B_{J'}\cdot X_{n-1} & j_l=(n-1)/2, J=J'\sqcup\{j_l\}.
\end{array}\right.
\]
\end{itemize}
\end{definition}

\begin{example}
The first non-trivial relation of type (3) in Definition~\ref{def:rnsln} appears for $n=5$, it is
\[
B_{\{1,2\}}\cdot B_{\{1,2\}}= B_1\cdot P_2\cdot B_1+B_2\cdot P_1\cdot B_2.
\]
We check in this special case that the map $\Phi_5$ is compatible with this relation. Recall that $\Phi_5$ is given by $B_{\{1,2\}}\mapsto B_1\cdot X_4$,  $B_1\mapsto B_1$ and $B_2\mapsto 0$. The equation above then maps to 
\[
B_1\cdot X_4\cdot B_1\cdot X_4=B_1\cdot X_4^2\cdot B_1+0.
\]
We see that in this case, $\Phi_5$ maps the relation in $\mathscr{R}_5$ to an equality in $\mathscr{R}_4$ and all's well. Keep this example in mind, and the following proof is a cakewalk.
\end{example}

\begin{proposition}
\label{prop:compatres}
With the notation from Definitions~\ref{def:rnsln} and \ref{def:rnres}, we have 
\[
\Phi_n(\mathscr{I}_n)\subseteq \mathscr{I}_{n-1}.
\]
In particular, the map $\Phi_n$ descends to a well-defined ring homomorphism
\[
\overline{\Phi}_n\colon \mathscr{R}_n/\mathscr{I}_n\to \mathscr{R}_{n-1}/\mathscr{I}_{n-1}. 
\]
\end{proposition}

\begin{proof}
We first deal with the relations of type (1). Recall that the map $\Phi_n$ is by definition $W(F)$-linear, in particular it will send $I(F)$ into $I(F)$. Since $\Phi_n$ sends $B_J$ to either $B_J$ or $B_{J'}\cdot X_{n-1}$ (with special case $0$ if $J'=\emptyset$), it is clear that relations of type (1) are preserved. 

The relations of type (2) are also preserved, since both $X_{2k+1}$ and $B_k$ are mapped to $0$ by $\Phi_n$. 

It remains to deal with relations of type (3). These relations are trivially preserved if neither $J$ nor $J'$ contain the highest possible index $j_l=(n-1)/2$. In that case, all the relevant $B_J$ and $P_J$ will exist both in $\mathscr{R}_n$ and $\mathscr{R}_{n-1}$, and the corresponding relation in $\mathscr{R}_n$ is just mapped to the same relation in $\mathscr{R}_{n-1}$. 

The appearance of $j_l=(n-1)/2$ causes some trouble which we deal with in a case distinction. 

First assume that $j_l\in J'$ and $j_l\not\in J$. In this case, the relation is 
\[
B_J\cdot B_{J'}=\sum_{k\in J} B_{\{k\}}\cdot P_{(J\setminus\{k\})\cap J'}\cdot B_{\Delta(J\setminus\{k\},J')}.
\]
By assumption, $(J\setminus \{k\})\cap J'$ will not contain $j_l$ hence no Pontryagin class appears. Also, the assumption implies that $j_l\in\Delta(J\setminus\{k\},J')$ hence the $j_l$ appears in the second index set for every $k\in J$ (and none of the $k$ can be equal $j_l$). Therefore, the equality maps under $\Phi_n$ to 
\[
B_J\cdot B_{J'\setminus\{j_l\}}\cdot X_{n-1}=\sum_{k\in J} B_{\{k\}}\cdot P_{(J\setminus\{k\})\cap J'} \cdot B_{\Delta(J\setminus\{k\},J')\setminus\{j_l\}}\cdot X_{n-1}.
\]
But this is just the product of a relation in $\mathscr{I}_{n-1}$ with $X_{n-1}$. 

Since the relation is not symmetric in $J$ and $J'$, it is necessary to say a few words about the case $j_l\in J$ and $j_l\not\in J'$ as well. Now there will be one summand where $k=j_l$, and this summand maps to $0$ because it contains $B_{\{j_l\}}$. For all other summands, $k\neq j_l$. As before, the index $j_l$ doesn't appear in the set $(J\setminus \{k\})\cap J'$ relevant for the Pontryagin classes. Also as before, the symmetric differences $\Delta(J\setminus\{k\},J')$ for $k\neq j_l$ will all contain $j_l$. With this information, we write out the image of the relation under $\Phi_n$:
\[
B_{J\setminus\{j_l\}}\cdot X_{n-1}\cdot B_{J'}=\sum_{j_l\neq k\in J} B_{\{k\}}\cdot P_{(J\setminus\{k\})\cap J'}\cdot B_{\Delta(J\setminus\{k,j_l\},J')}\cdot X_{n-1}.
\]
Again, this is just the product of a relation in $\mathscr{I}_{n-1}$ with $X_{n-1}$. 

Finally, we deal with the case where $j_l\in J\cap J'$. In this case, the symmetric difference $\Delta(J\setminus \{k\},J')$ will not contain $j_l$ unless $j_l=k$ (which contributes a summand $0$ under $\Phi_n$). However, now the index set $(J\setminus \{k\})\cap J'$ will contain $j_l$ unless $j_l=k$. Therefore, in this case, the image of the relation under $\Phi_n$ is
\[
B_{J\setminus\{j_l\}}\cdot X_{n-1}\cdot B_{J'\setminus\{j_l\}}\cdot X_{n-1}= \sum_{j_l\neq k\in J} B_{\{k\}}\cdot P_{((J\setminus\{k\})\cap J')\setminus\{j_l\}}\cdot X_{n-1}^2\cdot B_{\Delta(J\setminus\{k\},J')}.
\]
Again, this is the product of a relation in $\mathscr{I}_{n-1}$ with $X_{n-1}^2$. 

Since $\Phi_n(\mathscr{I}_n)\subset \mathscr{I}_{n-1}$, it follows that the restriction map descends to a $W(F)$-algebra map $\overline{\Phi}_n\colon \mathscr{R}_n/\mathscr{I}_n\to \mathscr{R}_{n-1}/\mathscr{I}_{n-1}$ as claimed. 
\end{proof}

\begin{lemma}
\label{lem:evenpoly}
If $n$ is even, then we have an isomorphism 
\[
\mathscr{R}_n/\mathscr{I}_n\cong\mathscr{R}_{n-1}/\mathscr{I}_{n-1}[X_n].
\] 
In particular, the restriction map $\overline{\Phi}_n\colon \mathscr{R}_n/\mathscr{I}_n\to\mathscr{R}_{n-1}/\mathscr{I}_{n-1}$ is surjective. 
\end{lemma}

\begin{proof}
The index set for the elements $P_i$ is the same for $n$ and $n-1$. In particular, $i\neq (n-1)/2$ which means that the $P_i$ in $\mathscr{R}_n$ are just mapped to the $P_i$ in $\mathscr{R}_{n-1}$ for all $i$, and  similarly for (the index sets of the) classes $B_J$. Moreover, in $\mathscr{R}_{n-1}/\mathscr{I}_{n-1}$ we have $X_{n-1}=B_{\{(n-2)/2\}}$. This proves the surjectivity of $\Phi_n$. Finally, note that the class $X_n$ doesn't appear in any relation in $\mathscr{R}_n$. 
\end{proof}

\begin{lemma}
\label{lem:oddproblem}
If $n$ is odd, then there is an exact sequence of graded $\op{W}(F)$-algebras:
\[
\mathscr{R}_n/\mathscr{I}_n\xrightarrow{\Phi_n} \mathscr{R}_{n-1}/\mathscr{I}_{n-1} \to \op{W}(F)[X_{n-1}]/(X_{n-1}^2)\to 0.
\] 
\end{lemma}

\begin{proof}
The elements $P_i\in\mathscr{R}_n$ with $i<(n-1)/2$ are mapped under ${\Phi_n}$ to the elements with the same name in $\mathscr{R}_{n-1}$. The same holds for the elements $B_J$ where the index set $J$ doesn't contain $(n-1)/2$. In particular, the subalgebra of $\mathscr{R}_{n-1}/\mathscr{I}_{n-1}$ generated by all $P_i$ and $B_J$ is in the image. The only elements in $\mathscr{R}_n$ we have not yet considered so far are the new $P_{(n-1)/2}$ and the elements $B_J$ where $J$ contains $(n-1)/2$.  The element $X_{n-1}^2$ is in the image of $P_{(n-1)/2}$, and the elements $B_J'X_{n-1}$ are in the image of $B_J$. However, the element $X_{n-1}$ itself is not in the image since we noted in Lemma~\ref{lem:evenpoly} that it is a polynomial variable in $\mathscr{R}_{n-1}$. Consequently, defining the morphism $\mathscr{R}_{n-1}/\mathscr{I}_{n-1}\to \op{W}(F)[X_{n-1}]/(X_{n-1}^2)$ by sending $X_{n-1}$ to itself and all the other generators to $0$ yields the desired exact sequence.
\end{proof}

\subsection{Statement of results} 

The first result describes $\op{H}^\bullet({\op{B}}\op{SL}_n,\mathbf{I}^\bullet)$ in terms of the rings $\mathscr{R}_n/\mathscr{I}_n$ described earlier. Note that by Lemma~\ref{lem:module}  $\op{H}^\bullet({\op{B}}\op{SL}_n,\mathbf{I}^\bullet)$ is a $\op{W}(F)$-module, and all the maps between cohomology groups we consider throughout the computation are $\op{W}(F)$-linear. Note also that $\mathscr{R}_n$ was defined using the appropriate graded commutativity so as to ensure the existence of the morphism in point (1) below. One should keep in mind that $\mathscr{R}_n$ is not commutative, and its image in $\op{H}^\bullet({\op{B}}\op{SL}_n,\mathbf{I}^\bullet)$ is commutative only after we establish that the Bockstein classes are $\op{I}(F)$-torsion.

\begin{theorem}
\label{thm:slnin}
Let $n\geq 2$ be a natural number. 
\begin{enumerate}[(1)]
\item The homomorphism of graded commutative $\op{W}(F)$-algebras defined by
\[
\theta_n\colon \mathscr{R}_n\to  \op{H}^\bullet({\op{B}}\op{SL}_n,\mathbf{I}^\bullet)\colon P_i\mapsto\op{p}_{2i},\; X_n\mapsto \op{e}_n, \;B_J\mapsto \beta(\overline{\op{c}}_{2j_1}\cdots\overline{\op{c}}_{2j_l})
\]
induces a graded $\op{W}(F)$-algebra isomorphism $\overline{\theta}_n\colon \mathscr{R}_n/\mathscr{I}_n \xrightarrow{\cong} \op{H}^\bullet({\op{B}}\op{SL}_n, \mathbf{I}^\bullet)$. Here the images of the generators are the characteristic classes in $\mathbf{I}$-cohomology as defined in Definitions~\ref{def:pontryagingln}, \ref{def:bockchern} and Remark \ref{rem:inclasses}. The isomorphisms are compatible with restriction via $\Phi_n$ and $j_n^\ast$, respectively.
\item The reduction morphism 
\[
\rho\colon \op{H}^{\bullet}_{\op{Nis}}({\op{B}}\op{SL}_n,\mathbf{I}^\bullet)\to\op{Ch}^\bullet({\op{B}}\op{SL}_n)
\]
induced from the projection $\mathbf{I}^n\mapsto\mathbf{K}^{\op{M}}_n/2$ is explicitly given by mapping
\[
\op{p}_{2i}\mapsto \overline{\op{c}}_{2i}^2, \; \beta(\overline{\op{c}}_{2j_1}\cdots\overline{\op{c}}_{2j_l})\mapsto \op{Sq}^2(\overline{\op{c}}_{2j_1}\cdots\overline{\op{c}}_{2j_l}),\; \op{e}_n\mapsto \overline{\op{c}}_n. 
\]
\end{enumerate}
\end{theorem}

This result will be proved in Section~\ref{sec:special}. For now, we deduce the following complete description of the Chow--Witt ring of ${\op{B}}\op{SL}_n$. 

\begin{theorem}
\label{thm:slnchw}
\begin{enumerate}[(1)]
\item The kernel of the Bockstein map $\beta$
\[
\mathbb{Z}/2\mathbb{Z}[\overline{\op{c}}_2,\overline{\op{c}}_3,\dots,\overline{\op{c}}_n]\cong \op{Ch}^\bullet({\op{B}}\op{SL}_n)\to \op{H}^\bullet({\op{B}}\op{SL}_n,\mathbf{I}^\bullet) 
\]
is given by the subring $\mathbb{Z}/2\mathbb{Z}[\overline{\op{c}}_{2i+1},\overline{\op{c}}_{2i}^2,\overline{\op{c}}_n] \subseteq \mathbb{Z}/2\mathbb{Z}[\overline{\op{c}}_2,\dots,\overline{\op{c}}_n]$. The kernel of the composition $\partial\colon \op{CH}^\bullet({\op{B}}\op{SL}_n)\to \op{Ch}^\bullet({\op{B}}\op{SL}_n)\xrightarrow{\beta} \op{H}^\bullet({\op{B}}\op{SL}_n,\mathbf{I}^\bullet)$ is the subring 
\[
\ker\partial=\mathbb{Z}[\op{c}_{2i+1},2\op{c}_{2i_1} \cdots \op{c}_{2i_k},{\op{c}}_{2i}^2,\op{c}_n]\subseteq\mathbb{Z}[\op{c}_2,\dots,\op{c}_n]\cong \op{CH}^\bullet({\op{B}}\op{SL}_n).
\]
\item
There is a cartesian square of graded $\op{GW}(F)$-algebras
\[
\xymatrix{
\widetilde{\op{CH}}^\bullet({\op{B}}\op{SL}_n)\ar[r] \ar[d] & \ker\partial \ar[d]^{\bmod 2} \\
\op{H}^{\bullet}_{\op{Nis}}({\op{B}}\op{SL}_n,\mathbf{I}^\bullet)\ar[r]_\rho &\op{Ch}^\bullet({\op{B}}\op{SL}_n).
}
\]
The right vertical morphism is the natural reduction mod $2$ restricted to $\ker\partial$, and the lower horizontal morphism is the reduction morphism described in Theorem~\ref{thm:slnin}. The Chow--Witt-theoretic Euler class satisfies $\op{e}_n=(\op{e}_n,\op{c}_n)$. For the Chow--Witt-theoretic Pontryagin classes, we have
\[
\op{p}_{i}=\left(\op{p}_{i}, (-1)^i\op{c}_{i}^2+2\sum_{j=\max\{0,2i-n\}}^{i-1}(-1)^j\op{c}_j\op{c}_{2i-j}\right)
\]
where the odd Pontryagin classes in $\mathbf{I}$-cohomology are $\op{I}(F)$-torsion and satisfy $\op{p}_{2i+1}=\beta(\overline{\op{c}}_{2i}\overline{\op{c}}_{2i+1})$. 

\end{enumerate}
\end{theorem}

\begin{proof}
(1)  
By the B\"ar sequence (\ref{eq:baer}), the kernel of the Bockstein map is exactly the image of the reduction morphism $\rho\colon \op{H}^\bullet({\op{B}}\op{SL}_n,\mathbf{I}^\bullet)\to \op{Ch}^\bullet({\op{B}}\op{SL}_n)$. The elements $\overline{\op{c}}_{2i}^2$ are the images of the Pontryagin classes, by Theorem~\ref{thm:slnin}, and therefore are in the kernel of the Bockstein map. The element $\overline{\op{c}}_n$ is the image of the Euler class, by Theorem~\ref{thm:slnin}, and therefore is also in the kernel of the Bockstein map. We note that we have an equality $\overline{\op{c}}_{2i+1}=\op{Sq}^2(\overline{\op{c}}_{2i})$ in $\op{Ch}^\bullet({\op{B}}\op{SL}_n)$ which is established in \cite[Proposition 10.3 and Remark 10.5]{fasel:ij} or could alternatively be proved using topological realization and cycle class maps. The elements $\overline{\op{c}}_{2i+1}=\op{Sq}^2(\overline{\op{c}}_{2i})$ are the images of $\beta(\overline{\op{c}}_{2i})$, by Theorem~\ref{thm:slnin}, and therefore are in the kernel of the Bockstein map. On the other hand, the elements $\overline{\op{c}}_{2i}$ (except possibly in the case $n=2i$) are not in the image of $\rho$, because $\op{Sq}^2(\alpha)=0$ for elements in the image of $\rho$ and $\op{Sq}^2(\overline{\op{c}}_{2i})=\overline{\op{c}}_{2i+1}$. Therefore we have  $\beta(\overline{\op{c}}_{2i})\neq 0$. To see that products of these with arbitrary elements also have non-trivial Bocksteins, we can use Lemma~\ref{lem:jacobi} which tells us that $\op{Sq}^2\colon \op{Ch}^n\to\op{Ch}^{n+1}$ is a derivation. Then we can inductively show that $\op{Sq}^2$ of a product of even Chern classes (excluding $\overline{\op{c}}_n$) is nontrivial by an explicit calculation in $\op{Ch}^\bullet({\op{B}}\op{SL}_n)$. Totaro's identification  $\op{Sq}^2=\rho\beta$, cf. Proposition~\ref{prop:sq2}, then proves the claim. The statement for the kernel of $\partial$ follows directly from the mod 2 statement.

(2) The statement about the cartesian square follows directly from Proposition~\ref{prop:cartesian}. In the claim about the Chow--Witt-theoretic characteristic classes, the reduction from Chow--Witt to $\mathbf{I}$-cohomology follows from the definition of the characteristic classes. The statement about $\op{e}_n$ and $\op{c}_n$ follows from Proposition~\ref{prop:chern}. The statement about the $\op{p}_i$ will be established in Propositions~\ref{prop:pontchern} and \ref{prop:ponttor}. 
\end{proof}

\begin{remark}
Pieces of the above computation have appeared in the literature, but it seems they haven't been combined into a single statement so far. Of course, the Chow-ring of $\op{B}\op{SL}_n$ is well-known thanks to Totaro \cite{totaro:bg}. A computation with the Chow--Witt group relevant for the Euler class of ${\op{B}}\op{GL}_n$ appeared in \cite{AsokFaselEuler}. Integral Stiefel--Whitney classes have been defined in  \cite{fasel:ij}, who computed $\widetilde{\op{CH}}^{\bullet}({\op{B}}\op{GL}_1)$
and $\op{H}^{\bullet}_{\op{Nis}}(\op{B}\op{GL}_1,\mathbf{I}^\bullet)$. 
Pontryagin classes appear in the $\eta$-local computations of Ananyevskiy, but the more complicated Bockstein classes have not been considered before. As pointed out by a referee, the hyperbolic morphism 
\[
\mathbf{K}^{\op{M}}_n(F)\to\mathbf{K}^{\op{MW}}_n(F)\colon \left\{a_1,\dots,a_n\right\}\mapsto [a_1^2,a_2,\dots,a_n]=h\cdot[a_1,\dots,a_n] 
\]
induces a morphism $H\colon \op{CH}^\bullet(X)\to\widetilde{\op{CH}}^\bullet(X)$ and we obtain new classes, so-called ``hyperbolic Chern classes'' $H(\op{c}_{2i_1}\cdots\op{c}_{2i_k})$, which are natural lifts of $2\op{c}_{2i_1}\cdots\op{c}_{2i_k}\in\op{CH}^\bullet(X)$ to $\widetilde{\op{CH}}^\bullet(X)$. 
\end{remark}

\begin{example}
To clarify matters, we spell out in detail the cartesian square for ${\op{B}}\op{SL}_3$: 
\[
\xymatrix{
\widetilde{\op{CH}}^\bullet({\op{B}}\op{SL}_3)\ar[r] \ar[d] & \mathbb{Z}[2{\op{c}_2},\op{c}_2^2,\op{c}_3] \ar[d] \\
\op{W}(F)[\op{p}_2,\beta(\overline{\op{c}}_2)]/(\op{I}(F)\beta(\overline{\op{c}}_2)) \ar[r] &\mathbb{Z}/2\mathbb{Z}[\overline{\op{c}}_2,\overline{\op{c}}_3].
}
\]
The characteristic classes in Chow--Witt theory are the following: we have the Pontryagin class $\op{p}_2=(\op{p}_2,\op{c}_2^2)$, the Euler class $\op{e}_3=(\beta(\overline{\op{c}}_2),\op{c}_3)$, and a class $H({\op{c}_2})=(0,2{\op{c}_2})$ lifting the class $2{\op{c}_2}$ in the Chow ring. 
\end{example}

\section{Relations between characteristic classes}
\label{sec:relations}

In this section, we establish some of the relations between the characteristic classes defined in Section~\ref{sec:character}. 

\subsection{A remark on torsion} 
\label{sec:torsion}
We briefly need to discuss the relation between torsion statements in the topological \cite{brown} and in the algebraic setting. The appropriate cohomology theory to compare torsion statements is the $\mathbf{I}^j$-cohomology and not the Chow--Witt ring (e.g. $\op{p}_i$ for $i$ odd are non-torsion in the Chow--Witt ring but torsion in $\mathbf{I}^j$-cohomology). 

One of the central statements about the integral cohomology of ${\op{B}}\op{SO}(n)$ is that all the torsion is simply 2-torsion. This allows for instance to prove relations in the cohomology ring by checking on the mod 2 cohomology. The key lemma in this reduction is \cite[Lemma 2.2]{brown}. Here is the appropriate version of the lemma in the algebraic setting. 

\begin{lemma}
\label{lem:brown22}
Suppose that for all $q$ in the composition 
\[
\op{H}^{q+1}(X,\mathbf{I}^{q+2})\xrightarrow{\eta} \op{H}^{q+1}(X,\mathbf{I}^{q+1})\xrightarrow{\eta} \op{H}^{q+1}(X,\mathbf{I}^{q})
\]
we have $\op{ker}(\eta)=\op{ker}(\eta \circ \eta)$.
Then the reduction $\rho\colon \op{H}^{q+1}(X,\mathbf{I}^{q+1})\to\op{Ch}^{q+1}(X)$ is injective on the kernel of $\eta\colon \op{H}^{q+1}(X,\mathbf{I}^{q+1})\to\op{H}^{q+1}(X,\mathbf{I}^q)$. 
\end{lemma}

\begin{proof}
This is proved via the same exact sequence manipulation as in topology, replacing the usual long exact cohomology sequence associated to $0\to 2\mathbb{Z}\to\mathbb{Z}\to\mathbb{Z}/2\mathbb{Z}\to 0$ by the B\"ar sequence (\ref{eq:baer}).
\end{proof}

\begin{remark}
The above lemma is not as powerful as \cite[Lemma 2.2]{brown} in topology: 
we can't verify the hypotheses for ${\op{B}}\op{SL}_n$
in non-geometric bidegrees. This is why in the arguments below we replace the use of Lemma 2.2 in \cite{brown} by more  direct arguments proving injectivity of reduction on the image of Bockstein. 
\end{remark}

As $\op{H}^\bullet({\op{B}}\op{SL}_n,\mathbf{I}^\bullet)$ is a module over $\op{W}(F)$, the appropriate algebraic equivalent of the $\mathbb{Z}$-torsion statements in topology is: all torsion in $\op{H}^\bullet({\op{B}}\op{SL}_n,\mathbf{I}^\bullet)$ is $\op{I}(F)$-torsion. It should be noted that e.g. for finite fields $\mathbb{F}_q$ with $q\equiv 3\bmod 4$, there is $4$-torsion in $\op{W}(\mathbb{F}_q)$, which implies that the $\op{H}^\bullet({\op{B}}\op{SL}_n,\mathbf{I}^\bullet)$ will contain non-trivial $4$-torsion. This, however, will arise from the $\op{W}(F)$-torsionfree part while the Bockstein classes will in fact be $2$-torsion. 

Recall the exact piece 
\[
\op{H}^j({\op{B}}\op{SL}_n,\mathbf{I}^{j+1})\xrightarrow{\eta} \op{H}^j({\op{B}}\op{SL}_n,\mathbf{I}^j)\xrightarrow{\rho}\op{Ch}^j({\op{B}}\op{SL}_n)
\]
of the B\"ar sequence (\ref{eq:baer}). We may alternatively say that all torsion in $\op{H}^j({\op{B}}\op{SL}_n,\mathbf{I}^j)$ is $\eta$-torsion (i.e. is in the image of the Bockstein map $\beta$), where here $\eta\colon \mathbf{I}^{j+1}\to\mathbf{I}^j$ is the inclusion of powers of the fundamental ideal as above. On the level of the Gersten complexes resp. the $\op{W}(F)$-module structure of cohomology, the $\eta$-torsion statement reduces exactly to the above $\op{I}(F)$-torsion statement.

\begin{lemma}
\label{lem:baertorsion}
For a set of integers $J=\{j_1,\dots,j_l\}$ with $0<j_1<\cdots<j_l\leq [(n-1)/2]$, we have $\op{I}(F)\beta(\overline{\op{c}}_{2j_1}\cdots \overline{\op{c}}_{2j_l})=0$ in $\op{H}^\bullet({\op{B}}\op{SL}_n,\mathbf{I}^\bullet)$.
\end{lemma}

\begin{proof}
This is formal from the B\"ar sequence (\ref{eq:baer}) which implies $\eta\beta=0$. More explicitly, the maps in the B\"ar sequence are $\op{W}(F)$-linear, and the source of the Bockstein map is annihilated by $\op{I}(F)$. 
\end{proof}

\begin{remark}
At this point, we know that the Bockstein classes will commute with everything else in $\op{H}^\bullet({\op{B}}\op{SL}_n,\mathbf{I}^\bullet)$. Note that the Pontryagin classes (and the Euler class in the even case) are of even degree and therefore also commute with everything. So we already see at this point that the subring of $\op{H}^\bullet({\op{B}}\op{SL}_n,\mathbf{I}^\bullet)$ generated by the characteristic classes is commutative.
\end{remark}

\subsection{Euler class for odd rank}

The analogue of relation (2) in the definition of $\mathscr{R}_n/\mathscr{I}_n$, cf. Definition~\ref{def:rnsln}, holds in $\op{H}^n({\op{B}}\op{SL}_n,\mathbf{I}^n)$. 

\begin{proposition}
\label{prop:eulerrel}
If $n=2k+1$, we have $\op{e}_n=\beta(\overline{\op{c}}_{n-1})$. 
\end{proposition}

\begin{proof}
This is proved in \cite[Theorem 10.1]{fasel:ij}. Note that our Stiefel--Whitney classes $\overline{\op{c}}_i$ agree with those of loc.cit. by Proposition~\ref{prop:sw}, and the identifications of the various Euler classes were discussed in Section~\ref{sec:prelims}. As the vector bundle we consider is the universal orientable rank $n$ vector bundle, its determinant is trivial.
\end{proof}

Via the B\"ar sequence, this implies, in particular, that the Euler class $\op{e}_n\in\op{H}^n({\op{B}}\op{SL}_n,\mathbf{I}^n)$ for odd $n$ is annihilated by $\eta\colon \op{H}^n({\op{B}}\op{SL}_n,\mathbf{I}^n)\to \op{H}^n({\op{B}}\op{SL}_n,\mathbf{I}^{n-1})$. By Lemma~\ref{lem:baertorsion}, $\op{e}_n$ is also annihilated by $\op{I}(F)$ for odd $n$.

\subsection{Pontryagin classes}

We start by giving some computations related to Pontryagin classes. The following square of classifying spaces will be relevant:
\[
\xymatrix{
{\op{B}}\op{SL}_n\ar[r]^{\sigma_n} \ar[d]_{\op{id} \times \op{conj}} & {\op{B}}\op{Sp}_{2n}\ar[d] \\
{\op{B}}\op{SL}_n\times {\op{B}}\op{SL}_n  \ar[r]_{\op{diag}} & {\op{B}}\op{SL}_{2n}
}
\]
Here the top horizontal morphism is the symplectification, the bottom horizontal morphism is induced from the diagonal block inclusion, and the right vertical morphism is induced from the standard inclusion $\op{Sp}_{2n}\hookrightarrow \op{SL}_{2n}$. To describe the left vertical morphism, recall that the symplectification of a vector bundle is the Whitney sum with its dual, equipped with the natural symplectic form. 
The left vertical morphism is then the product of the identity on one factor and the conjugation morphism (mapping a vector bundle to its dual) on the other factor. By construction, we get a commutative square of algebraic groups, as the one above only with $B$s removed. Technically, of course, we haven't said exactly in what sense the classifying spaces are actually spaces. If we work in the Morel--Voevodsky $\mathbb{A}^1$-homotopy category, cf. \cite{MV}, we actually get a commutative square of spaces. Working with finite-dimensional approximations, we get a commutative square of Chow--Witt rings of classifying spaces by Proposition~\ref{prop:inducedhom}.

\begin{proposition}
\label{prop:pontchern}
Under the projection $\widetilde{\op{CH}}^\bullet({\op{B}}\op{SL}_n)\to\op{CH}^\bullet({\op{B}}\op{SL}_n)$, the Pontryagin class $\op{p}_i$ maps to 
\[
(-1)^i\op{c}_{i}^2+2\sum_{j=\max\{0,2i-n\}}^{i-1}(-1)^j\op{c}_j\op{c}_{2i-j}.
\]
In particular, the Pontryagin class $\op{p}_i\in \op{H}^{2i}({\op{B}}\op{SL}_n,\mathbf{I}^{2i})$ maps to $\overline{\op{c}}_{i}^2$ under the reduction map $\rho\colon \op{H}^{2i}({\op{B}}\op{SL}_n,\mathbf{I}^{2i})\to\op{Ch}^{2i}({\op{B}}\op{SL}_n)$. 
\end{proposition}

\begin{proof}
Recall from Definition~\ref{def:pontryagingln} that the Pontryagin classes are defined via the symplectification $\sigma_n$. Under $\widetilde{\op{CH}}^\bullet({\op{B}}\op{Sp}_{2n})\to\op{CH}^\bullet({\op{B}}\op{Sp}_{2n})$, the Pontryagin class $\op{p}_i$ maps to $\op{c}_{2i}$, cf. Theorem~\ref{thm:symplectic}. Hence
it suffices to determine the image of the Chern class $\op{c}_{2i}$ under the morphism $\sigma_n^*\colon \op{CH}^\bullet({\op{B}}\op{Sp}_{2n})\to\op{CH}^\bullet({\op{B}}\op{SL}_n)$ induced by symplectification. Using the above commutative square
(and noting that $\op{c}_{2i}$ for the symplectic group is induced from $\op{CH}^\bullet({\op{B}}\op{SL}_{2n})$ via the forgetful morphism), we thus have to compute the image of $\op{c}_{2i}$ under the composition
\[
 \op{CH}^\bullet({\op{B}}\op{SL}_{2n})\xrightarrow{\op{diag}^*} \op{CH}^\bullet({\op{B}}(\op{SL}_n\times\op{SL}_n))\xrightarrow{\op{id}^\ast\times \op{conj}^\ast} \op{CH}^\bullet({\op{B}}\op{SL}_n).
\]

By Proposition~\ref{prop:vbconjugate}, for a vector bundle $\mathscr{E}$, the Chern classes of the dual vector bundle $\overline{\mathscr{E}}$ are given as $\op{c}_i(\overline{\mathscr{E}})=(-1)^i\op{c}_i(\mathscr{E})$. 
Now we can apply the Whitney sum formula from Proposition~\ref{prop:chern} to determine the Chern classes of the symplectification of a vector bundle:
\[
\op{c}_i(\mathscr{E}\oplus\overline{\mathscr{E}})=\sum_{j=0}^i \op{c}_j(\mathscr{E}) \boxtimes\op{c}_{i-j}(\overline{\mathscr{E}}) = \sum_{j=0}^i (-1)^{i-j}\op{c}_j(\mathscr{E}) \op{c}_{i-j}(\mathscr{E}).
\]
By induction, these classes are trivial for odd $i$, and for even $i$ we get the desired formula. 
\end{proof}

\begin{remark}
The above formula shows that the odd Pontryagin classes are non-$I(F)$-torsion in the Chow--Witt ring since they are non-torsion in the Chow ring. However, in $\mathbf{I}$-cohomology, the above formula reduces exactly to the one in \cite{brown}. Proposition~\ref{prop:ponttor} will show that the odd Pontryagin classes are indeed torsion in the $\mathbf{I}$-cohomology (which is a consequence of the same result for the Euler class, cf. Proposition~\ref{prop:eulerrel}). So in some sense, only the even $p_i$ contribute new information, the odd $p_i$ are completely determined by the Chern and Bockstein classes.
\end{remark}

We now discuss some version of the splitting principle, which will be used to deduce independence of Pontryagin classes for ${\op{B}}\op{SL}_n$. For this, we have the following square similar to the one discussed before: 
\[
\xymatrix{
{\op{B}}\op{Sp}_{2n}\ar[r] \ar[d]_{\op{id}\times\op{conj}} & {\op{B}}\op{SL}_{2n}\ar[d]^{\sigma_{2n}} \\
{\op{B}}\op{Sp}_{2n}\times {\op{B}}\op{Sp}_{2n}  \ar[r]_>>>>>{\op{diag}} & {\op{B}}\op{Sp}_{4n}
}
\]
Here the top horizontal morphism is induced from the standard embedding $\op{Sp}_{2n}\hookrightarrow\op{SL}_{2n}$, and the left vertical morphism is, as before, given by the conjugation on the second factor. The diagram is commutative because the corresponding square of algebraic groups is commutative. 

\begin{proposition}
\label{prop:slnsplit}
The restriction of the projection $\widetilde{\op{CH}}^\bullet({\op{B}}\op{SL}_{2n})\to \widetilde{\op{CH}}^\bullet({\op{B}}\op{Sp}_{2n})$ to the $\op{GW}(F)$-subalgebra of $\widetilde{\op{CH}}^\bullet({\op{B}}\op{SL}_{2n})$ which is generated by the Pontryagin classes $\op{p}_{2},\dots,\op{p}_{2n-2}$ and the Euler class $\op{e}_{2n}$ is injective.
\end{proposition}

\begin{proof}
By definition the Pontryagin classes are defined as Pontryagin classes of the symplectification of the universal bundle over ${\op{B}}\op{SL}_{2n}$; in particular, they come from $\widetilde{\op{CH}}^\bullet({\op{B}}\op{Sp}_{4n})$. To determine the images of $\op{p}_{2i}$ under the restriction, we can use the commutative square discussed previously and determine the image of $\op{p}_{2i}\in\widetilde{\op{CH}}^\bullet({\op{B}}\op{Sp}_{4n})$ under the composition
\[
 \widetilde{\op{CH}}^\bullet({\op{B}}\op{Sp}_{4n})\xrightarrow{\op{diag}^\ast} \widetilde{\op{CH}}^\bullet({\op{B}}(\op{Sp}_{2n}\times\op{Sp}_{2n}))\xrightarrow{\op{id}^\ast\times \op{conj}^\ast} \widetilde{\op{CH}}^\bullet({\op{B}}\op{Sp}_{2n}).
\]
Now we can apply the Whitney sum formula from Theorem~\ref{thm:symplectic} to determine the Pontryagin classes of the symplectification of a symplectic bundle, using Proposition~\ref{prop:sympconj}:
\[
\op{p}_i(\mathscr{E}\oplus\overline{\mathscr{E}})=\sum_{j=0}^i \op{p}_j(\mathscr{E})\boxtimes \op{p}_{i-j}(\overline{\mathscr{E}})= \sum_{j=0}^i \langle-1\rangle^{i-j}\op{p}_j(\mathscr{E}) \op{p}_{i-j}(\mathscr{E}).
\]

By induction, using the symplectic splitting principle of Proposition~\ref{prop:symsplit}, one can now show that  the restriction along ${\op{B}}\op{SL}_2^{\times n}\to{\op{B}}\op{Sp}_{2n}$ maps the Pontryagin class $\op{p}_i$ to $\sigma_i^2(\op{p}_{1,1},\dots,\op{p}_{1,n})$, the square of the $i$-th elementary symmetric polynomial in the elements $\op{p}_{1,1},\dots,\op{p}_{1,n}$. Since the restriction to the quaternionic torus is injective, cf. point (2) of Theorem~\ref{thm:symplectic}, this means that the Pontryagin class $\op{p}_{2i}$ maps to $\op{p}_i^2$. This shows the claim.
\end{proof}
 
The following is not a relation in the Chow--Witt ring or the $\mathbf{I}^j$-cohomology; it describes the relation of the Euler class (which is a characteristic class in Chow--Witt theory) to the Pontryagin class of the symplectified bundle (which at this point is not a characteristic class in ${\op{B}}\op{SL}_n$). It is a statement about a restriction map which will be used later. 

\begin{proposition}
\label{prop:eulersquare}
The restriction morphism 
\[
\widetilde{\op{CH}}^{2n}({\op{B}}\op{Sp}_{2n})\to \widetilde{\op{CH}}^{2n}({\op{B}}\op{SL}_n)\subset \op{H}^{2n}({\op{B}}\op{SL}_n,\mathbf{I}^{2n})\times_{\op{Ch}^{2n}({\op{B}}\op{SL}_n)} \op{CH}^{2n}({\op{B}}\op{SL}_n)
\]
induced by symplectification maps the top Pontryagin class $\op{p}_n$ as follows: 
\begin{enumerate}
\item the image in $\ker\partial\subset\op{CH}^{2n}({\op{B}}\op{SL}_n)$ is $(-1)^n\op{c}_n^2$; 
\item the image in $\op{H}^{2n}({\op{B}}\op{SL}_n,\mathbf{I}^{2n})$ is $\langle-1\rangle^{n/2}\op{e}_n^2$ if $n$ is even and $\op{e}_n^2$ otherwise.
\end{enumerate}
\end{proposition}

\begin{proof}
Recall from Proposition~\ref{prop:spuniq} that the top Pontryagin class in ${\op{B}}\op{Sp}_{2n}$ is given by the Euler class. By compatibility of Thom classes with pullbacks, this means that $\op{p}_n$ is obtained by pullback of the Euler class $\op{e}_{2n}$ along the standard inclusion ${\op{B}}\op{Sp}_{2n}\to{\op{B}}\op{SL}_{2n}$. Using the previous commutative square, we thus have to compute the image of $\op{e}_{2n}\in\widetilde{\op{CH}}^\bullet({\op{B}}\op{SL}_{2n})$ under the composition 
\[
\widetilde{\op{CH}}^\bullet({\op{B}}\op{SL}_{2n})\xrightarrow{\op{diag}^\ast} \widetilde{\op{CH}}^\bullet({\op{B}}(\op{SL}_n\times\op{SL}_n)) \xrightarrow{\op{id}^\ast\times \op{conj}^\ast} \widetilde{\op{CH}}^\bullet({\op{B}}\op{SL}_n).
\]
By the Whitney sum formula, Lemma~\ref{lem:eulermult}, the image of the Euler class of the symplectification of a vector bundle $\mathscr{E}$ in Chow theory is given by 
\[
\op{c}_{2n}(\mathscr{E}\oplus\overline{\mathscr{E}})= \op{c}_n(\mathscr{E})\op{c}_n(\overline{\mathscr{E}}) = (-1)^n\op{c}_n^2(\mathscr{E}).
\]
Similarly, we can use the Whitney sum formula to determine the image of the Euler class of the symplectification of a vector bundle in $\mathbf{I}$-cohomology. To determine the action of vector bundle conjugation on the $\mathbf{I}$-cohomology for $n$ even, we can use Proposition~\ref{prop:slnsplit}. The restriction of $\op{p}_n$ to ${\op{B}}\op{Sp}_n$ equals $\langle-1\rangle^{n/2}\op{e}_n^2$, by Proposition~\ref{prop:sympconj}, which implies the claim. In the case where $n$ is odd, we know $\op{e}_{2n}(\mathscr{E}\oplus\overline{\mathscr{E}})=\op{e}_n(\mathscr{E})\op{e}_n(\overline{\mathscr{E}})$, but by Proposition~\ref{prop:eulerrel}, $\op{e}_n(\mathscr{E})$ is $\op{I}(F)$-torsion and so will be the product. In particular it is in the image of the Bockstein map, cf. Lemma~\ref{lem:baertorsion}. The effect of conjugation of vector bundles on the image of the Bockstein map is determined by the effect of conjugation on $\op{Ch}^\bullet({\op{B}}\op{SL}_n)$. But the latter is trivial, which implies the claim.
\end{proof}

\begin{remark}
Although it looks different, this is the same relation as in \cite[Theorem 1.5]{brown}. The differences arise because our choice of indexing as well as sign of the Pontryagin classes for vector bundles differ from those made in \cite{brown}. For us, the Pontryagin classes of vector bundles are exactly the characteristic classes of the symplectification, with no added sign or reindexing.
\end{remark}

\begin{corollary}
\label{cor:redmod2}
The reduction morphism $\rho\colon \op{H}^\bullet({\op{B}}\op{SL}_n,\mathbf{I}^\bullet)\to\op{Ch}^\bullet({\op{B}}\op{SL}_n)$ is given by 
\[
\op{p}_{2i}\mapsto \overline{\op{c}}^2_{2i},\; \beta(\overline{\op{c}}_{2j_1}\cdots \overline{\op{c}}_{2j_l})\mapsto \op{Sq}^2(\overline{\op{c}}_{2j_1}\cdots \overline{\op{c}}_{2j_l}), \; \op{e}_n\mapsto \overline{\op{c}}_n.
\]
In particular, part (2) of Theorem~\ref{thm:slnin} is true. 
\end{corollary}

\begin{proof}
This follows directly from Proposition~\ref{prop:pontchern} (via the commutative square from the key diagram) and Proposition~\ref{prop:chern}.
\end{proof}

\subsection{Relations in the mod 2 Chow ring}

By definition, it is clear that we have a morphism $\mathscr{R}_n\to\op{H}^\bullet_{\op{Nis}}({\op{B}}\op{SL}_n,\mathbf{I}^\bullet)$. A first step towards the proof is to establish enough relations between characteristic classes to show that the ideal $\mathscr{I}_n$ is annihilated by the composition 
\[
\mathscr{R}_n \xrightarrow{\theta_n} \op{H}^\bullet_{\op{Nis}}({\op{B}}\op{SL}_n,\mathbf{I}^\bullet)\xrightarrow{\rho} \op{Ch}^\bullet({\op{B}}\op{SL}_n). 
\]

\begin{lemma}
\label{lem:oddeulermod2}
Assume $n$ is odd. With the above notation we have
\[
\rho(\op{e}_n)=\rho\circ \beta(\overline{\op{c}}_{n-1})=\overline{\op{c}}_{n}.
\]
\end{lemma}

\begin{proof}
This follows from \cite[Proposition 10.3, Remark 10.5]{fasel:ij}, Totaro's identification of $\op{Sq}^2=\rho\circ\beta$, cf. Proposition~\ref{prop:sq2}, and the identification of Stiefel--Whitney classes with reductions of Chern classes in Proposition~\ref{prop:sw}.
\end{proof}

\begin{proposition}
\label{prop:type3mod2}
For two index sets $J$ and $J'$, the elements 
\[
B_J\cdot B_{J'}-\sum_{k\in J} B_{\{k\}}\cdot P_{(J\setminus\{k\})\cap J'}\cdot B_{\Delta(J\setminus\{k\},J')}
\]
have trivial images under the composition $\rho\circ\theta_n\colon \mathscr{R}_n\to\op{Ch}^\bullet({\op{B}}\op{SL}_n)$.
\end{proposition}

\begin{proof}
The proof that $\rho \circ \theta_n$ maps type (3) relations to zero is similar to the topological case. Note that $\rho\circ \theta_n$ maps an element $B_J$ to $\op{Sq}^2(\overline{\op{c}}_{2j_1}\cdots \overline{\op{c}}_{2j_1})$, by Totaro's identification $\op{Sq}^2 = \rho\circ \beta$, cf. Proposition~\ref{prop:sq2}. 

We need another preliminary equality concerning Steenrod squares of squares of Chern class reductions. From the fact that $\op{Sq}^2$ is a derivation we have
\[
\op{Sq}^2 (X\cdot\overline{\op{c}}_{2i}^2)=\op{Sq}^2(X)\cdot\overline{\op{c}}_{2i}^2+ X\cdot \op{Sq}^2\overline{\op{c}}_{2i}^2=\op{Sq}^2(X)\cdot\rho(\op{p}_i).
\]
where the last equality follows since $\overline{\op{c}}_{2i}^2$ is integrally defined and hence has trivial $\op{Sq}^2$. (Alternatively, $\op{Sq}^2(X^2)=2X\op{Sq}^2(X)=0$ since we have mod 2 coefficients.)

The argument for the special linear case is now the same as in \cite{brown}, using the above equality together with the Jacobi identity 
\[
\op{Sq}^2 (XY) \op{Sq}^2 Z = \op{Sq}^2 X \op{Sq}^2 (YZ) + \op{Sq}^2 (XZ) \op{Sq}^2 Y
\]
cf. Lemma~\ref{lem:jacobi}. The proof is by induction on the size of the index set  $J={j_1,...,j_l}$, setting $X=B_{j_1},Y=B_{\{j_2,...,j_l\}}$ and $Z=B_J$, distinguishing the two cases whether $j_1$ is already contained in $J$ or not.
\end{proof}

\begin{remark}
In the presentation $\mathscr{R}_n/\mathscr{I}_n$,  the relation for $B_{\{i,j\}}\cdot B_{k}$ for three distinct indices $i,j,k$ is a lift of the Jacobi identity for $\op{Sq}^2$, cf. Lemma~\ref{lem:jacobi}. 
\end{remark}

\begin{corollary}
\label{cor:slninmod2}
The composition $\rho\circ\theta_n\colon \mathscr{R}_n\to\op{Ch}^\bullet({\op{B}}\op{SL}_n)$ factors through the quotient $\mathscr{R}_n/\mathscr{I}_n$. 
\end{corollary}

\begin{proof}
This follows directly from Lemma~\ref{lem:oddeulermod2} and Proposition~\ref{prop:type3mod2}. 
\end{proof}

This completes the preliminary computations.

\section{\texorpdfstring{$\op{H}^{\bullet}_{\op{Nis}}({\op{B}}\op{SL}_n,\mathbf{I}^\bullet)$}{H(BSLn,I)}: proof of the theorem}
\label{sec:special}

The goal of this section is to prove Theorem~\ref{thm:slnin} which establishes a computation of $\op{H}^\bullet({\op{B}}\op{SL}_n,\mathbf{I}^\bullet)$ very close to the classical computations of $\op{H}^\bullet({\op{B}}\op{SO}(n),\mathbb{Z})$ in \cite{brown}. 

\subsection{Setup for inductive proof}

The proof of Theorem~\ref{thm:slnin} will be by an induction on the rank, and a distinction  of cases between $n$ even or odd. The argumentation essentially follows \cite{brown}, with some modifications as in \cite{RojasVistoli} regarding independence of classes. 

\begin{remark}
The philosophical reason why the inductive argument works for the series of groups $\op{SL}_n$ is that the stabilization morphism ${\op{B}}\op{SL}_{n-1}\to{\op{B}}\op{SL}_n$ has homotopy fiber a sphere $\mathbb{A}^n\setminus \{0\}$. In particular, the only essentially new contribution at each stage is the Euler class which is related to the fundamental class of $\mathbb{A}^n\setminus\{0\}$ by transgression, cf. \cite{AsokFaselEuler}. This is no longer true for orthogonal groups, where even-dimensional quadrics $\op{Q}_{2n}$ appear, leading to contributions from non-geometric bidegrees. This will be discussed elsewhere.
\end{remark}

We first need to set up a comparison between the candidate presentation as formulated in Theorem~\ref{thm:slnin} and the localization sequence of Proposition~\ref{prop:locsln}. The first step will be to establish the commutativity of the following cube diagram:
\[
\xymatrix{
\mathscr{R}_n \ar[rr]^{\theta_n} \ar[dd]_{\Phi_n} \ar[rd]^{\pi_n} && \op{H}^\bullet({\op{B}}\op{SL}_n,\mathbf{I}^\bullet) \ar[dr]^\rho \ar'[d]_{j^\ast}[dd] & \\
&\mathscr{R}_n/\mathscr{I}_n\ar[rr]_(.3){\overline{\theta}_n} \ar[dd]_(.3){\overline{\Phi}_n}&& \op{Ch}^\bullet({\op{B}}\op{SL}_n) \ar[dd]^{j^\ast} \\
\mathscr{R}_{n-1} \ar'[r][rr]_(.3){\theta_{n-1}}\ar[rd]_{\pi_{n-1}} && \op{H}^\bullet({\op{B}}\op{SL}_{n-1},\mathbf{I}^\bullet) \ar[dr]_\rho \\&\mathscr{R}_{n-1}/\mathscr{I}_{n-1}\ar[rr]_{\overline{\theta}_{n-1}}&& \op{Ch}^\bullet({\op{B}}\op{SL}_{n-1}).
}
\]
Let us first describe the maps in the cube: the maps $j^\ast$ on the right face are the restriction maps appearing in the localization sequence of Proposition~\ref{prop:locsln} and the maps $\rho$ on the right face are induced by the reduction of coefficients $\rho\colon \mathbf{I}^q\to\mathbf{K}^{\op{M}}_q/2$. The horizontal maps $\mathscr{R}_m\to\op{H}^\bullet({\op{B}}\op{SL}_m,\mathbf{I}^\bullet)$ are the ones described in Theorem~\ref{thm:slnin}, and the horizontal maps $\mathscr{R}_m/\mathscr{I}_m\to\op{Ch}^q({\op{B}}\op{SL}_m)$ are the compositions of these with the reduction, which factor through the presentation by Corollary~\ref{cor:slninmod2}.  The maps $\pi_m\colon \mathscr{R}_m\to\mathscr{R}_m/\mathscr{I}_m$ are the natural quotient maps. 

By Proposition~\ref{prop:compatres}, we know that $\overline{\Phi}_n$ is actually a well-defined morphism and therefore the left face of the cubical diagram is commutative. The top and bottom faces are already known to commute by Corollary~\ref{cor:slninmod2}. The right face is commutative because the restriction maps in the localization sequence of Proposition~\ref{prop:locsln} are compatible with change-of-coefficients. Because of the surjectivity of the maps $\pi_m$, the front face is commutative if the back face is commutative. 

So it remains to show that the back face is commutative. Actually, we can only give a conditional statement:

\begin{proposition}
\label{prop:backface}
If $n$ is even, the back face of the cubical diagram is commutative.

If $n$ is odd, the back face is commutative if we assume that the map $\theta_{n-1}\colon \mathscr{R}_{n-1}\to \op{H}^\bullet({\op{B}}\op{SL}_{n-1},\mathbf{I}^\bullet)$ factors through an isomorphism 
\[
\Xi_{n-1}\colon \mathscr{R}_{n-1}/\mathscr{I}_{n-1}\xrightarrow{\cong} \op{H}^\bullet({\op{B}}\op{SL}_{n-1},\mathbf{I}^\bullet). 
\]
\end{proposition}

\begin{proof}
We first deal with the even case. By Lemma~\ref{lem:evenpoly}, $\mathscr{R}_n/\mathscr{I}_n\cong \mathscr{R}_{n-1}/\mathscr{I}_{n-1}[X_n]$; and the restriction sends Pontryagin classes and Bockstein classes to themselves, noting that the index sets for these are the same. So we only need to show 
\[
0=\theta_{n-1}\Phi_n(X_n) =j^\ast\theta_n(X_n)=j^\ast\op{e}_n.
\]
But $j^\ast\op{e}_n=0$ follows exactly by the definition of the Euler class in terms of the localization sequence, cf. Proposition~\ref{prop:locsln}.

Now we deal with the odd case. Note that in this case $(n-1)/2$ is an integer. The classes $P_i$ map to the Pontryagin classes $\op{p}_{2i}$ under $\theta_n$. In particular, for $i<(n-1)/2$ the two compositions agree on $P_i$ by Proposition~\ref{prop:pontryaginstab}. For the top Pontryagin class, i.e., the case $P_{\frac{n-1}{2}}$, the compositions agree by Proposition~\ref{prop:eulersquare}. Both compositions send the Euler class to $0$. 

It remains to deal with the Bockstein classes. For classes $B_J$ which already exist in $\mathscr{R}_{n-1}$, the fact that both compositions agree follows from the commutativity of the restriction maps $j^\ast$ with long exact sequences, cf. Proposition~\ref{prop:locsln}: a product of Chern classes which exist in $\mathscr{R}_{n-1}$ just restricts to itself, and by the compatibility mentioned the same will be true for their Bockstein. So it remains to deal with classes of the form $B_J$ with $J=J'\sqcup\{j_l\}$ where $j_l=(n-1)/2$. If $J'=\emptyset$, we need to show that $\beta(\overline{\op{c}}_{2j_l})=0$. But this is true since $\overline{\op{c}}_{2j_l}$ is the reduction of the Euler class $\op{e}_{n-1}$, i.e., is an integrally defined class and therefore has trivial Bockstein. Now consider the case $J'\neq\emptyset$. Such a class will be mapped under $\Phi_n$ to $B_{J'}\cdot X_{n-1}$, hence $\theta_{n-1}\Phi_n(B_J)=\beta(\overline{\op{c}}_{2j_1}\cdots\overline{\op{c}}_{2j_{l-1}})\cdot \op{e}_{n-1}$. On the other hand, $j^\ast\theta_n$ will map $B_J$ to $\beta(\overline{\op{c}}_{2j_1}\cdots\overline{\op{c}}_{2j_l})$ (viewed as an element of $\op{H}^q({\op{B}}\op{SL}_{n-1},\mathbf{I}^q)$). After composing with the reduction $\rho$, we have 
\[
\op{Sq}^2(\overline{\op{c}}_{2j_1}\cdots\overline{\op{c}}_{2j_l})= \op{Sq}^2(\overline{\op{c}}_{2j_1}\cdots\overline{\op{c}}_{2j_{l-1}})\overline{\op{c}}_{2j_l}+ \, 
\overline{\op{c}}_{2j_1}\cdots\overline{\op{c}}_{2j_{l-1}}\cdot \op{Sq}^2(\overline{\op{c}}_{2j_l}).
\]
The second term on the right hand side is trivial because $\overline{\op{c}}_{2j_l}$ is the reduction of the Euler class $\op{e}_{n-1}$. (Note that there is no difference between writing the product with $\overline{\op{c}}_{2j_l}$ as product with $\op{e}_{n-1}$ because the product will be a 2-torsion class anyway.) Now the difference of the corresponding Bockstein classes will have trivial reduction in $\op{Ch}^{\bullet}$
because of the above computation with $\op{Sq}^2$. By assumption, the Bockstein classes are detected by $\op{Ch}^{\bullet}$, proving the claim. 
\end{proof}

\subsection{Induction step, even case}

We first deal with $\op{H}^\bullet({\op{B}}\op{SL}_n,\mathbf{I}^\bullet)$ 
for $n$ even. \emph{In the whole section, we work under the inductive assumption that Theorem~\ref{thm:slnin} holds for $\op{H}^\bullet({\op{B}}\op{SL}_{n-1},\mathbf{I}^\bullet)$.} As in the symplectic case, the global structure of the argument is rather close to \cite{brown}. 
However, there are two major differences requiring modifications
in parts of the proof. First, the injectivity arguments cannot be done using the localization sequence because of the non-geometric bidegrees.
Second, dealing with torsion using an analogue \cite[Lemma 2.2]{brown} doesn't work directly, also due to non-geometric bidegrees. 

\begin{lemma}
Assume we know Theorem~\ref{thm:slnin} for ${\op{B}}\op{SL}_{n-1}$. Then the restriction map $j^\ast\colon \op{H}^\bullet({\op{B}}\op{SL}_n,\mathbf{I}^\bullet)\to \op{H}^\bullet({\op{B}}\op{SL}_{n-1},\mathbf{I}^\bullet)$ is surjective. 
\end{lemma}

\begin{proof}
It suffices to show that the composition $j^\ast\circ \theta_n$ is surjective. By the commutative cubical diagram, it suffices to check that $\theta_{n-1}\circ\Phi_n$ is surjective. By assumption, we know that there is an isomorphism $\Xi_{n-1}\colon \mathscr{R}_{n-1}/\mathscr{I}_{n-1}\to \op{H}^\bullet({\op{B}}\op{SL}_{n-1},\mathbf{I}^q)$ with $\Xi_{n-1}\circ\pi_{n-1}=\theta_{n-1}$, hence $\theta_{n-1}$ is surjective. By Lemma~\ref{lem:evenpoly}, the map $\Phi_n$ is also surjective. This proves the claim.
\end{proof}

\begin{proposition}
\label{prop:surjeven}
Assume we know Theorem~\ref{thm:slnin} for ${\op{B}}\op{SL}_{n-1}$. Then the presentation map $\theta_n\colon \mathscr{R}_n\to\op{H}^\bullet({\op{B}}\op{SL}_n,\mathbf{I}^\bullet)$ is a surjection.
\end{proposition}

\begin{proof}
The proof proceeds by induction on the cohomological degree $q$. The base case is established by the vanishing of cohomology in negative degrees. Assume we have established the surjectivity for degrees less than $q$, and let $\alpha\in\op{H}^q({\op{B}}\op{SL}_n,\mathbf{I}^q)$ be a class. As before, using the inductive assumption and Lemma~\ref{lem:evenpoly}, we find that $j^\ast \alpha$ is in the image of $\theta_{n-1}\circ\Phi_n$, i.e., we can find a class $\alpha'\in\mathscr{R}_n/\mathscr{I}_n$ such that $j^\ast\alpha=j^\ast\theta_n\alpha'$. Hence, we can assume without loss of generality that $j^\ast \alpha=0$. But the exactness of the localization sequence, cf. Proposition~\ref{prop:locsln}, implies that there exists a class $\gamma\in\op{H}^{q-n}({\op{B}}\op{SL}_n,\mathbf{I}^{q-n})$ such that $\op{e}_n\gamma=\alpha$. By inductive assumption, there is an element $\gamma'\in\mathscr{R}_{q-n}$ such that $\theta_n\gamma'=\gamma$. But then $\theta_n(X_n\gamma')=\op{e}_n\gamma=\alpha$, proving the claim.
\end{proof}


At this point, we deviate from Brown's argument. Brown goes on to prove that $\theta_n$ factors through the relations because of a statement about torsion which requires a priori the injectivity of the multiplication with the Euler class. Instead, we will use a splitting principle, establishing the torsion statement directly and deducing from that the injectivity of the Euler class on the torsion-free part. 

\begin{proposition}
\label{prop:evensplit}
The Pontryagin classes $\op{p}_{2},\dots, \op{p}_{n-2}$ and the Euler class
$\op{e}_n$ are algebraically independent in $\op{H}^\bullet({\op{B}}\op{SL}_n,\mathbf{I}^\bullet)$. In particular, the subalgebra of the cohomology generated by the Pontryagin classes and $\op{e}_n$ is a polynomial ring over $\op{W}(F)$ in these generators and multiplication with $\op{e}_n$ is injective on this subalgebra.
\end{proposition} 

\begin{proof}
This follows directly from Proposition~\ref{prop:slnsplit}. 
\end{proof}

This implies, in particular also that the subalgebra above is $\op{W}(F)$-torsion free. From the surjectivity of $\theta_n\colon \mathscr{R}_n\to\op{H}^\bullet({\op{B}}\op{SL}_n,\mathbf{I}^\bullet)$ we see that the $\op{W}(F)$-torsion in $\op{H}^\bullet({\op{B}}\op{SL}_n,\mathbf{I}^\bullet)$ is exactly the ideal generated by the image of the Bockstein map $\beta\colon \op{Ch}^\bullet({\op{B}}\op{SL}_n)\to \op{H}^\bullet({\op{B}}\op{SL}_n,\mathbf{I}^\bullet)$. However, this ideal will by definition be annihilated by $\eta\colon \op{H}^\bullet({\op{B}}\op{SL}_n,\mathbf{I}^\bullet)\to \op{H}^\bullet({\op{B}}\op{SL}_n,\mathbf{I}^{\bullet-1})$ and the exactness of the B\"ar sequence implies that the ideal generated by the image of $\beta$ coincides with the image of $\beta$. 

The remaining questions about the structure of $\op{H}^\bullet({\op{B}}\op{SL}_n,\mathbf{I}^\bullet)$ only concern the $\eta$-torsion: the only relations claimed are relations for the image of $\beta$, and of course we want to show that all relations in cohomology are accounted for in $\mathscr{R}_n/\mathscr{I}_n$. For this, Brown uses  \cite[Lemma 2.2]{brown}, which in our setting is not readily available. As a replacement, the following proposition proves the relevant consequences of the lemma directly.

\begin{proposition}
\label{prop:lem22even}
Assume that we know Theorem~\ref{thm:slnin} for ${\op{B}}\op{SL}_{n-1}$. Then for all degrees $q$, in the composition
\[
\op{Ch}^q({\op{B}}\op{SL}_n)\xrightarrow{\beta} \op{H}^{q+1}({\op{B}}\op{SL}_n,\mathbf{I}^{q+1})\xrightarrow{\rho} \op{Ch}^{q+1}({\op{B}}\op{SL}_n),
\]
the map $\rho$ is injective on the image of $\beta$. 
\end{proposition}

\begin{proof}
Note that the assumption implies that the analogue of the claim is true for ${\op{B}}\op{SL}_{n-1}$. We consider the following commutative diagram 
\[
\xymatrix{
&\op{Ch}^q({\op{B}}\op{SL}_n) \ar[r]^\beta \ar[rd]_(.2){\op{Sq}^2}|!{[r];[d]}\hole \ar[ld]_{j^\ast} & \op{H}^{q+1}({\op{B}}\op{SL}_n,\mathbf{I}^{q+1}) \ar[ld]^(.3){j^\ast} \ar[d]^\rho 
\\
\op{Ch}^q({\op{B}}\op{SL}_{n-1}) \ar[r]_\beta \ar[rd]_{\op{Sq}^2} &\op{H}^{q+1}({\op{B}}\op{SL}_{n-1},\mathbf{I}^{q+1}) \ar[d]_\rho &\op{Ch}^{q+1}({\op{B}}\op{SL}_n) \ar[ld]^{j^\ast}\\
& \op{Ch}^{q+1}({\op{B}}\op{SL}_{n-1})
}
\]
where the triangles are given by Totaro's identification $\rho\beta=\op{Sq}^2$, cf. Proposition~\ref{prop:sq2}. The morphism between the two triangles is the restriction map from ${\op{B}}\op{SL}_n$ to ${\op{B}}\op{SL}_{n-1}$, which commutes with all maps involved. Consider a class $\alpha\in\op{Ch}^q({\op{B}}\op{SL}_n)$ with $\op{Sq}^2\alpha=0$. Our claim is that $\beta\alpha=0$. Restricting to ${\op{B}}\op{SL}_{n-1}$, we get a class $j^\ast\alpha$ with $\op{Sq}^2 j^\ast\alpha=0$. The assumption implies that $\beta j^\ast \alpha=j^\ast\beta\alpha=0$. 

Now we consider the localization sequence, cf. Proposition~\ref{prop:locsln}:
\[
\op{H}^{q+1-n}({\op{B}}\op{SL}_n,\mathbf{I}^{q+1-n})\xrightarrow{\op{e}_n} 
\op{H}^{q+1}({\op{B}}\op{SL}_n,\mathbf{I}^{q+1}) \xrightarrow{j^\ast} 
\op{H}^{q+1}({\op{B}}\op{SL}_{n-1},\mathbf{I}^{q+1}).
\]
We have established above that $j^\ast\beta\alpha=0$. As a consequence, there exists a class $\sigma\in\op{H}^{q+1-n}({\op{B}}\op{SL}_n,\mathbf{I}^{q+1-n})$ such that $\op{e}_n\sigma=\beta\alpha$, and we wish to show that $\sigma$ necessarily is trivial. What we have shown so far is that for classes $\alpha\in\op{Ch}^\bullet({\op{B}}\op{SL}_n)$ with $\op{Sq}^2(\alpha)=0$, the classes $\beta\alpha$ are in the image of the multiplication with the Euler class.

Next, we will show that $\rho$ is injective on the image of $\beta$ by an induction on the cohomological degree. The base case for the induction is given by vanishing of the cohomology in negative degrees. Now assume that the reduction map $\rho$ is injective on the image of $\beta$ in degrees less than $q$. Let $\alpha\in\op{Ch}^q({\op{B}}\op{SL}_n)$ be a class with $\op{Sq}^2(\alpha)=0$. As discussed above, there exists a class $\sigma\in\op{H}^{q+1-n}({\op{B}}\op{SL}_n,\mathbf{I}^{q+1-n})$ such that $\op{e}_n\sigma=\beta\alpha$. 

We claim that $\sigma$ is in the image of $\beta\colon \op{Ch}^{q-n}({\op{B}}\op{SL}_n)\to \op{H}^{q+1-n}({\op{B}}\op{SL}_n,\mathbf{I}^{q+1-n})$. If it is not, then by the remark after Proposition~\ref{prop:evensplit} it contains a nontrivial monomial in the Pontryagin and Euler classes. By Proposition~\ref{prop:evensplit}, $\op{e}_n\sigma$ cannot be $\op{W}(F)$-torsion, contradicting the equality $\op{e}_n\sigma=\beta(\alpha)$.
So $\sigma$ is in fact in the image of $\beta$, hence we can apply our inductive assumption that $\rho$ is  injective on the image of $\beta$ in degrees less than $q$ (which contains $\sigma$).

Now consider the diagram 
\[
\xymatrix{
\op{H}^{q+1-n}({\op{B}}\op{SL}_n,\mathbf{I}^{q+1-n}) \ar[d]_\rho \ar[r]^{\op{e}_n} & \op{H}^{q+1}({\op{B}}\op{SL}_n,\mathbf{I}^{q+1}) \ar[d]^\rho \\
\op{Ch}^{q+1-n}({\op{B}}\op{SL}_n) \ar[r]_{\overline{\op{c}}_{n}} & \op{Ch}^{q+1}({\op{B}}\op{SL}_n).
}
\]
The vertical morphisms are induced from $\mathbf{I}^m\to\mathbf{K}^{\op{M}}_m/2$. The horizontal morphisms are essentially push-forward maps in the localization sequence. Commutativity follows from Proposition~\ref{prop:chern}. 

Now by our assumption $\rho\op{e}_n\sigma=\rho\beta\alpha=\op{Sq}^2\alpha=0$. Commutativity of the previous diagram implies that $\overline{\op{c}}_n\rho\sigma=0$. But $\overline{\op{c}}_n$ is injective since $\op{Ch}^\bullet({\op{B}}\op{SL}_n)$ is a polynomial ring generated by the Chern classes, cf. Proposition~\ref{prop:chern}. This implies $\rho\sigma=0$, and the inductive assumption implies $\sigma=0$. Therefore, $\beta\alpha=0$ as claimed.
\end{proof}

The above statement now allows to deal with the torsion phenomena in the cohomology ring $\op{H}^\bullet({\op{B}}\op{SL}_n,\mathbf{I}^\bullet)$. As a first direct application, we can establish that all the relations claimed in Theorem~\ref{thm:slnin} are indeed satisfied.

\begin{corollary}
\label{cor:welldefslneven}
Under our inductive assumption, the morphism $\theta_n\colon \mathscr{R}_n\to \op{H}^\bullet({\op{B}}\op{SL}_n,\mathbf{I}^\bullet)$ maps $\mathscr{I}_n$ to zero and therefore factors through a  well-defined surjective morphism 
\[
\Xi_n\colon \mathscr{R}_n/\mathscr{I}_n\to \op{H}^\bullet({\op{B}}\op{SL}_n,\mathbf{I}^\bullet).
\]
\end{corollary}

\begin{proof}
All the relations in the even case involve torsion. The type (1) relations are satisfied by Lemma~\ref{lem:baertorsion}. The type (3) relations are satisfied in $\op{Ch}^\bullet({\op{B}}\op{SL}_n)$ by Corollary~\ref{cor:slninmod2}. Note that the type (3) relations are in the image of $\beta$. In particular, by Proposition~\ref{prop:lem22even}, any such relation is mapped to $0$ already in $\op{H}^\bullet({\op{B}}\op{SL}_n,\mathbf{I}^\bullet)$. This proves the claim.
\end{proof}

Another direct consequence is that we can finally state that multiplication with the Euler class is injective on the full cohomology ring $\op{H}^\bullet({\op{B}}\op{SL}_n,\mathbf{I}^\bullet)$:

\begin{corollary}
\label{cor:eulerinj}
Assume that we know Theorem~\ref{thm:slnin} for ${\op{B}}\op{SL}_{n-1}$. For all $q$, the map
\[
\op{e}_n\colon \op{H}^{q-n}({\op{B}}\op{SL}_n,\mathbf{I}^{q-n})\to \op{H}^{q}({\op{B}}\op{SL}_n,\mathbf{I}^{q})
\]
is injective. 
\end{corollary}

\begin{proof}
By Proposition~\ref{prop:evensplit}, we already know that the kernel of the Euler class must be contained in the image of $\beta\colon \op{Ch}^{q-1-n}({\op{B}}\op{SL}_n)\to \op{H}^{q-n}({\op{B}}\op{SL}_n,\mathbf{I}^{q-n})$. Let $\alpha$ be a class in the image of $\beta$, and consider the commutative square that appeared in the proof of Proposition~\ref{prop:lem22even}:
\[
\xymatrix{
\op{H}^{q-n}({\op{B}}\op{SL}_n,\mathbf{I}^{q-n}) \ar[d]_\rho \ar[r]^{\op{e}_n} & \op{H}^{q}({\op{B}}\op{SL}_n,\mathbf{I}^{q}) \ar[d]^\rho \\
\op{Ch}^{q-n}({\op{B}}\op{SL}_n) \ar[r]_{\overline{\op{c}}_{n}} & \op{Ch}^{q}({\op{B}}\op{SL}_n).
}
\]
If $\alpha$ is in the kernel of $\op{e}_n$, then it maps to $0$ in the lower right. Now $\overline{\op{c}}_n$ is injective because $\op{Ch}^\bullet({\op{B}}\op{SL}_n)$ is a polynomial ring generated by the reductions of Chern classes, cf.~Proposition~\ref{prop:sw}. By Proposition~\ref{prop:lem22even}, the left-hand vertical reduction map is injective on the image of $\beta$, implying that the class $\alpha$ we started with is already $0$. This shows the required injectivity.
\end{proof}

\begin{proposition}
Assume that we know Theorem~\ref{thm:slnin} for ${\op{B}}\op{SL}_{n-1}$. Then the presentation morphism $\Xi_n\colon \mathscr{R}_n/\mathscr{I}_n\to\op{H}^\bullet({\op{B}}\op{SL}_n,\mathbf{I}^\bullet)$ of Corollary~\ref{cor:welldefslneven} is injective, i.e., there are no more relations. 
\end{proposition}

\begin{proof}
Injectivity of $\Xi_n$ can now be established by an induction over $q$. The base case is given  by the vanishing of the kernel in negative degrees. Consider the following commutative diagram:
\[
\xymatrix{
(\mathscr{R}_n/\mathscr{I}_n)_{q-n} \ar[r]^{X_n} \ar[d]_{\Xi_n} & (\mathscr{R}_n/\mathscr{I}_n)_{q} \ar[r]^{\Phi_n} \ar[d]_{\Xi_n} & (\mathscr{R}_{n-1}/\mathscr{I}_{n-1})_q \ar[d]^{\Xi_{n-1}}\\
\op{H}^{q-n}({\op{B}}\op{SL}_n,\mathbf{I}^{q-n}) \ar[r]_{\op{e}_n} & \op{H}^{q}({\op{B}}\op{SL}_n,\mathbf{I}^{q}) \ar[r]_{j^\ast} & 
\op{H}^{q}({\op{B}}\op{SL}_{n-1},\mathbf{I}^{q})
}
\]
By the inductive assumptions, $\Xi_n$ is injective on the degree $q-n$ part of $\mathscr{R}_n/\mathscr{I}_n$ and $\Xi_{n-1}$ is an isomorphism. 
Moreover, $\op{e}_n$ is injective by Corollary~\ref{cor:eulerinj}. The conclusion follows from the $5$-lemma. 
\end{proof}

\subsection{Induction step, odd case}

Now suppose that $n$ is odd. We are still in the computation of $\op{H}^\bullet({\op{B}}\op{SL}_n,\mathbf{I}^\bullet)$ via induction on $n$. \emph{So we assume throughout the whole next subsection that Theorem~\ref{thm:slnin} holds for $\op{H}^\bullet({\op{B}}\op{SL}_{n-1},\mathbf{I}^\bullet)$.}

In the even case, things were easier because the restriction map $\Phi\colon \mathscr{R}_n/\mathscr{I}_n\to\mathscr{R}_{n-1}/\mathscr{I}_{n-1}$ was surjective. This is not the case for $n$ odd. By Lemma~\ref{lem:oddproblem}, the problematic class is the Euler class $\op{X}_{n-1}$ whose square is the image of the Pontryagin class $P_{n-1}$ but which is not itself the image of any class from $\mathscr{R}_n/\mathscr{I}_n$. This requires establishing more precise information on the behaviour of characteristic classes under the restriction map $j^\ast\colon \op{H}^\bullet({\op{B}}\op{SL}_n,\mathbf{I}^\bullet)\to\op{H}^\bullet({\op{B}}\op{SL}_{n-1},\mathbf{I}^\bullet)$. The other major difference to the even case is that multiplication by the Euler class is no longer injective because the Euler class is torsion, cf. Proposition~\ref{prop:eulerrel}.

To work our way towards showing that $\theta_n$ is surjective, we need to compare the image of the restriction map to the image of the presentation map.

\begin{proposition}
\label{prop:equalim}
Assume we know Theorem~\ref{thm:slnin} for ${\op{B}}\op{SL}_{n-1}$. In the composition 
\[
\mathscr{R}_n\xrightarrow{\theta_n} \op{H}^\bullet({\op{B}}\op{SL}_n,\mathbf{I}^\bullet)\xrightarrow{j^\ast} \op{H}^\bullet({\op{B}}\op{SL}_{n-1},\mathbf{I}^\bullet)
\]
we have $\op{Im} j^\ast\theta_n=\op{Im} j^\ast$. 
\end{proposition}

\begin{proof}
The setup for the proof is very similar to the one in \cite[p.286]{brown}: for the proof, we consider the following diagram
\[
\xymatrix{
(\mathscr{R}_n)_q \ar[r]^{\theta_n} \ar[d]_{\Phi_n} & \op{H}^q({\op{B}}\op{SL}_n,\mathbf{I}^q) \ar[d]^{j^\ast} \\
(\mathscr{R}_{n-1})_q \ar[r]_{\theta_{n-1}} & \op{H}^q({\op{B}}\op{SL}_{n-1},\mathbf{I}^q) \ar[d]^\partial \\
& \op{H}^{q+1}_{{\op{B}}\op{SL}_n}(E_n,\mathbf{I}^{q}) \ar[r]_\cong& \op{H}^{q+1-n}({\op{B}}\op{SL}_n,\mathbf{I}^{q-n})
}
\]
Here the middle vertical column is a piece of the localization sequence from Proposition~\ref{prop:locsln}; note that the exactness implies that the image of $j^\ast$ is the kernel of $\partial$. The left-hand square is commutative by Proposition~\ref{prop:backface}. The isomorphism in the lower-right corner of the diagram is the d\'evissage isomorphism.

By the inductive assumption on ${\op{B}}\op{SL}_{n-1}$ and  Lemma~\ref{lem:oddproblem}, we can write any element in $\op{H}^\bullet({\op{B}}\op{SL}_{n-1},\mathbf{I}^\bullet)$ as a sum $j^\ast \theta_n(u_1) +\op{e}_{n-1} j^\ast \theta_n(u_2)$. We first compute the image of $\op{e}_{n-1}$ under the boundary map $\partial\colon \op{H}^{n-1}({\op{B}}\op{SL}_{n-1},\mathbf{I}^{n-1})\to \op{H}^{0}({\op{B}}\op{SL}_n,\mathbf{I}^{-1})$. From the Gersten resolution, we know that $\eta\colon \op{H}^{0}({\op{B}}\op{SL}_n,\mathbf{W})\to \op{H}^{0}({\op{B}}\op{SL}_n,\mathbf{I}^{-1})$ is surjective because $\eta\colon \op{W}(E)\to\op{I}^{-1}(E)$ is an isomorphism for fields. Note that $\op{H}^0({\op{B}}\op{SL}_n,\mathbf{W})\cong \op{W}(F)$ was established in the proof of Theorem~\ref{thm:symplectic} for the case $n=2$; for $n\geq 3$ it follows inductively using that the stabilization sequences induce isomorphisms $\op{H}^0({\op{B}}\op{SL}_{n-1},\mathbf{W})\cong \op{H}^0({\op{B}}\op{SL}_n,\mathbf{W})$. In particular, we can view $\op{H}^0({\op{B}}\op{SL}_n,\mathbf{I}^{-1})$ as a cyclic $\op{W}(F)$-submodule generated by $\eta(1)$; consequently, $\partial\op{e}_{n-1}=\eta(u)$ for some element $u\in\op{H}^0({\op{B}}\op{SL}_n,\mathbf{W})\cong\op{W}(F)$. By the previous observations, the inductive assumption on ${\op{B}}\op{SL}_{n-1}$ and  Lemma~\ref{lem:oddproblem} imply that the image of $\partial\colon \op{H}^{n-1}({\op{B}}\op{SL}_{n-1},\mathbf{I}^{n-1})\to \op{H}^{0}({\op{B}}\op{SL}_n,\mathbf{I}^{-1})$ is a cyclic $\op{W}(F)$-submodule of $\op{H}^0({\op{B}}\op{SL}_n,\mathbf{I}^{-1})$ generated by $\partial\op{e}_{n-1}$. Consider the following commutative diagram: 
\[
\xymatrix{
&\op{H}^{0}({\op{B}}\op{SL}_n,\mathbf{W}) \ar[d]^\eta \ar[r]^{\op{e}_n} & \op{H}^{n}({\op{B}}\op{SL}_n,\mathbf{I}^{n}) \ar[d]^\eta
\\
\op{H}^{n-1}({\op{B}}\op{SL}_{n-1},\mathbf{I}^{n-1}) \ar[r]_\partial &  \op{H}^{0}({\op{B}}\op{SL}_n,\mathbf{I}^{-1}) \ar[r]_{\op{e}_n} & \op{H}^{n}({\op{B}}\op{SL}_n,\mathbf{I}^{n-1})
}
\]
Since $\op{e}_n=\beta(\overline{\op{c}}_{n-1})$, the image of $\op{e}_n$ under $\eta\colon \op{H}^n({\op{B}}\op{SL}_n,\mathbf{I}^n)\to \op{H}^n({\op{B}}\op{SL}_n,\mathbf{I}^{n-1})$ will be trivial, by the B\"ar sequence (\ref{eq:baer}). So the composition 
\[
{\op{e}_n}\circ\eta\colon \op{H}^0({\op{B}}\op{SL}_n,\mathbf{W})\to \op{H}^n({\op{B}}\op{SL}_n,\mathbf{I}^{n-1})
\]
is trivial. In particular,  $\op{e}_n\colon \op{H}^0({\op{B}}\op{SL}_n,\mathbf{I}^{-1})\to\op{H}^n({\op{B}}\op{SL}_n,\mathbf{I}^{n-1})$ is the zero map, because every element of $\op{H}^0({\op{B}}\op{SL}_n,\mathbf{I}^{-1})$ is in the image of $\eta$. Hence the boundary map 
\[
\partial\colon \op{H}^{n-1}({\op{B}}\op{SL}_{n-1},\mathbf{I}^{n-1})\to\op{H}^0({\op{B}}\op{SL}_n,\mathbf{I}^{-1})
\]
is surjective. Now we combine all the above statements:  $\op{H}^0({\op{B}}\op{SL}_n,\mathbf{I}^{-1})$ is a cyclic $\op{W}(F)$-module, and both $\eta(1)$ and $\partial(\op{e}_{n-1})$ are generators. In particular, in our previous identification $\partial(\op{e}_{n-1})=\eta(u)$, the element $u\in\op{W}(F)$ must be a unit.

Via the derivation property of the boundary map $\partial$, cf. Proposition~\ref{prop:gysinder}, we can now compute the boundary of an arbitrary element in $\op{H}^\bullet({\op{B}}\op{SL}_{n-1},\mathbf{I}^\bullet)$:
\[
\partial(j^\ast \theta_n(u_1) +\op{e}_{n-1} j^\ast \theta_n(u_2))= \partial(\op{e}_{n-1}) \theta_n(u_2) = u\eta\theta_n(u_2).
\]
In particular, the  boundary of such an element is trivial if and only if $\eta \theta_n(u_2)=0$. By the B\"ar sequence (\ref{eq:baer}), this means that $\theta_n(u_2)$ is in the image of the Bockstein and by linearity it suffices to show that elements of the form $\beta(\overline{\op{c}}_{2j_1}\cdots\overline{\op{c}}_{2j_l})\op{e}_{n-1}$ can be lifted. 
But now, using the inductive assumption on ${\op{B}}\op{SL}_{n-1}$, it suffices to do the computation in $\op{Ch}^\bullet({\op{B}}\op{SL}_{n-1})$. Since $\op{Sq}^2$ is a derivation, we have 
\[
\op{Sq}^2(\overline{\op{c}}_{2j_1}\cdots\overline{\op{c}}_{2j_l})\op{e}_{n-1} = 
\op{Sq}^2(\overline{\op{c}}_{2j_1}\cdots\overline{\op{c}}_{2j_l} \rho(\op{e}_{n-1})) = \op{Sq}^2(\overline{\op{c}}_{2j_1}\cdots\overline{\op{c}}_{2j_l} \overline{\op{c}}_{n-1}).
\]
Since $\theta_{n-1}$ is surjective, $\beta(\overline{\op{c}}_{2j_1}\cdots\overline{\op{c}}_{2j_l})\op{e}_{n-1}$ is in the image of $\theta_{n-1}$. By Lemma~\ref{lem:oddproblem} and the previous computation, it is in fact in the image of $\theta_{n-1}\Phi_n$, hence in the image of $j^\ast\theta_n$. 
\end{proof}

\begin{corollary}
\label{cor:surjodd}
Assume we know Theorem~\ref{thm:slnin} for ${\op{B}}\op{SL}_{n-1}$. Then the homomorphism $\theta_n\colon \mathscr{R}_n\to\op{H}^\bullet({\op{B}}\op{SL}_n,\mathbf{I}^\bullet)$ is surjective. 
\end{corollary}

\begin{proof}
The proof is by induction on the cohomological degree $q$, pretty much like for Proposition~\ref{prop:surjeven}. The base case is established by vanishing of cohomology in negative degrees. Assume we have established surjectivity in degrees less than $q$, and let $\alpha\in\op{H}^q({\op{B}}\op{SL}_n,\mathbf{I}^q)$ be a class. Then $j^\ast\alpha$ is in the image of $j^\ast\theta_n$, by Proposition~\ref{prop:equalim}. Hence we can assume without loss of generality that $j^\ast\alpha=0$. Then exactness of the localization sequence, cf. Proposition~\ref{prop:locsln}, implies that $\alpha$ is in the image of multiplication with the Euler class, i.e., there is a class $\gamma\in\op{H}^{q-n}({\op{B}}\op{SL}_n,\mathbf{I}^{q-n})$ such that $\op{e}_n\gamma=\alpha$. By the inductive assumption, there is an element $\gamma'\in\mathscr{R}_n/\mathscr{I}_n$ such that $\theta_n\gamma'=\gamma$. But then $\theta_n(X_n\gamma')=\op{e}_n\gamma=\alpha$, proving the claim.
\end{proof}

\begin{proposition}
\label{prop:oddsplit}
Assume that we know Theorem~\ref{thm:slnin} for ${\op{B}}\op{SL}_{n-1}$. Denote by $\mathscr{P}_n$ the $W(F)$-subalgebra of $\mathscr{R}_n$ generated by the elements $P_1,\dots,P_{(n-1)/2}$. Then the following composition is injective
\[
\mathscr{P}_n\xrightarrow{\theta_n|_{\mathscr{P}_n}}\op{H}^\bullet({\op{B}}\op{SL}_n,\mathbf{I}^\bullet) \xrightarrow{j^\ast} \op{H}^\bullet({\op{B}}\op{SL}_{n-1},\mathbf{I}^\bullet). 
\]
In particular, there are no relations between Pontryagin classes in $\op{H}^\bullet({\op{B}}\op{SL}_n,\mathbf{I}^\bullet)$. 
\end{proposition}

\begin{proof}
We note that $n$ is odd, hence $n-1$ is even and $(n-1)/2$ is an integer. By Proposition~\ref{prop:compatres}, the elements $P_i$ map to $\op{p}_{2i}$ for $i<(n-1)/2$, and $P_{(n-1)/2}$ maps to $\op{e}_{n-1}^2$. By the inductive assumption, the subalgebra of $\op{H}^\bullet({\op{B}}\op{SL}_{n-1},\mathbf{I}^\bullet)$ generated by the Pontryagin classes and the Euler class is a polynomial ring over $\op{W}(F)$ on these generators, cf. also Proposition~\ref{prop:evensplit}. In particular, the Pontryagin classes for ${\op{B}}\op{SL}_n$ are algebraically independent, and the map is injective as claimed. 
\end{proof}

\begin{proposition}
\label{prop:lem22odd}
Assume that we know Theorem~\ref{thm:slnin} for ${\op{B}}\op{SL}_{n-1}$. Then for all degrees $q$, in the composition
\[
\op{Ch}^q({\op{B}}\op{SL}_n)\xrightarrow{\beta} \op{H}^{q+1}({\op{B}}\op{SL}_n,\mathbf{I}^{q+1})\xrightarrow{\rho} \op{Ch}^{q+1}({\op{B}}\op{SL}_n),
\]
the map $\rho$ is injective on the image of $\beta$. 
\end{proposition}

\begin{proof}
The argument proceeds along the lines of the proof of Proposition~\ref{prop:lem22even}; it is an induction on the cohomological degree. The base case is again vanishing in negative degrees. So assume that we know that $\rho$ is injective on the image of $\beta$ in degrees less than $q$. The first part of the argument of Proposition~\ref{prop:lem22even} can be copied verbatim to prove the following: if $\alpha\in\op{Ch}^q({\op{B}}\op{SL}_n)$ is  a class with $\op{Sq}^2(\alpha)=0$, then there exists a class $\sigma\in\op{H}^{q+1-n}({\op{B}}\op{SL}_n,\mathbf{I}^{q+1-n})$ such that $\op{e}_n\sigma=\beta(\alpha)$. We want to show that $\op{e}_n\sigma$ is trivial, giving injectivity of $\rho$ on the image of $\beta$. 

Now there is a crucial difference to the even-rank case. The class $\sigma$ need not be in the image of $\beta$ because the Euler class is a torsion class. Assuming that $\sigma$ is in the image of $\beta$, then the rest of the argument for Proposition~\ref{prop:lem22even} goes through to show that $\sigma$ is in fact trivial. From Corollary~\ref{cor:surjodd} we deduce that if $\sigma$ is not in the image of $\beta$, then it contains a nontrivial monomial in the Pontryagin classes, as in the proof of Proposition~\ref{prop:lem22even}. 
By our assumption, $\rho{\op{e}_n}\sigma=\rho\beta\alpha=\op{Sq}^2\alpha=0$. By Proposition~\ref{prop:pontchern}, $\op{p}_{2i}$ maps to $\overline{\op{c}}_{2i}^2$. Recall that $\op{Ch}^\bullet({\op{B}}\op{SL}_n)$ is a polynomial algebra and that $\rho$ is a ring homomorphism. Combining all these statements, if $\sigma$ is not in the image of $\beta$, then the monomials in the Pontryagin classes have their coefficients in $\op{I}(F)$. By Proposition~\ref{prop:eulerrel}, the Euler class is in the image of $\beta$, and Lemma~\ref{lem:baertorsion} implies that $\op{e}_n$ is annihilated by $\op{I}(F)$. In particular, while $\sigma$ itself is non-trivial, its image $\op{e}_n\sigma=\beta(\alpha)$ is trivial, finishing the proof.
\end{proof}

\begin{corollary}
\label{cor:welldefslnodd}
Assume we know Theorem~\ref{thm:slnin} for ${\op{B}}\op{SL}_{n-1}$. Then the relations in $\mathscr{I}_n$ are satisfied in the $\mathbf{I}^\bullet$-cohomology, i.e., the homomorphism $\theta_n\colon \mathscr{R}_n\to \op{H}^\bullet({\op{B}}\op{SL}_n,\mathbf{I}^\bullet)$ factors 
\[
\mathscr{R}_n\to\mathscr{R}_n/\mathscr{I}_n\xrightarrow{\Xi_n} \op{H}^\bullet({\op{B}}\op{SL}_n,\mathbf{I}^\bullet).
\]
\end{corollary}

\begin{proof}
The relation (2) has been verified in Proposition~\ref{prop:eulerrel}. 
All other relations involve torsion elements. The type (1) relations are satisfied by Lemma~\ref{lem:baertorsion}. The type (3) relations are satisfied in $\op{Ch}^\bullet({\op{B}}\op{SL}_n)$ by Corollary~\ref{cor:slninmod2}. Note that the type (3) relations are in the image of $\beta$. In particular, by Proposition~\ref{prop:lem22odd}, any such relation is mapped to $0$ already in $\op{H}^\bullet({\op{B}}\op{SL}_n,\mathbf{I}^\bullet)$. This proves the claim.
\end{proof}

\begin{proposition}
Assume we know Theorem~\ref{thm:slnin} for ${\op{B}}\op{SL}_{n-1}$. Then the homomorphism $\Xi_n\colon \mathscr{R}_n/\mathscr{I}_n\to \op{H}^\bullet({\op{B}}\op{SL}_n,\mathbf{I}^\bullet)$ is an isomorphism. 
\end{proposition}

\begin{proof}
Surjectivity of the morphism has already been established in Corollary~\ref{cor:surjodd}. Injectivity on the torsion-free part has been established in Proposition~\ref{prop:oddsplit}. So we need to show injectivity on the torsion part. By Proposition~\ref{prop:lem22odd}, it suffices to show the statement after reduction to  $\op{Ch}^\bullet({\op{B}}\op{SL}_n)$. 

To conclude $\Xi_n$ is an isomorphism, we essentially follow the strategy of \cite[p.286/287]{brown}. However, as we are unable to prove the relation
between $\op{Sq}^1\op{w}^k$ and $\op{w}^{\overline{k}}$ in loc. cit., we replace this part of the argument by a simplification of the one provided by \cite[p.285]{cadek} for ${\op{B}}\op{O}(n)$ with both orientations considered simultaneously. For the translation to our motivic but otherwise easier setting corresponding to ${\op{B}}\op{SO}(n)$ with just the trivial orientation, set $\op{w}_1=1$, replace $\op{w}_{i}$ by $\overline{\op{c}}_i$ and $\op{Sq}^1$ by $\op{Sq}^2$ as before. Hence we only need to look at the right summands in the products of the first type ($l=0$) in loc. cit. Moreover all the $x_{i}$ may be suppressed
in loc. cit., and the $y_I$ correspond to our $B_I$. Now simplifying to products of the form
\[
\prod \overline{\op{c}}_{2i}^{2m_i} \prod \op{Sq}^2(B_J)
\]
and using that there are no relations between the various $\overline{\op{c}}_{2s}^{2t}$, we are reduced to show that for a given monomial $\prod \overline{\op{c}}_{2i}^{2m_i}$ the corresponding $\op{Sq}^2(B_J)$ may not be combined to zero in a nontrivial way. However, this easily follows using iterations of the Leibniz rule for $\op{Sq}^2(XY)$ together with the formula $\op{Sq}^2(\overline{\op{c}}_{2i})=\overline{\op{c}}_{2i+1}$
and the fact that $\op{Ch}^\bullet({\op{B}}\op{SL}_n)$ is a free $\mathbb{Z}/2\mathbb{Z}$-vector space with a basis given by all possible monomials in the $\overline{\op{c}}_i$, hence any subfamily of these is linear independent.
\end{proof}

Finally, the following statement explains why only the even Pontryagin classes appear in the presentation of $\op{H}^\bullet({\op{B}}\op{SL}_n,\mathbf{I}^\bullet)$ of Theorem~\ref{thm:slnin}. It is an analogue of the classical fact that the odd Chern classes of complexifications are 2-torsion. Note that Proposition~\ref{prop:pontchern} implies that the odd Pontryagin classes are non-torsion classes in Chow--Witt theory. 

\begin{proposition}
\label{prop:ponttor}
If $i$ is an odd natural number, then the Pontryagin class $\op{p}_i\in\op{H}^\bullet({\op{B}}\op{SL}_n,\mathbf{I}^\bullet)$ is $\op{I}(F)$-torsion.
\end{proposition}

\begin{proof}
We first recall the classical argument: if $\mathscr{E}\to X$ is an oriented vector bundle over a smooth scheme $X$, then its symplectification $\sigma_n(\mathscr{E})$ is invariant under conjugation. In particular, using Proposition~\ref{prop:sympconj}, we have 
\[
\op{p}_i(\sigma_n(\mathscr{E}))= \op{p}_i(\overline{\sigma_n(\mathscr{E})})= \langle-1\rangle^i\op{p}_i(\sigma_n(\mathscr{E})).
\]
Note that in $\mathbf{I}$-cohomology, $\langle-1\rangle=-1$.
This implies that, for odd $i$, the Pontryagin class $\op{p}_i$ is annihilated by $2=1-\langle-1\rangle$. But in our situation, this is not enough -- we still need to show that the odd Pontryagin classes are annihilated by all of $\op{I}(F)$.

For the top Pontryagin class $\op{p}_n$ of $\op{SL}_n$, this follows from Proposition~\ref{prop:eulersquare} together with the corresponding torsion claim for the Euler class in Proposition~\ref{prop:eulerrel}.

To prove the claim for $\op{p}_i$ and all other $\op{SL}_n$, $i$ odd and $n>i$, we will use induction on $n$. We first remark that the Pontryagin classes are compatible with stabilization, cf. Proposition~\ref{prop:pontryaginstab}.  The $\mathbf{I}^n$-cohomology analogue of Proposition~\ref{prop:locsln} yields an exact localization sequence
\[
\op{H}^{2i-n}({\op{B}}\op{SL}_n,\mathbf{I}^{2i-n})\xrightarrow{\op{e}_n} \op{H}^{2i}({\op{B}}\op{SL}_n,\mathbf{I}^{2i})\xrightarrow{j^\ast} \op{H}^{2i}({\op{B}}\op{SL}_{n-1},\mathbf{I}^{2i})
\]
The restriction $j^\ast$ maps the Pontryagin class $\op{p}_{i}$ for $\op{SL}_n$ to the one for $\op{SL}_{n-1}$. By induction, we may assume that $\op{p}_{i}$ is $\op{I}(F)$-torsion in the last group of the sequence. Hence the element $a\op{p}_{i} $ in $\op{H}^{2i}({\op{B}}\op{SL}_n,\mathbf{I}^{2i})$ will be divisible by $\op{e}_n$ for any $a\in\op{I}(F)$. For odd $n$ this implies that the element $a\op{p}_i$ is $\op{I}(F)$-torsion by Proposition~\ref{prop:eulerrel}. For even $n$, by induction on the cohomological degree $i$, we may assume that the class $\op{e}_n^{-1}(a\op{p}_{i})$ in $\op{H}^{2i-n}({\op{B}}\op{SL}_n,\mathbf{I}^{2i-n})$ is either $\op{I}(F)$-torsion or contains a nontrivial monomial in the even Pontryagin classes as a summand. In the latter case, it cannot possibly be annihilated by any element from $\op{I}(F)$. Since we know that it is annihilated by $2=1-\langle-1\rangle$ by the classical argument, this case cannot occur. So we know that $a\op{p}_{i}$ is the product of the Euler class with something annihilated by $\op{I}(F)$. We can check by reduction to $\op{Ch}^\bullet$ that $\op{e}_n^{-1}(a\op{p}_{i})$ must already be zero, showing that $\op{p}_i$ is $\op{I}(F)$-torsion. By induction, we have proved the claim.
\end{proof}

\begin{remark}
Note that Proposition~\ref{prop:pontchern} implies  
\[
\rho(\op{p}_{2k+1})=\overline{\op{c}}_{2k+1}^2= \op{Sq}^2(\overline{\op{c}}_{2k})\overline{\op{c}}_{2k+1}= \op{Sq}^2(\overline{\op{c}}_{2k}\overline{\op{c}}_{2k+1}). 
\]
In particular, once we know that $\op{p}_{2k+1}$ is in the image of the Bockstein map, the following classical identification  will follow from Theorem~\ref{thm:slnin}:
\[
\op{p}_{2k+1}=\beta(\overline{\op{c}}_{2k}\overline{\op{c}}_{2k+1}).
\]
\end{remark}

\section{Euler classes for odd-rank modules}
\label{sec:bs}

In this section, we discuss some consequences of the computations of the Chow--Witt rings of ${\op{B}}\op{SL}_n$. Recall the following question, which was formulated in \cite[Question 7.12]{bhatwadekar:sridharan}: given a scheme $X$ of odd dimension $n$ and a rank $n$ vector bundle $\mathscr{E}$ on $X$, does the vanishing of the top Chern class $\op{c}_n(\mathscr{E})$
suffice to conclude that the vector bundle splits off a free rank one summand? 

While this question was originally posed and partially answered using Euler class groups, the question can, by the work of Morel, Asok and Fasel, also be studied from an $\mathbb{A}^1$-homotopical viewpoint if $F$ is a perfect field of characteristic unequal to $2$ and $X$ is a smooth affine scheme over $F$. Actually, \cite{AsokFaselCohomotopy} showed that the two approaches to the vector bundle splitting question, the Euler class groups appearing in \cite{bhatwadekar:sridharan} and the Chow--Witt groups, are equivalent. 

The answer to the question of Bhatwadekar and Sridharan is positive over $\mathbb{C}$ by Murthy's theorem. In the Chow--Witt-theoretic perspective, this follows immediately from the fact that the natural map $\widetilde{\op{CH}}^n(X)\to\op{CH}^n(X)$ is an isomorphism whenever $X$ is smooth of dimension $n$ over a quadratically closed field. 

The answer is also known to be positive over $\mathbb{R}$ and more generally over real closed fields, by \cite{bhatwadekar:sridharan:99,bhatwadekar:das:mandal,bhatwadekar:sane}. In the context of Chow--Witt groups, we can reprove one of the main results of \cite{bhatwadekar:sridharan:99}, namely the odd rank case of Theorem~II of loc.cit:

\begin{proposition}
\label{prop:recoverbs}
Let $F=\mathbb{R}$, and let $X$ be a smooth affine variety of odd dimension $n$ over $F$ which is orientable, i.e., whose canonical module is trivial. Then a rank $n$ vector bundle $\mathscr{E}$ over $X$ splits off a free rank one summand if $\op{c}_n(\mathscr{E})=0$.
\end{proposition}

\begin{proof}
By the representability theorem for vector bundles, cf. \cite{MField} or \cite{gbundles}, \cite[Question 7.12]{bhatwadekar:sridharan} translates into a question concerning $\mathbb{A}^1$-homotopy classes of maps $X\to{\op{B}}\op{SL}_n$. By the obstruction-theoretic approach to vector bundle classification, cf. \cite{MField,AsokFasel}, splitting of the vector bundle is equivalent to the vanishing of the Euler class $\op{e}_n(\mathscr{E})\in\widetilde{\op{CH}}^n(X)$. The obstruction-theoretic Euler class is identified with the Chow--Witt-theoretic Euler class in $\widetilde{\op{CH}}^n(X)$ by  \cite[Theorem 1]{AsokFaselEuler}. 


By \cite[Theorem 16.3.8]{fasel:memoir}, we have $\op{H}^n(X,\mathbf{I}^n)\cong\bigoplus_{\mathcal{C}}\mathbb{Z}$ where $\mathcal{C}$ denotes the set of compact connected components of $X(\mathbb{R})$. The same assertion holds for $\op{H}^n(X,\mathbf{I}^{n+1})$ by \cite[Proposition 5.1]{fasel:orbits}. The proof in loc.cit. also implies that the map $\op{H}^n(X,\mathbf{I}^{n+1})\to \op{H}^n(X,\mathbf{I}^{n})$ is injective. Therefore, by Proposition~\ref{prop:cartesian}, it suffices to show that $\op{e}_n(\mathscr{E})\in \op{H}^n(X,\mathbf{I}^n)$ vanishes. 
For this, recall from Theorem~\ref{thm:slnin} that $\op{e}_n(\mathscr{E})=\beta(\overline{\op{c}}_{n-1})$, in particular the Euler class is torsion. But $\op{H}^n(X,\mathbf{I}^n)$ is torsion-free, so the Euler class is trivial.
\end{proof}

We can identify two main steps in the proof of Proposition~\ref{prop:recoverbs}: the first reduces the vanishing of the Chow--Witt-theoretic Euler class to the vanishing of the Euler class in $\mathbf{I}^n$-cohomology. This is done via the injectivity of $\op{H}^n(X,\mathbf{I}^{n+1})\to\op{H}^n(X,\mathbf{I}^n)$ which in particular implies that the fiber product isomorphism of Proposition~\ref{prop:cartesian} holds. Second, the vanishing of the $\mathbf{I}^n$-cohomological Euler class is established, via explicit knowledge of the group $\op{H}^n(X,\mathbf{I}^n)$. 

\begin{remark}
Real realization induces a map $\op{H}^n(X,\mathbf{I}^m)\to\op{H}^n_{\op{sing}}(X(\mathbb{R}),\mathbb{Z})$ for all varieties $X$ over $\mathbb{R}$. Jacobson \cite{jacobson} has shown that this is an isomorphism for all $n\geq 0$ and all $m> \op{dim} X$. In the case $X={\op{B}}\op{SL}_n$, this map induces an isomorphism $\op{H}^\bullet({\op{B}}\op{SL}_n,\mathbf{I}^\bullet)\cong \op{H}^\bullet_{\op{sing}}(\op{B}\op{SO}(n),\mathbb{Z})$ of graded rings, which maps the algebraic characteristic classes to their topological counterparts. This can be seen by comparing the inductive proofs in Section~\ref{sec:special} to the topological counterparts in \cite{brown}. Alternatively, this follows from the more general results on the real cycle class maps in \cite{4real}. For $X$ smooth, affine and oriented over $\mathbb{R}$, the results from \cite{fasel:memoir} then can be used to show that the vanishing of the $\mathbf{I}^n$-cohomological Euler class is equivalent to the vanishing of the topological Euler class, providing a very satisfactory reinterpretation of the main results of Bhatwadekar--Sridharan in \cite{bhatwadekar:sridharan:99}. 

According to \cite[Theorem 8.5]{jacobson} and \cite{bachmann} there is a natural map 
\[
\op{H}^\bullet({\op{B}}\op{SL}_n,\mathbf{I}^\bullet)\to \op{H}^\bullet_{\mathrm{r\acute{e}t}}({\op{B}}\op{SL}_n,\mathbb{Z})
\]
from $\mathbf{I}$-cohomology to real-\'etale cohomology, induced by localization with respect to $\rho=-[-1]\in\op{K}^{\op{MW}}_1(F)$. We suggest to use this to study vanishing of Euler classes for varieties $X$ over base fields $F$ other than $\mathbb{R}$. 
\end{remark}

\begin{remark}
Note that Proposition~\ref{prop:recoverbs} also holds for real closed fields, using the comparison result \cite[Theorem 3]{AsokFaselCohomotopy} and the corresponding statement for Euler class groups from \cite{bhatwadekar:sridharan:99,bhatwadekar:das:mandal,bhatwadekar:sane}. More generally, it is conceivable that  $\op{H}^n(X,\mathbf{I}^n)$ is torsion-free if $X$ is smooth affine and oriented of dimension $n$ over a real pythagorean field. 
\end{remark}

A more general version of this procedure is given by the following result: 

\begin{proposition}\label{greatapplication}
\label{prop:enodd}
Let $F$ be a field of characteristic $\neq 2$ and let $X$ be a smooth affine variety of odd dimension $n$ over $F$. Assume one of the following:
\begin{enumerate}
\item there is no non-trivial $2$-torsion in $\op{CH}^n(X)$, or 
\item the images of the groups $\op{H}^{n-1}(X,\mathbf{K}^{\op{M}}_n)$ and $\op{H}^{n-1}(X,\mathbf{K}^{\op{M}}_n/2)$ under the Bockstein map $\op{H}^{n-1}(X,\mathbf{K}^{\op{M}}_n/2)\to \op{H}^n(X,\mathbf{I}^{n+1})$ are equal.
\end{enumerate}
Then a vector bundle $\mathscr{E}$ of rank $n$ over $X$ splits off a free rank one summand if and only if $\op{c}_n(\mathscr{E})=0$ and $\beta(\overline{\op{c}}_{n-1}(\mathscr{E}))=0$.
\end{proposition}

\begin{proof}
The reduction to proving triviality of the Chow--Witt-theoretic Euler class is done as in the proof of Proposition~\ref{prop:recoverbs}. Now either of the two conditions imply that the fiber product map $c$ of Proposition~\ref{prop:cartesian} is an isomorphism. This is done in the same way as in the proof of Proposition~\ref{prop:cartesian}, but replaces the injectivity of the maps $2$ and $\eta$, respectively, by the surjectivity of the preceding maps (which are not drawn in the key diagram of Section~\ref{sec:key}). Therefore Theorem~\ref{thm:slnchw} tells us that the Euler class $\op{e}_n(\mathscr{E})\in\widetilde{\op{CH}}^n({\op{B}}\op{SL}_n)$ is completely determined by its image $\op{c}_n(\mathscr{E})\in\op{CH}^n({\op{B}}\op{SL}_n)$ in Chow theory and $\beta(\overline{\op{c}}_{n-1})\in\op{H}^n_{\op{Nis}}({\op{B}}\op{SL}_n,\mathbf{I}^n)$. 
\end{proof}

\begin{remark}
If the classes $\op{c}_n$ and $\beta(\overline{\op{c}}_{n-1})$ vanish, the obstruction to splitting off a free rank one summand is then an element in ${}_2\op{CH}^n(X)$. This suggests that the vanishing implies that the vector bundle splits off a free rank one summand after a degree 2 \'etale covering, and the vector bundle over $X$ splits off a line bundle (which possibly is the norm of a line bundle on the \'etale covering). 
\end{remark}

The two major steps directly translate into two key aspects when studying the question of Bhatwadekar and Sridharan. First, one needs to deal with the potential failure of the fiber product isomorphism of Proposition~\ref{prop:cartesian} by investigating the validity of the conditions in Proposition~\ref{greatapplication}, hopefully for large classes of schemes $X$. Second, we need methods to establish the vanishing of the Euler class in $\mathbf{I}^n$-cohomology. From Theorem~\ref{thm:slnin}, we get $\op{e}_n=\beta(\overline{\op{c}}_{n-1})\in \op{H}^n({\op{B}}\op{SL}_n,\mathbf{I}^n)$,  subject to the compatibility condition 
\[
\rho\circ\beta(\overline{\op{c}}_{n-1})=\op{Sq}^2(\overline{\op{c}}_{n-1})= \overline{\op{c}}_n. 
\]
In particular, we see that the  Euler class $\op{e}_n\in\op{H}^n(X,\mathbf{I}^n)$ for odd rank oriented vector bundles has only ``stable'' information. There are then various possibilities to establish vanishing of this Euler class: the identification $\op{e}_n=\beta(\overline{\op{c}}_{n-1})$ implies that the Euler class is $2$-torsion, and this was enough to deduce vanishing of $\op{e}_n\in \op{H}^n(X,\mathbf{I}^n)$ in the proof of Proposition~\ref{prop:recoverbs} above. Vanishing of $\overline{\op{c}}_{n-1}$ is another way. Finally, a positive answer to the following question would yield a general vanishing statement for the $\mathbf{I}^n$-cohomological Euler class:

\begin{question}
\label{rem:bsnew}
Let $F$ be a field of characteristic $\neq 2$ and let $X$ be smooth affine variety of odd dimension $n$ over $F$. For a class $\sigma\in\op{Ch}^{n-1}(X)$, does $\op{Sq}^2(\sigma)=0$ imply $\beta(\sigma)=0$? In other words, is the reduction $\rho\colon \op{H}^n_{\op{Nis}}(X,\mathbf{I}^n)\to\op{Ch}^n(X)$ injective on the image of 
$\beta\colon \op{Ch}^{n-1}(X)\to\op{H}^n(X,\mathbf{I}^n)$?
\end{question} 

An alternative formulation of Question~\ref{rem:bsnew}, using Lemma~\ref{lem:brown22}, is the following: if $\alpha$ is a class in $\op{H}^n(X,\mathbf{I}^{n+1})$ such that $\eta^2\alpha=0$ in $\op{H}^n(X,\mathbf{I}^{n-1})$, do we already have $\eta\alpha=0$ in $\op{H}^n(X,\mathbf{I}^n)$? In the topological situation, $4$-torsion can never appear in the top degree cohomology of a manifold; and the above question (via its torsion reformulation) asks for an algebraic analogue of this fact.


We end our discussion by the following special case of Proposition~\ref{greatapplication}. Note that the conditions appearing below are in particular satisfied for Mohan Kumar's varieties in \cite{mohan:kumar} for odd primes.

\begin{proposition}
\label{prop:greatapplication2}
Let $F$ be a perfect field of characteristic $\neq 2$, and let $n$ be an odd integer. Let $X\subset \mathbb{P}^n$ be a smooth affine open subscheme with $\op{CH}^n(X)$ $2$-torsionfree, e.g. the complement of a union of smooth hypersurfaces for which each closed point has odd degree. Let $\mathscr{E}\to X$ be a rank $n$ vector bundle over $X$. If $\op{c}_n(\mathscr{E})=0$, then $\mathscr{E}$ splits off a free rank one summand. 
\end{proposition}

\begin{proof}
By \cite[Proposition 11.6]{fasel:ij}, we know that the Bockstein map 
\[
\beta\colon \op{Ch}^{n-1}(\mathbb{P}^n)\to\op{H}^n(\mathbb{P}^n,\mathbf{I}^n)
\]
is trivial. Since $\beta$ is compatible with the restriction maps
$j^*$ in the localization sequences, we have a commutative square
\[
\xymatrix{
\op{Ch}^{n-1}(\mathbb{P}^n)\ar[r]^{j^*} \ar[d]_\beta & \op{Ch}^{n-1}(X)\ar[d]^\beta \\
\op{H}^n(\mathbb{P}^n,\mathbf{I}^n)\ar[r]^{j^*} & \op{H}^n(X,\mathbf{I}^n)
}
\]
Now $j^*\colon \op{Ch}^{n-1}(\mathbb{P}^n)\to\op{Ch}^{n-1}(X)$ is surjective, hence the right-hand $\beta$ must also be the zero map. In particular, the top Stiefel--Whitney classes of any bundle (satisfying the conditions in the statement) vanishes automatically. 
\end{proof}

This also complements the example discussed in \cite{bhatwadekar:fasel:sane}. Note that our methods do not apply directly to their example, which has 2-torsion in the Chow-groups, but the above argument at least implies vanishing of $\op{e}_n\in\op{H}^n(U,\mathbf{I}^n)$ for the variety $U$ in loc.cit. 



\end{document}